\documentclass[format=acmsmall, review=false, screen=true, nonacm=true]{acmart}
\usepackage[utf8]{inputenc}
\usepackage{amsfonts}
\usepackage{amsmath}
\usepackage{amsthm}
\usepackage{mathtools}
\usepackage{graphicx}
\usepackage{stackengine}
\usepackage{tabularx}
\usepackage{booktabs}
\usepackage{xspace}
\usepackage{multirow}
\usepackage{cases}
\usepackage{subcaption}
\usepackage[inline]{enumitem}
\usepackage{xcolor}
\usepackage{microtype}

\usepackage{tikz}
\usetikzlibrary{calc,positioning,decorations.pathreplacing,fit}

\def\rdots{\rotatebox[origin=l]{29}{$\scriptscriptstyle\ldots\mathstrut$}}

\def\ve#1{\mathchoice{\mbox{\boldmath$\displaystyle\bf#1$}}
{\mbox{\boldmath$\textstyle\bf#1$}}
{\mbox{\boldmath$\scriptstyle\bf#1$}}
{\mbox{\boldmath$\scriptscriptstyle\bf#1$}}}
\newcommand\vea{{\ve a}}
\newcommand\vealpha{{\ve \alpha}}
\newcommand\veb{{\ve b}}
\newcommand\vecc{{\ve c}}

\newcommand\vece{{\ve e}}

\newcommand\veg{{\ve g}}
\newcommand\veG{{\ve G}}
\newcommand\veh{{\ve h}}
\newcommand\vek{{\ve k}}
\newcommand\vel{{\ve l}}

\newcommand\vep{{\ve p}}

\newcommand\ver{{\ve r}}
\newcommand\ves{{\ve s}}
\newcommand\vet{{\ve t}}
\newcommand\veu{{\ve u}}
\newcommand\vev{{\ve v}}
\newcommand\vew{{\ve w}}
\newcommand\vex{{\ve x}}
\newcommand\vey{{\ve y}}
\newcommand\vez{{\ve z}}
\newcommand{\norm}[1]{\left\| #1 \right\|}

\newcommand\vegamma{{\ve \gamma}}

\newcommand\vebeta{{\ve \beta}}

\newcommand\vezero{{\ve 0}}
\newcommand\veone{{\ve 1}}

\newcommand{\domain}{\mathfrak{D}}

\newcommand\II{\mathcal{I}}

\newcommand{\df}{\mathrel{\mathop:}=}
\newcommand{\rd}[1]{\lfloor#1\rceil}
\newcommand{\eps}{\epsilon}

\DeclarePairedDelimiter\ceil{\lceil}{\rceil}
\DeclarePairedDelimiter\floor{\lfloor}{\rfloor}
\newcommand{\hy}{\hbox{-}\nobreak\hskip0pt}

\newcommand{\chr}[1]{\textcolor{red}{#1} \marginpar{\footnotesize \flushleft Christoph}}
\newcommand{\martin}[1]{\textcolor{blue}{#1} \marginpar{\footnotesize \flushleft Martin}}

\makeatletter
\newtheorem*{rep@theorem}{\rep@title}
\newcommand{\newreptheorem}[2]{%
\newenvironment{rep#1}[1]{%
 \def\rep@title{#2 \ref{##1}}%
 \begin{rep@theorem}}%
 {\end{rep@theorem}}}
\makeatother

\newcommand{\sv}[1]{#1}
\newcommand{\lv}[1]{}

\newtheorem{theorem}{Theorem}

\newreptheorem{theorem}{Theorem}
\newreptheorem{corollary}{Corollary}
\newreptheorem{lemma}{Lemma}

\theoremstyle{remark}
\newtheorem*{remark}{Remark}
\newtheorem*{claim*}{Claim}
\newtheorem*{example*}{Example}

\newcommand{\ignore}[1]{}

\def\N{\mathbb{N}}
\def\R{\mathbb{R}}
\def\Z{\mathbb{Z}}

\def \G {\mathcal{G}}
\def \T {\mathcal{T}}
\def \l {\langle}
\def \r {\rangle}

\DeclareMathOperator{\best}{\ngspc\mathrm{-best}}
\DeclareMathOperator{\poly}{\mathrm{poly}}
\DeclareMathOperator{\suppo}{\mathrm{supp}}

\def \ngspc {\hspace{-1.5pt}}

\def \C {\mathcal{C}}

\def \Oh {\mathcal{O}}
\def \oh {o}
\def \AA {\mathcal{A}}
\def \AAap {\AA^{\mathrm{ap}}}
\newcommand{\oracle}[1]{#1^{\mathrm{oracle}}}
\def \AAA {\oracle{\AA}}

\def \CC {\mathcal{C}}
\def \DD {\mathcal{D}}
\def \FF {\mathcal{F}}
\def \PP {\mathcal{P}}
\def \RR {\mathcal{R}}
\def \RRR {\oracle{\RR}}
\def \FFF {\oracle{\FF}}

\def \l {\langle}
\def \r {\rangle}

\def \la {\langle}
\def \ra {\rangle}

\DeclareMathOperator{\cone}{cone}
\DeclareMathOperator{\spann}{span}
\DeclareMathOperator{\convol}{\mathrm{convol}}

\def \poly {{\rm poly}}
\def \sign {{\rm sign}}
\def \undffd {\mathtt{undef}}
\DeclareMathOperator{\tw}{\mathrm{tw}}

\DeclareMathOperator{\td}{\mathrm{td}}
\DeclareMathOperator{\ttd}{\mathrm{th}}

\DeclareMathOperator{\cl}{\mathrm{cl}}
\DeclareMathOperator{\height}{\mathrm{height}}

\DeclareMathOperator{\Sol}{\mathrm{Sol}}
\DeclareMathOperator{\fin}{\mathrm{fin}}
\def \transpose {{\intercal}}

\newcommand{\NP}{$\mathsf{NP}$\xspace}
\newcommand{\NPh}{$\mathsf{NP}$\hy\textsf{hard}\xspace}
\newcommand{\NPc}{$\mathsf{NP}$\hy\textsf{complete}\xspace}
\newcommand{\FPT}{$\mathsf{FPT}$\xspace}

\newcommand{\Wh}[1]{$\mathsf{W[#1]}$\hy\textsf{hard}\xspace}

\newcommand{\Trule}{\rule{0pt}{3ex}}
\newcommand{\Brule}{\rule[-2ex]{0pt}{0pt}}

\newcommand{\prob}[3]{
	\begin{center}
		\begin{tabularx}{\textwidth}{|l X|}
			\hline
			\multicolumn{2}{|c|}{#1\Trule} \\
			\textbf{Input:\ }&{#2}\\
			\textbf{Task:\ }&{#3\Brule}\\
			\hline
		\end{tabularx}
	\end{center}
}

\title{An Algorithmic Theory of Integer Programming}
\titlenote{Preliminary versions of some of the results appearing here have appeared in ICALP 2018~\cite{EisenbrandHK:2018,KouteckyLO:2018}}

\author{Friedrich Eisenbrand}
\affiliation{%
	\institution{EPFL}
	\city{Lausanne}
	\country{Switzerland}}
\email{friedrich.eisenbrand@epfl.ch}
\author{Christoph Hunkenschröder}
\affiliation{%
	\institution{TU}
	\city{Berlin}
	\country{Germany}}
\email{hunkenschroeder@tu-berlin.de}
\author{Kim-Manuel Klein}
\affiliation{%
	\institution{Universität Kiel}
	\city{Kiel}
	\country{Germany}}
\email{kmk@informatik.uni-kiel.de}
\author{Martin Kouteck{\'y}}
\affiliation{%
	\institution{Charles University}
	\city{Prague}
	\country{Czech Republic}}
\email{koutecky@iuuk.mff.cuni.cz}
\author{Asaf Levin}
\email{levinas@ie.technion.ac.il}
\author{Shmuel Onn}
\affiliation{%
	\institution{Technion -- Israel Institute of Technology}
	\city{Haifa}
	\country{Israel}}
\email{onn@ie.technion.ac.il}
\renewcommand{\shortauthors}{F. Eisenbrand, Ch. Hunkenschröder, K.-M. Klein, M. Koutecký, A. Levin, and S. Onn}


\begin{document}
	\begin{abstract}
		We study the general integer programming problem where the number of variables $n$ is a
		variable part of the input. We consider two natural parameters of the constraint matrix $A$:
		its {\em numeric measure} $a$
		and its {\em sparsity measure} $d$.
		We show that integer programming can be solved in time $g(a,d)\poly(n,L)$, where $g$ is some computable function of the parameters $a$
		and $d$, and $L$ is the binary encoding length of the input.
		In particular, integer programming is fixed-parameter tractable parameterized by $a$ and $d$, and is solvable in polynomial time for every fixed $a$ and $d$.
		Our results also extend to nonlinear separable convex
		objective functions.
		Moreover, for linear objectives, we derive a strongly-polynomial algorithm,
		that is, with running time $g(a,d)\poly(n)$, independent of the rest of the input data.
		
		We obtain these results by developing an algorithmic framework based on the idea of iterative
		augmentation: starting from an initial feasible solution, we show how to quickly find augmenting
		steps which rapidly converge to an optimum. A central notion in this framework is the Graver basis
		of the matrix $A$, which constitutes a set of fundamental augmenting steps. The iterative
		augmentation idea is then enhanced via the use of other techniques such as new and improved bounds on the Graver basis, rapid solution of integer programs with bounded variables, proximity theorems and
		a new proximity-scaling algorithm, the notion of a reduced objective function, and others.
		
		As a consequence of our work, we advance the state of the art of solving block-structured integer programs. In particular, we develop
		near-linear time algorithms for $n$-fold, tree-fold, and $2$-stage stochastic integer programs.
		We also discuss some of the many applications of these classes.
	\end{abstract}
	\maketitle
	
\tableofcontents	
	
	\section{Introduction}
	Our focus is on the integer (linear) programming problem in standard form
	\begin{gather}
	\min\left\{f(\vex) \mid A\vex=\veb,\, \vel\leq\vex\leq\veu,\, \vex\in\Z^{n}\right\} \tag{IP}, \text{ and} \label{IP}
	\\
	\min\left\{\vew \vex \mid A\vex=\veb,\, \vel\leq\vex\leq\veu,\, \vex\in\Z^{n}\right\} \tag{ILP}, \label{ILP}
	\end{gather}
	with $A$ an integer $m\times n$ matrix, $f: \R^n \to \R$ a separable convex function, $\veb\in\Z^m$, and $\vel,\veu\in(\Z\cup\{\pm\infty\})^n$.
	\eqref{IP} is well-known to be strongly \NPh already in the special case \eqref{ILP} when $f(\vex) = \vew \vex$ is a linear objective function for some vector $\vew \in \Z^n$.
	In spite of that, in this paper we identify broad natural and useful conditions under which~\eqref{IP} can be solved in polynomial time, even when the number of variables $n$ is a variable part of the input.
	
	Specifically, we consider two natural parameters of the constraint matrix $A$: its numeric measure $a$ and its sparsity measure $d$ defined as follows.
	The numeric measure depends only on the values of the entries of the matrix $A$ and is essentially the largest absolute value of any coefficient: $a \df \max\{2,\|A\|_\infty\}$.
	On the other hand, the sparsity measure $d$ depends only on the structure of non-zeroes of $A$ and we use the notion of primal and dual treedepth to capture it.
	These are defined as follows.
	Let $G_P(A)$ denote the \emph{primal graph of $A$}, which has $\{1, \dots, n\}$ as its vertex set, and an edge between vertices $i$ and $j$ exists if $A$ contains a row which is nonzero in coordinates $i$ and $j$.
	The \emph{dual graph of $A$} is $G_D(A) \df G_P(A^{\intercal})$.
	The \emph{treedepth} of a graph denoted $\td(G)$ is the smallest height of a rooted forest $F$ such that each edge of $G$ is between vertices which are in a descendant-ancestor relationship in $F$.
	The \emph{primal treedepth of $A$} is $\td_P(A) \df \td(G_P(A))$, and analogously the \emph{dual treedepth of $A$} is $\td_D(A) \df \td(G_D(A))$.
	Then, the sparsity measure $d$ is defined as $d \df \min \{\td_P(A), \td_D(A)\}$.
	Denote by $\la A, f, \veb, \vel, \veu \ra$ the binary encoding length of an~\eqref{IP} instance.\footnote{We define the encoding length of $f$ to be the length of $f_{\max}$, which is the difference between the maximum and minimum values of $f$ on the domain.}
	The function $f$ is given by an oracle.
	Our first main result then reads as follows:
	\begin{theorem} \label{thm:main}
		There exists a computable function $g$ such that problem~\eqref{IP} can be solved in time \[g(a,d) \poly(n,L), \qquad \text{where } d \df \min\{\td_P(A), \td_D(A)\} \text{ and } L \df \la A, f, \veb, \vel, \veu \ra \enspace .\]
	\end{theorem}
	Note that for our algorithm to be fast it suffices if at least one of $\td_P(A)$ and $\td_D(A)$ is small.
	We also develop an algorithmic framework among others suitable for obtaining strongly polynomial algorithms, and as a consequence of Theorem~\ref{thm:main} we show a strongly polynomial algorithm for~\eqref{ILP}:
	\begin{corollary} \label{cor:ilp_stronglypoly}
		There exists a computable function $g$ such that problem~\eqref{ILP} can be solved with an algorithm whose number of arithmetic operations is bounded by \[g(a,d) \poly(n), \qquad \text{where }d \df \min\{\td_P(A), \td_D(A)\} \enspace .\]
	\end{corollary}
	Note that already for $a=1$ or $d=1$~\eqref{ILP} is \NPh (Proposition~\ref{prop:ilphardness}).
	Moreover, arguably the two most important tractable classes of~\eqref{IP} are formed by instances whose constraint matrix is either totally unimodular or has small number $n$ of columns, yet our results are incomparable with either: the class of totally unimodular matrices might have large $d$, but has $a=1$, and the matrices considered here have variable $n$.
	
	We also show that Theorem~\ref{thm:main} cannot be improved in multiple senses.
	First, treedepth cannot be replaced with the more permissive notion of treewidth, since~\eqref{IP} is \NPh already when $\min\{\tw_P(A), \tw_D(A)\} = 2$ and $a=2$ (Corollaries~\ref{cor:twdhardness} and~\ref{cor:twphardnness}).
	Second, the parameterization cannot be relaxed by removing the parameter $a$: \eqref{IP} is para-\NPh parameterized by $\td_P(A)$~\cite[Thm 21]{DvorakEGKO21} and strongly \Wh{1} parameterized by $\td_D(A)$~\cite[Thm 5]{KnopKM:2017} alone; the fact that a problem is \Wh{1} is strong evidence that it is not fixed-parameter tractable, and a problem is para-\NPh if it is \NPh already for some constant value of the parameter.
	Third, the requirement that $f$ is separable convex cannot be relaxed since~\eqref{IP} with a non-separable convex or a separable concave function are \NPh even for small values of our parameters (Proposition~\ref{prop:nphconvexconcave}).
	We also provide several concrete lower bounds, for example showing that the function $g$ must be at least double-exponential unless the exponential time hypothesis (ETH) fails (Theorems~\ref{thm:tdp_lowerbound} and~\ref{thm:tdd_lowerbound}).
	
	\medskip
	
	Beyond Theorem~\ref{thm:main}, we develop several new techniques and use them to construct improved algorithms for~\eqref{IP} when $a$ and $\td_P(A)$ or $\td_D(A)$ is small.
	Let us highlight the main technical contributions.
	We prove a new proximity theorem and use it to obtain a scaling algorithm:
	\begin{repcorollary}{cor:scaling_algo}[Proximity-scaling algorithm, informal]
		Solving~\eqref{IP} can be reduced to solving $\log \|\veu-\vel\|_\infty$ auxiliary instances with polynomial lower and upper bounds.
	\end{repcorollary}
	Next, consider the following question: given a linear objective function $\vew \vex$, how to obtain an equivalent function $\bar{\vew} \vex$ with small $\|\bar{\vew}\|_\infty$?
	We say that $\bar{\vew}$ is \emph{equivalent} if $\vew \vex > \vew \vey \Leftrightarrow \bar{\vew} \vex > \bar{\vew}\vey$ for all feasible $\vex, \vey$.
	Frank and Tardos~\cite{FT} provide an algorithm to compute such a vector.
	We
	\begin{itemize}
		\item extend their result also to separable convex objective functions (Corollary~\ref{cor:sepconvex_red}),
		\item give stronger non-constructive bounds on equivalent functions (Theorem~\ref{thm:lin_red} and Corollary~\ref{cor:sepconvex_red}),
		\item and show that these bounds are asymptotically the best possible (Theorems~\ref{thm:lin_red_lowerbound} and~\ref{thm:sepconvex_lb}).
	\end{itemize}
	Such bounds allow us to construct fast strongly-polynomial algorithms for~\eqref{IP}.
	All of the results mentioned above, combined with a wealth of useful technical statements, allow us to construct improved algorithms for some previously studied classes of IP:
	\begin{theorem}[{Informal, see Corollaries~\ref{cor:nfold},~\ref{cor:2stage},~\ref{cor:tf}}]
		$N$-fold, tree-fold, and 2-stage stochastic IP is solvable in near-linear fixed-parameter tractable time
		$$g(k_1, \dots, k_\tau) \cdot n \log^{\Oh(1)} n \cdot L,$$ where $k_1, \dots, k_\tau$ are the relevant instance parameters and $L$ is a measure of the input length.
	\end{theorem}
	Multi-stage stochastic IP can be solved in near-linear time in a restricted regime, see Corollary~\ref{cor:mss}.
	
	Finally, we show double-exponential lower bounds for~\eqref{IP} with parameters $\td_P(A)$ or $\td_D(A)$ based on the Exponential Time Hypothesis (ETH) (Theorems~\ref{thm:tdd_lowerbound} and~\ref{thm:tdp_lowerbound}).
	No such bound was known for $\td_P(A)$, and our bound for $\td_D(A)$ improves and refines the recent lower bound of Knop et al.~\cite{KnopPW:2018}.
	
	\medskip
	
	We note that our results also hold for~\eqref{IP} whose constraints are given in the inequality form $A\vex \leq \veb$: introducing slack variables leads to~\eqref{IP} in standard form with a constraint matrix $A_I \df (A~I)$, with $\min\{\td_P(A_I), \td_D(A_I)\} \leq \min\{\td_P(A)+1, \td_D(A)\}$ (Lemma~\ref{lem:feas_td}).
	
	\subsection{Related Work}
	\paragraph*{Parametric, arithmetic, numeric input.}
	To clearly state our results and compare them to previous work we introduce the following terminology.
	The input to a problem will be partitioned into three parts
	$(\alpha, \beta, \gamma)$, where $\alpha$ is the
	{\em parametric input}, $\beta$ is the
	{\em arithmetic input}, and $\gamma$ is the {\em numeric input}.
	The \emph{time} an algorithm takes to solve a problem is the number of arithmetic operations and oracle queries, and all numbers involved in the computation are required to have length polynomial in $(\beta,\gamma)$.
	A \emph{polynomial algorithm} for the problem is one that solves it in time $\poly(\beta, \gamma)$, while a \emph{strongly-polynomial algorithm} solves it in time $\poly(\beta)$, i.e., independent of the numeric input.
	Similarly, a \emph{fixed-parameter tractable} (\FPT) \emph{algorithm} solves the problem in time $g(\alpha) \poly(\beta,\gamma)$, while a \emph{strongly \FPT algorithm} solves it in time $g(\alpha)\poly(\beta)$, where $g$ is a computable function.
	If such an algorithm exists, we say that the problem is \emph{(strongly) fixed-parameter tractable} (\FPT) \emph{parameterized by $\alpha$}.
	Having multiple parameters $\alpha_1, \dots, \alpha_k$ simultaneously is understood as taking the \emph{aggregate parameter} $\alpha = \alpha_1 + \cdots + \alpha_k$.
	When we want to highlight the fact that an oracle is involved (e.g., when the oracle calls are expected to take a substantial portion of the time), we say that the algorithm works in certain (polynomial, strongly polynomial, \FPT, etc.) \emph{oracle time}.
	Each part of the input may have several entities, which may be presented in unary or binary.
	For the parametric input the distinction between unary and binary is irrelevant, but it defines the function $g$.
	
	\medskip
	
	We distinguish work prior to 2018 when the conference papers~\cite{EisenbrandHK:2018,KouteckyLO:2018} forming the basis of our work have been published, and subsequent results.
	We summarize how our work compares with the best known time complexities before and after 2018 in Table~\ref{tab:results}, and signify the before/after distinction with bold horizontal lines.
	For detailed comparison with the previously best known results see Section~\ref{sec:apps}.
	\begin{table}[tb]
		\centering
		\newcommand{\minitab}[2][l]{\begin{tabular}{#1}#2\end{tabular}\hspace{-5pt}}
		\resizebox{\textwidth}{!}{%
			\begin{tabular}{llr}
				\toprule
				Type of instance             & Previous best run time & Our result\\
				\toprule
				\multirow{4}{*}{$n$-fold}          & $a^{\Oh(rst + st^2)} n^3 L$~\cite{HOR} & \multirow{4}{*}{\minitab[r]{$a^{\Oh(r^2s + rs^2)} (nt)^2 \log^3(nt) + \RR(A,L)$~Cor~\ref{cor:nfold} \\ $a^{\Oh(r^2s + rs^2)} (nt) \log(nt)L$~Cor~\ref{cor:nfold}}}  \\
				& $n^{f_1(a,r,s,t)}$~\cite{DeLoeraHL:2015} &  \\
				
				& $t^{\Oh(r)} (ar)^{r^2} n^3 L$ if $A_2 = (1~1~\cdots~1)$~\cite{KnopKM:2017b} &  \\
				\cmidrule[1.5\heavyrulewidth]{2-2}
				& $(ars)^{\Oh(r^2s + s^2)} (nt) \log^6 (nt) L$~\cite{JansenLR:2018} &  \\
				\midrule
				\multirow{2}{*}{2-stage stochastic}          & $f_2(a,r,s) n^3 L$~\cite{DHK} & \multirow{2}{*}{\minitab[r]{$f_3(a,r,s)n^2 \log^5 n + \RR(A,L)$~Cor~\ref{cor:2stage} \\ $f_3(a,r,s)n \log^3 n L$~Cor~\ref{cor:2stage}}}  \\
				\cmidrule[1.5\heavyrulewidth]{2-2}
				& $f_3(a,r,s) n^2 \log n L$~\cite{Klein} &  \\
				\midrule
				\multirow{2}{*}{Tree-fold}  & \multirow{2}{*}{$f_{4}(a,r_1,\dots,r_{\tau},t) n^3 L$~\cite{MC}} & $f_{5}(a,r_1,\dots,r_{\tau}) (nt)^2 \log^3(nt) + \RR(A,L)$~Cor~\ref{cor:tf}\\
				& & $f_{5}(a,r_1,\dots,r_{\tau}) (nt) \log (nt)L$~Cor~\ref{cor:tf} \\
				\midrule
				\multirow{3}{*}{Multi-stage stochastic} & $f_{6}(a, n_1, \dots, n_{\tau}, l) n^3 L$~\cite{AH} & \multirow{3}{*}{\minitab[r]{
						$f_{7}(a, n_1, \dots, n_{\tau}) n^3 \log^2 n + \RR(A,L)$~Cor~\ref{cor:mss} \\
						$f_{7}(a, n_1, \dots, n_{\tau}) n^2 L$~Cor~\ref{cor:mss}\\
						$f_{7}(a, n_1, \dots, n_{\tau}) n^{1+o(1)}(\log f_{\max})^{\tau-1}$~Cor~\ref{cor:mss}}}\\
				\cmidrule[1.5\heavyrulewidth]{2-2}
				& \multirow{2}{*}{$f_{7}(a, n_1, \dots, n_{\tau}, l) n^2 \log n L$~\cite{Klein}} & \\
				& & \\
				\midrule
				Bounded $\td_P(A)$ & $f_{8}(a, \|\veb\|_\infty, \td_P(A)) n L$~\cite{GO} & strongly \FPT~Cor~\ref{cor:primal} \\
				\midrule
				{Bounded $\td_D(A)$} &{Open whether \FPT} & near linear \FPT ~/ strongly \FPT Cor~\ref{cor:dual} \\
				\bottomrule
		\end{tabular}}
		\caption{Run time improvements for~\eqref{ILP} implied by this paper.
			All of our results also extend to separable convex~\eqref{IP}, potentially with some overhead.
			This table contains simplifications necessary for brevity; for a detailed comparison see Section~\ref{sec:apps}.
			We denote by $L$ a certain measure of the arithmetic input $\veb, \vel, \veu, f_{\max}$, where $f_{\max} \df \max_{\vex, \vex': \vel \leq \vex, \vex' \leq \veu} |f(\vex) - f(\vex')|$ (typically $L = (\log\|\veu-\vel, \veb\|_\infty) \cdot (\log f_{\max} + \log \|\veb\|_1)$).
			We consider $\l A \r$ to be part of the arithmetic input.
			We denote by $a = \max \{2, \|A\|_\infty\}$, by $r,s,t, n_1, \dots, n_\tau, r_1, \dots, r_\tau$ the relevant block dimensions and by $\RR(A,L)$ the time complexity of an algorithm solving the fractional relaxation of~\eqref{IP}.
			In the case of~\eqref{ILP}, there exists a strongly polynomial LP algorithm~\cite{Tar}, i.e., with time complexity independent of $L$.
			If some related work appeared after 2018, it lies below a thick horizontal line.
			We do not include the conference papers~\cite{EisenbrandHK:2018,KouteckyLO:2018} on which this paper is based in the comparison.
			The parameter dependence functions satisfy $f_2 > f_3$, $f_4 > f_5$, and $f_6 > f_7$.}
		\label{tab:results}
	\end{table}
	\paragraph*{Work prior to 2018}
	Concerning strongly-polynomial algorithms, so far we are aware only of algorithms for totally unimodular ILP~\cite{HS}, bimodular ILP~\cite{ArtmannWZ:2017}, so-called binet ILP~\cite{AppaKP:2007}, and $n$-fold IP with constant block dimensions~\cite{DeLoeraHL:2015}.
	All remaining results, such as Lenstra's famous algorithm or the fixed-parameter tractable algorithm for $n$-fold IP which has recently led to several breakthroughs~\cite{MC,JansenKMR:2018,KnopKM:2017b,KnopKM:2017}, are not strongly polynomial.
	
	Let us turn to \FPT algorithms for~\eqref{IP}.
	In the '80s it was shown by Lenstra and Kannan~\cite{Kannan:1987,Lenstra:1983} that~\eqref{ILP} can be solved in time~$n^{\Oh(n)} L$, where $L$ is the length of the binary encoding of the input, thus \FPT parameterized by $n$.
	Other large classes which are known to be \FPT are $n$-fold~\cite{HOR}, tree-fold~\cite{MC}, $2$-stage and multi-stage stochastic~\cite{AH}, as well as algorithms for ILPs with bounded treewidth~\cite{GOR}, treedepth~\cite{GO} and fracture number~\cite{DvorakEGKO21} of graphs related to the matrix $A$.
	The class of $4$-block $n$-fold IPs has an algorithm with time complexity $n^{g(k)}$~\cite{HKW} where $k$ is the maximum of the largest absolute value of a coefficient and the largest dimension, and is not known to be \FPT.
	
	More precisely, it follows from Freuder's algorithm~\cite{Freu} and was reproven by Jansen and Kratsch~\cite{JK} that~\eqref{IP} is \FPT parameterized by the primal treewidth $\tw_P(A)$ and the largest domain $\|\veu-\vel\|_\infty$.
	Regarding the dual graph $G_D(A)$, the parameters $\td_D(A)$ and $\tw_D(A)$ were only recently considered by Ganian et al.~\cite{GOR}.
	They show that even deciding feasibility of~\eqref{IP} is \NPh on instances with $\tw_I(A) = 3$ ($\tw_I(A)$ denotes the treewidth of the \emph{incidence graph}; $\tw_I(A) \leq \tw_D(A) + 1$ always holds, see Lemma~\ref{lem:inc_prim_tw}) and $\|A\|_\infty = 2$~\cite[Theorem 12]{GOR}.
	Furthermore, they show that~\eqref{IP} is \FPT parameterized by $\tw_I(A)$ and parameter $\Gamma$, which is an upper bound on any prefix sum of $A \vex$ for any feasible solution $\vex$.
	
	Dvořák et al~\cite{DvorakEGKO21} introduce the  parameter fracture number.
	Having a bounded \emph{variable fracture number} $\mathfrak{p}^V(A)$ implies that deleting a few columns of $A$ breaks it into independent blocks of small size, similarly for \emph{constraint fracture number} $\mathfrak{p}^C(A)$ and deleting a few rows.
	Because bounded $\mathfrak{p}^V(A)$ implies bounded $\td_P(A)$ and bounded $\mathfrak{p}^C(A)$ implies bounded $\td_D(A)$, our results generalize theirs.
	The remaining case of \emph{mixed fracture number} $\mathfrak{p}(A)$, where deleting both rows and columns is allowed, reduces to the $4$-block $n$-fold ILP problem, which is not known to be either \FPT or \Wh{1}.
	
	\paragraph*{Subsequent work}
	Jansen, Lassota, and Rohwedder~\cite{JansenLR:2018} showed a near-linear time algorithm for $n$-fold IP, which has a slightly better parameter dependence but slightly worse dependence on $n$ when compared with our algorithms, and only applies to the case of~\eqref{ILP} while our algorithm also solves~\eqref{IP} and generalizes to tree-fold IP.
	Knop, Pilipczuk, and Wrochna~\cite{KnopPW:2018} gave lower bounds for~\eqref{ILP} with few rows and also~\eqref{ILP} parameterized by $\td_D(A)$; our lower bound of Theorem~\ref{thm:tdd_lowerbound} generalizes their latter bound.
	Klein~\cite{Klein} proved a lemma (Proposition~\ref{prop:klein}) which allowed him to give a double-exponential (in terms of the parameters) algorithm for $2$-stage stochastic IP, when prior work had no concrete bounds on the parameter dependence.
	His results also generalize to multi-stage stochastic IP.
	On the side of hardness, Eiben et al. have shown that~\eqref{ILP} is \NPh already when the more permissive \emph{incidence treedepth} $\td_I(A)$ is $5$ and $\|A\|_\infty=2$~\cite{EibenGKOPW:2019}.
	
	\subsubsection*{Organization.}
	The paper contains four main sections.
	In Section~\ref{sec:fastip}, we prove Theorem~\ref{thm:main}.
	Our exposition is unified, streamlined, largely self-contained, and already attains complexities comparable with prior work.
	Its structure lays the groundwork for later improvements.
	Section~\ref{sec:framework} develops a framework for improving the complexity of algorithms for~\eqref{IP}.
	Its main technical contributions have been outlined above and culminate in near-linear algorithms for~\eqref{IP} with small coefficients and small $\td_P(A)$ or $\td_D(A)$.
	Section~\ref{sec:apps} reaps the fruits of the framework established previously by discussing the implications of our results.
	We review the motivation and applications of the classes of $n$-fold, tree-fold, 2-stage and multi-stage stochastic IPs and derive the time complexities for these classes implied by our algorithms.
	Finally, Section~\ref{sec:hardness} contains hardness results, in particular the double-exponential lower bounds for~\eqref{IP} parameterized by $\td_P(A)$ or $\td_D(A)$.

\section{A Fast Algorithm for IP with Small Treedepth and Coefficients}
\label{sec:fastip}
We write vectors in boldface (e.g., $\vex, \vey$) and their entries in normal font (e.g., the $i$-th entry of~$\vex$ is~$x_i$).
For positive integers $m \leq n$ we set $[m,n] \df \{m,\ldots, n\}$ and $[n] \df [1,n]$, and we extend this notation for vectors: for $\vel, \veu \in \Z^n$ with $\vel \leq \veu$, $[\vel, \veu] \df \{\vex \in \Z^n \mid \vel \leq \vex \leq \veu\}$.
If~$A$ is a matrix, $A_{i,j}$ denotes the $j$-th coordinate of the $i$-th row, $A_{i, \bullet}$ denotes the $i$-th row and $A_{\bullet, j}$ denotes the $j$-th column.
We use $\log \df \log_2$.
For an integer $a \in \Z$, $\l a \r \df 1 + \ceil{\log (|a| + 1)}$ denotes the binary encoding length of $a$; we extend this notation to vectors, matrices and tuples of these objects.
For example, $\l A, \veb \r = \l A \r + \l \veb \r$, and $\l A \r = \sum_{i,j} \l A_{i,j} \r$.
For a graph~$G$ we denote by $V(G)$ its set of vertices.
For a function $f: \Z^n \to \Z$ and two vectors $\vel, \veu \in \Z^n$, we define $f_{\max}^{[\vel, \veu]} \df \max_{\vex, \vex' \in [\vel, \veu]} |f(\vex) - f(\vex')|$; if $[\vel, \veu]$ is clear from the context we omit it and write just $f_{\max}$.
We assume that $f: \R^n \to \R$ is a separable convex function, i.e., it can be written as $f(\vex) = \sum_{i=1}^n f_i(x_i)$ where $f_i$ is a convex function of one variable, for each $i \in [n]$.
Moreover, we require that for each $\vex \in \Z^n$, $f(\vex) \in \Z$.
We assume $f$ is given by a comparison oracle.
We use $\omega$ to denote the smallest number such that matrix multiplication of $n \times n$ matrices can be performed in time $\Oh(n^\omega)$.
We say that a system of equations $A\vex = \veb$ is \emph{pure} if the rows of $A$ are linearly independent.
The next statement follows easily by Gaussian elimination, hence we assume $m \leq n$ throughout the paper.
\begin{proposition}[{Purification~\cite[Theorem 1.4.8]{GLS}}] \label{prop:purification}
Given $A \in \Z^{m \times n}$ and $\veb \in \Z^m$ one can in time $\Oh(\min\{n,m\}nm)$ either declare $A\vex=\veb$ infeasible, or output a pure equivalent subsystem $A'\vex = \veb'$.
\end{proposition}
In Section~\ref{sec:framework} we will replace Proposition~\ref{prop:purification} with Proposition~\ref{prop:tw_pure} which gives a better complexity bound for matrices of small treewidth.
The goal of this section is to prove the following theorem:
\begin{reptheorem}{thm:main}
	There is a computable function $g$ such that~\eqref{IP} can be solved in time \[g(\|A\|_\infty, \min\{\td_P(A), \td_D(A)\}) \, \cdot \, n^2 \log \|\veu-\vel, \veb\|_\infty \log \left(2 f_{\max}\right) + \Oh(n^{\omega} \la A \ra) \enspace \]
\end{reptheorem}
In Sections~\ref{sec:iterative}--\ref{sec:norms} we shall develop the necessary ingredients to prove this theorem.
Then, we will conclude in Section~\ref{sec:mainrproof} by providing its proof which puts these ingredients together.

\subsection{Introduction to Iterative Augmentation} \label{sec:iterative}
Let us introduce Graver bases and discuss how they are used for optimization.
We define a partial order $\sqsubseteq$ on $\R^n$ as follows: for $\vex,\vey\in\R^n$ we write $\vex\sqsubseteq \vey$ and say that $\vex$ is
{\em conformal} to $\vey$ if, for each $i \in [n]$, $x_iy_i\geq 0$ (that is, $\vex$ and $\vey$ lie in the same orthant) and $|x_i|\leq |y_i|$.
For a matrix $A \in \Z^{m \times n}$ we write $\ker_{\Z}(A) = \{\vex \in \Z^n \mid A\vex = \vezero\}$.
It is well known that every subset of $\Z^n$ has finitely many
$\sqsubseteq$-minimal elements~\cite{Gordan}.

\begin{definition}[Graver basis~\cite{Graver:1975}]\label{def:graver}
The {\em Graver basis} of an integer $m\times n$ matrix $A$ is the finite set $\G(A)\subset\Z^n$ of $\sqsubseteq$-minimal elements in $\ker_{\Z}(A) \setminus \{\vezero\}$.
\end{definition}

One important property of $\G(A)$ is as follows:

\begin{proposition}[{Positive Sum Property~\cite[Lemma 3.4]{Onn}}] \label{prop:possum}
Let $A \in \Z^{m \times n}$.
For any $\vex \in \ker_{\Z}(A)$, there exists an $n' \leq 2n-2$ and a decomposition $\vex = \sum_{j=1}^{n'} \lambda_j \veg_j$ with $\lambda_j \in \N$ and $\veg_j \in \G(A)$ for each $j \in [n']$, and with $\veg_j \sqsubseteq \vex$, i.e., all $\veg_j$ belonging to the same orthant as $\vex$.
\end{proposition}

We say that $\vex \in \Z^n$ is \emph{feasible} for~\eqref{IP} if $A\vex = \veb$ and $\vel \leq \vex \leq \veu$.
Let $\vex$ be a feasible solution for~\eqref{IP}.
We call $\veg$ a \emph{feasible step} if $\vex + \veg$ is feasible for~\eqref{IP}.
Further, call a feasible step $\veg$ \emph{augmenting} if $f(\vex+\veg) < f(\vex)$.
An important implication of Proposition~\ref{prop:possum} is that if \emph{any} augmenting step exists, then there exists one in $\G(A)$~\cite[Lemma 3.3.2]{DHK}.

An augmenting step $\veg$ and a \emph{step length} $\lambda \in \N$ form an \emph{$\vex$\hy{}feasible step pair} if $\vel \le \vex + \lambda\veg \le \veu$.
An augmenting step $\veh$ is a \emph{Graver-best step} for $\vex$ if $f(\vex + \veh) \leq f(\vex + \lambda \veg)$ for all $\vex$\hy{}feasible step pairs $(\veg,\lambda) \in \G(A)\times \N$.
A slight relaxation of a Graver-best step is a halfling:
an augmenting step $\veh$ is a \emph{halfling} for $\vex$ if $f(\vex) - f(\vex + \veh) \geq \frac{1}{2}(f(\vex) - f(\vex + \lambda \veg))$ for all $\vex$\hy{}feasible step pairs $(\veg,\lambda) \in \G(A)\times \N$.
A \emph{halfling augmentation procedure} for~\eqref{IP} with a given feasible solution $\vex_0$ works as follows. Let $i \df 0$.
\begin{enumerate}
  \item \label{step1}If there is no halfling for $\vex_i$, return it as optimal.
  \item If a halfling $\veh_i$ for $\vex_i$ exists, set $\vex_{i+1} := \vex_i + \veh_i$, $i \df i+1$, and go to \ref{step1}.
\end{enumerate}

We assume that the bounds $\vel, \veu$ are finite.
Since there are several approaches how to achieve this, we postpone the discussion on dealing with infinite bounds to Section~\ref{sec:infinite_bounds}.

\begin{lemma}[Halfling convergence]
	\label{lem:halfling}
	Given a feasible solution~$\vex_0$ for~\eqref{IP}, the halfling augmentation procedure finds an optimum of~\eqref{IP} in at most $3n \log \left(f(\vex_0) -f(\vex^*)\right) \leq 3n \log \left(f^{[\vel, \veu]}_{\max}\right)$ steps.
\end{lemma}
Before we prove the lemma we need a useful proposition about separable convex functions:
\begin{proposition}
	[{Separable convex superadditivity~\cite[Lemma 3.3.1]{DHK}}]
	\label{prop:superadditivity}
	Let $f(\vex) = \sum_{i=1}^n f_i(x_i)$ be separable convex, let $\vex \in \R^n$,
	and let $\veg_1,\dots,\veg_k \in \R^n$ be vectors that are pairwise conformal.
	Then
	\begin{align}
	\label{eq:conv_ineq}
	f \left( \vex + \sum_{j=1}^k \alpha_j \veg_j \right) - f(\vex)
	&\geq \sum_{j=1}^k \alpha_j \left( f(\vex + \veg_j) - f(\vex) \right)
	\end{align}
	for arbitrary integers $\alpha_1,\dots,\alpha_k \in \N$.
\end{proposition}
\begin{proof}[{Proof of Lemma~\ref{lem:halfling}}]
	Let $\vex^*$ be an optimal solution of~\eqref{IP}.
	By Proposition~\ref{prop:possum} we may write $\vex^* - \vex_0 = \sum_{j=1}^{n'} \lambda_j \veg_j$ such that $\veg_j \sqsubseteq \vex^* - \vex_0$ for all $j \in [n']$, and $n' \leq 2n-2$.
	We apply Proposition~\ref{prop:superadditivity} to $\vex_0$ and the $n'$ vectors $\lambda_j \veg_j$ with $\alpha_j \df 1$, so by~\eqref{eq:conv_ineq} we have
	\[
	0 \geq f(\vex^*) - f(\vex_0) = f\left(\vex_0 + \sum_{j=1}^{n'} \lambda_j \veg_j\right) - f(\vex_0) \geq \sum_{j=1}^{n'} \left(f(\vex_0 + \lambda_j \veg_j) - f(\vex_0) \right),
	\]
	and multiplying by $-1$ gives
	$0 \leq f(\vex_0) - f(\vex^*) \leq \sum_{j=1}^{n'} (f(\vex_0) - f(\vex_0 + \lambda_j \veg_j))$.
	By an averaging argument, there must exist an index $\ell \in [n']$ such that
	\begin{equation}
	f(\vex_0) - f(\vex_0 + \lambda_\ell \veg_\ell) \geq \frac{1}{n'} \left(f(\vex_0) - f(\vex^*)\right) \geq \frac{1}{n'} f_{\max} \enspace . \label{eq:halfling_convergence}
	\end{equation}
	Consider a halfling $\veh$ for $\vex_0$: by definition, it satisfies $f(\vex_0) - f(\vex_0 + \veh) \geq \frac{1}{2}(f(\vex_0) - f(\vex_0 + \lambda_i \veg_i))$.
	Say that the halfling augmentation procedure required $s$ iterations.
	\sv{ 
		For $i \in [s-1]$ we have that 
		\[f(\vex_i) - f(\vex^*) \leq \left(1-\frac{1}{4n-4}\right)\left(f(\vex_{i-1}) - f(\vex^*)\right) = \frac{4n-5}{4n-4}\left(f(\vex_{i-1}) - f(\vex^*)\right)\]
		and, by repeated application of the above, $f(\vex_i) - f(\vex^*) \leq \left(\frac{4n-5}{4n-4}\right)^{i}\left(f(\vex_0) - f(\vex^*)\right)$.
		Since $i$ is not the last iteration, $f(\vex_i) - f(\vex^*) \geq 1$ by the integrality of $f$.
		Take $t \df 4n-4$ and compute an upper bound on $i$.
		We start with $1 \leq \left(\frac{t-1}{t}\right)^i \left(f(\vex_0) - f(\vex^*)\right)$.
		Taking the natural logarithm gives $0 \leq i \ln \left( \frac{t-1}{t} \right) + \ln \left( f(\vex_0) - f(\vex^*) \right)$ and moving terms around then gives $-i \ln \left( \frac{t-1}{t} \right) = i \ln \left( \frac{t}{t-1} \right) \leq \ln \left( f(\vex_0) - f(\vex^*) \right)$.
		Dividing by $\ln \left(\frac{t-1}{t}\right)$, we obtain $i \leq \left( \ln \left( \frac{t}{t-1} \right) \right)^{-1} \ln \left( f(\vex_0) - f(\vex^*) \right)$.
		Now Taylor expansion gives for $t\geq 3$
		$\ln \left(1+\frac{1}{t-1} \right)\geq\frac{1}{t-1}-\frac{1}{2(t-1)^2}$.
		From this it follows for all $t \geq 3$ that
		$\left(\ln \left(1+\frac{1}{t-1}\right)\right)^{-1} \leq t$.
		Plugging back $t \df 4n-4$ we get that for all $n\geq 2$ we have
		$t\geq 3$ and hence
		\[i\leq (4n-4) \ln \left(f(\vex_0)-f(\vex^*)\right) = (4n-4) \cdot \ln 2 \cdot \log_2 \left(f(\vex_0)-f(\vex^*)\right),\]
		and the number of iterations is at most one unit larger.
		Since $f(\vex_0) - f(\vex^*) \leq f_{\max}$ and $\ln(2) = 0.693147 \dots \leq 3/4$, we have that the number of iterations is at most $3 n \log (f_{\max})$.
	}
\end{proof}

Clearly it is now desirable to show how to find halflings quickly.
The following lemma will be helpful in that regard.

\begin{lemma}[Powers of Two] \label{lem:powersoftwo}
  Let $\Gamma_{2} = \{1,2,4,8, \dots\}$ and $\vex$ be a feasible solution of~\eqref{IP}.
  If $\veh$ satisfies $f(\vex + \veh) \leq f(\vex + \lambda \veg)$ for each $\vex$\hy feasible step pair $(\veg, \lambda) \in \G(A) \times \Gamma_{2}$, then $\veh$ is a halfling.
\end{lemma}
\begin{proof}
  Consider any Graver-best step pair $(\veg^*, \lambda^*) \in \G(A) \times \N$, let $\lambda \df 2^{\lfloor \log \lambda^* \rfloor}$,
  and choose $\frac{1}{2} < \gamma \leq 1$ in such a way that
  $\lambda = \gamma \lambda^*$.
  Convexity of $f$ yields
  \begin{align*}
      f(\vex_0) - f(\vex_0 + \lambda \veg^*) &\geq
          f(\vex_0) - \left[ (1-\gamma)f(\vex_0) + \gamma f(\vex_0 + \lambda^* \veg^*) \right] \\
          &= \gamma \left( f(\vex_0) - f(\vex_0 + \lambda^* \veg^*)\right) \\
          &\geq \frac{1}{2} \left( f(\vex_0) - f(\vex_0 + \lambda^* \veg^*)\right) \enspace .
  \end{align*}
This shows that $\lambda \veg^*$ is a halfling, and by the definition of $\veh$, $f(\vex+\veh) \leq f(\vex + \lambda \veg^*)$ and thus $\veh$ is a halfling as well.
\end{proof}
This makes it clear that the main task is to find, for each $\lambda \in \Gamma_2$, a step $\veh$ which is at least as good as any feasible $\lambda \veg$ with $\veg \in \G(A)$.
We need the notion of a \emph{best} solution:
\begin{definition}[$S$-best solution]
  Let $S,P \subseteq \R^n$.
  We say that $\vex^* \in P$ is a solution of
  \begin{equation}
    S\best \left\{f(\vex) \mid \vex \in P\right\} \tag{$S\best$}\label{eq:sbest}
  \end{equation}
  if $f(\vex^*) \leq \min \{f(\vex) \mid \vex \in P \cap S\}$.
If $P \cap S$ is empty, we say $S\best \left\{f(\vex) \mid \vex \in P\right\}$ has no solution.
\end{definition}
In other words, $\vex^*$ has to belong to $P$ and be at least as good as any point in $P \cap S$.
Note that to define the notion of an $S\best$ solution to be a ``no solution'' if $P \cap S = \emptyset$ might look unnatural as one might require \emph{any} $\vex \in P$ if $P \cap S = \emptyset$.
However, this would make~\eqref{eq:sbest} as hard as finding some $\vex \in P$ (just take $S=\emptyset$), but intuitively~\eqref{eq:sbest} should be an easier problem.
The following is a central notion.
\begin{definition}[Augmentation IP]
For an~\eqref{IP} instance $(A, f, \veb, \vel, \veu)$, its feasible solution $\vex \in \Z^n$, and an integer $\lambda \in \N$, the \emph{Augmentation IP} problem is to solve
\begin{equation}
\G(A)\best \{f(\vex+\lambda\veg) \mid A\veg = \vezero, \, \vel \leq \vex+\lambda\veg \leq \veu, \, \veg \in \Z^n \} \enspace . \tag{AugIP} \label{AugIP}
\end{equation}
Let $(A,f,\veb, \vel, \veu)$ be an instance of~\eqref{IP}, $\vex$ a feasible solution, and $\lambda \in \N$. We call the pair $(\vex, \lambda)$ an~\emph{\eqref{AugIP} instance} for $(A,f,\veb, \vel, \veu)$.
If clear from the context, we omit the~\eqref{IP} instance $(A,f,\veb, \vel, \veu)$.
\end{definition}
By Lemma~\ref{lem:powersoftwo} we obtain a halfling by solving~\eqref{AugIP} for each $\lambda \in \Gamma_2$ and picking the best solution.
Given an initial feasible solution $\vex_0$ and a fast algorithm for~\eqref{AugIP}, we can solve~\eqref{IP} quickly:
\begin{lemma}[(\eqref{AugIP} and $\vex_0$) $\implies$ \eqref{IP}] \label{lem:lambda_oracle_initi}
  Given an initial feasible solution $\vex_0$ to~\eqref{IP},~\eqref{IP} can be solved by solving \[3n (\log \|\veu - \vel\|_\infty + 1) \log (f(\vex_0) - f(\vex^*)) \leq 3n (\log \|\veu - \vel\|_\infty + 1) \log \left(f^{[\vel, \veu]}_{\max}\right)\] instances of~\eqref{AugIP}, where $\vex^*$ is any optimum of~\eqref{IP}.
\end{lemma}
\begin{proof}
  Observe that no $\lambda \in \Gamma_2 = \{1,2,4, \dots\}$ greater than $\|\veu - \vel\|_\infty$ results in a non-zero $\vex$\hy feasible step pair.
  Thus, by Lemma~\ref{lem:powersoftwo}, to compute a halfling for $\vex$ it suffices to solve \eqref{AugIP} for all $\lambda \in \Gamma_2$, $\lambda \leq \|\veu - \vel\|_\infty$, and there are at most $\log \|\veu - \vel\|_\infty +1$ of these.
  By Lemma~\ref{lem:halfling}, $3n \log \left(f(\vex_0) - f(\vex^*)\right) \leq 3n \log \left(f_{\max}\right)$ halfling augmentations suffice and we are done.
\end{proof}

\paragraph{Feasibility.} Our goal now is to satisfy the requirement of an initial solution $\vex_0$.
\begin{lemma}[\eqref{AugIP} $\implies \vex_0$] \label{lem:initial}
  Given an instance of~\eqref{IP}, it is possible to compute a feasible solution $\vex_0$ for~\eqref{IP} or decide that~\eqref{IP} is infeasible by solving $ \Oh(n \log (\|A,\veb,\vel,\veu\|_\infty)^2)$ many \eqref{AugIP} instances, plus $\Oh(n^\omega)$ time needed to compute an integral solution of $A \vez = \veb$.
  Moreover, $\la \vex_0 \ra \leq \poly{\la \veb \ra}$.
\end{lemma}
\begin{proof}
We first compute an integer solution to the system of equations $A\vez=\veb$.
This can be done by computing the Hermite normal form of $A$ in time $\Oh(n^{\omega-1} m) \leq \Oh(n^{\omega})$~\cite{StorjohannL96} (using $m \leq n$).
Then either we conclude that there is no integer solution to $A\vez=\veb$ and hence \eqref{IP} is infeasible, or we find a solution $\vez\in\Z^n$ with
  encoding length polynomially bounded in the encoding length of $A, \veb$.

Next, we will solve an auxiliary IP.
Define new relaxed bounds by
  \[
  {\hat l}_i:=\min\{l_i,z_i\},\ \ {\hat u}_i:=\max\{u_i,z_i\},\quad i \in [n],
  \]
and define an objective function $\hat{f} \df \sum_{i=1}^n \hat{f}_i$ as, for each $i \in [n]$, $\hat{f}_i(x_i) \df \textrm{dist}(x_i, [l_i, u_i])$, which is $0$ if $x_i \in [l_i, u_i]$ and $\max \{l_i-x_i, x_i-u_i\}$ otherwise.
This function has at most three linear pieces, the first decreasing, the second constantly zero, and the third increasing, and thus each $\hat{f}_i$ is convex and $\hat{f}$ is separable convex.
Moreover, a solution $\vex$ has $\hat{f}(\vex) = 0$ if and only if $\vel \leq \vex \leq \veu$.

By Lemma~\ref{lem:halfling}, an optimum $\vex_0$ of $\min \left\{\hat{f}(\vex) \,\middle|\, A\vex = \veb, \, \hat{\vel} \leq \vex \leq \hat{\veu}, \, \vex \in \Z^n\right\}$ can be computed by solving $3n (\log \|\hat{\veu} - \hat{\vel}\| + 1) \log \left(\hat{f}^{[\hat{\vel}, \hat{\veu}]}_{\max}\right)$ instances of~\eqref{AugIP}.
Since $\|\hat{\vel}, \hat{\veu}\|_\infty$ is polynomially bounded in $\|A,\veb\|_\infty$ and $\|\vel, \veu\|_\infty$ and, by definition of $\hat{f}$, $\hat{f}^{[\hat{\vel}, \hat{\veu}]}_{\max}$ is bounded by $n \cdot \|\hat{\vel}, \hat{\veu}\|_\infty$, we have that the number of times we have to solve~\eqref{AugIP} is bounded by $\Oh(n\log(\|A,\veb,\vel,\veu\|_\infty)^2)$.
Finally, if $\hat{f}(\vex_0) = 0$ then $\vex_0$ is a feasible solution of~\eqref{IP} and otherwise~\eqref{IP} is infeasible.
\end{proof}

As a corollary of Lemmas~\ref{lem:initial} and~\ref{lem:halfling}, we immediately obtain that a polynomial \eqref{AugIP} algorithm is sufficient for solving~\eqref{IP} in polynomial time:
\begin{corollary}[{\eqref{AugIP} $\implies$~\eqref{IP}}] \label{cor:lambda_oracle}
  Problem~\eqref{IP} can be solved by solving $\Oh(nL^2)$ instances of~\eqref{AugIP}, where $L \df \log ( \|A,f_{\max},\veb,\vel,\veu\|_\infty)$, plus time $\Oh(n^\omega + \min\{n,m\}nm)$.
\end{corollary}

\subsection{The Graphs of $A$ and Treedepth}

\begin{definition}[Primal and dual graph]\label{primaldual-graph}
Given a matrix $A \in \Z^{m \times n}$, its \emph{primal graph} $G_P(A) = (V,E)$ is defined as $V = [n]$ and $E = \left\{\{i,j\} \in \binom{[n]}{2} ~\middle|~ \exists k \in [m]: A_{k,i}, A_{k,j} \neq 0\right\}$.
In other words, its vertices are the columns of $A$ and two vertices are connected if there is a row with non-zero entries at the corresponding columns.
The \emph{dual graph of $A$} is defined as $G_D(A) \df G_P(A^\transpose)$, that is, the primal graph of the transpose of $A$.
\end{definition}

From this point on we always assume that $G_P(A)$ and $G_D(A)$ are connected, otherwise $A$ has (up to row and column permutations) a diagonal structure $A=\left(\begin{smallmatrix}A_1 & & \\ & \ddots & \\ & & A_d\end{smallmatrix} \right)$ and solving~\eqref{IP} amounts to solving $d$ smaller~\eqref{IP} instances independently.

\begin{definition}[Treedepth]
\label{def:tree-depth}
The {\em closure} $\cl(F)$ of a rooted tree $F$
is the graph obtained from $F$ by making every vertex adjacent to all of its ancestors.
We consider both $F$ and $\cl(F)$ as undirected graphs.
The \emph{height} of a tree $F$ denoted $\height(F)$ is the maximum number of vertices on any root-leaf path.
The {\em treedepth} $\td(G)$ of a connected graph $G$ is the minimum
height of a tree $F$ such that $G\subseteq \cl(F)$.
A \emph{$\td$-decomposition of $G$} is a tree $F$ such that $G \subseteq \cl(F)$.
A $\td$-decomposition $F$ of $G$ is \emph{optimal} if $\height(F) = \td(G)$.
\end{definition}

Computing $\td(G)$ is \NPh, but fortunately can be done quickly when $\td(G)$ is small:
\begin{proposition}[{\cite{ReidlRVS:2014}}]\label{prop:tddecomposition}
The treedepth $\td(G)$ of a graph $G$ with an optimal $\td$-decomposition $F$ can be computed in time $2^{\td(G)^2} \cdot |V(G)|$.
\end{proposition}

We define the \emph{primal treedepth of $A$} to be $\td_P(A) \df \td(G_P(A))$ and the \emph{dual treedepth of $A$} to be $\td_D(A) \df \td(G_D(A))$.

We often assume that an optimal $\td$-decomposition is given since the time required to find it is dominated by other terms.
Moreover, in many applications a small $\td$-decomposition of $G_P(A)$ or $G_D(A)$ is clear from the way $A$ was constructed and does not have to be computed as part of the algorithm.

It is clear that a graph $G$ has at most $\td(G)^2 |V(G)|$ edges because the closure of each root-leaf path of a $\td$-decomposition of $G$ contains at most $\td(G)^2$ edges, and there are at most $n$ leaves.
Thus, constructing $G_P(A)$ or $G_D(A)$ can be done in linear time if $A$ is given in its sparse representation.
Throughout we shall assume that $G_P(A)$ or $G_D(A)$ are given.

To facilitate our proofs and to provide more refined complexity bounds we introduce a new parameter called topological height.
This notion is useful in our analysis and proofs, and we later show that it plays a crucial role in complexity estimates of~\eqref{IP} (Theorems~\ref{thm:tdd_lowerbound} and~\ref{thm:tdp_lowerbound}).
It has not been studied elsewhere to the best of our knowledge.
\begin{definition}[Topological height] \label{def:topheight}
A vertex of a rooted tree $F$ is \emph{degenerate} if it has exactly one child, and \emph{non-degenerate} otherwise (i.e., if it is a leaf or has at least two children).
The \emph{topological height of $F$}, denoted $\ttd(F)$, is the maximum number of non-degenerate vertices on any root-leaf path in $F$.
Equivalently, $\ttd(F)$ is the height of $F$ after contracting each edge from a degenerate vertex to its unique child.
Clearly, $\ttd(F) \leq \height(F)$.

We shall now define the level heights of $F$, which relate to lengths of paths between non-degenerate vertices.
For a root-leaf path $P=(v_{b(0)}, \dots, v_{b(1)}, \dots, v_{b(2)}, \dots, v_{b(e)})$ with $e$ non-degenerate vertices $v_{b(1)}, \dots, v_{b(e)}$ (potentially $v_{b(0)} = v_{b(1)}$), define $k_1(P) \df |\{v_{b(0)}, \dots, v_{b(1)}\}|$, $k_i(P) \df |\{v_{b(i-1)}, \dots, v_{b(i)}\}|-1$ for all $i \in [2,e]$, and $k_i(P) \df 0$ for all $i > e$.
For each $i \in [\ttd(F)]$, define $k_i(F) \df \max_{P: \text{root-leaf path}} k_i(P)$.
We call $k_1(F), \dots, k_{\ttd(F)}(F)$ the \emph{level heights of $F$}.
See Figure~\ref{fig:td}.
\end{definition}
\begin{figure}[bt]
  \centering
  \begin{subfigure}[b]{0.45\textwidth}
    \lv{\includegraphics[width=\textwidth]{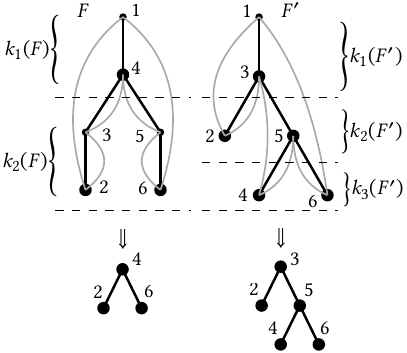}}
    \sv{\includegraphics[width=\textwidth]{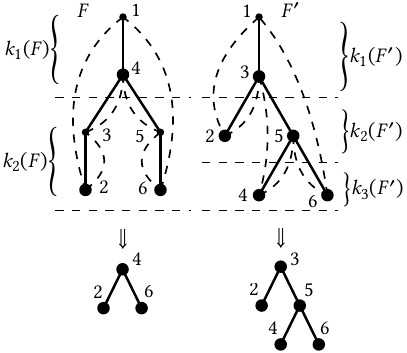}} \hfill
    \caption{Two optimal $\td$-decompositions $F$ and $F'$ of the cycle on six vertices (in \lv{grey}\sv{dashed} edges). Non-degenerate vertices are enlarged. The trees obtained by contracting edges outgoing from vertices with only one child are pictured below. Notice that even though both $F$ and $F'$ are optimal $\td$-decompositions, their topological height differs. Dashed lines depict ``levels'' of $F$ and $F'$, and we have $k_1(F) = k_2(F) = k_1(F') = 2$ and $k_2(F') = k_3(F') = 1$.} \label{fig:td}
  \end{subfigure}
  \hfill
  \begin{subfigure}[b]{0.53\textwidth}
    \hfill
    \lv{\includegraphics[width=\textwidth]{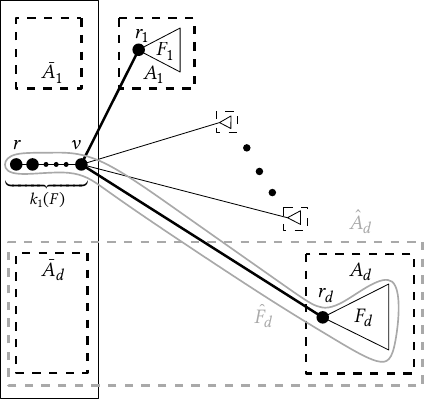}}
    \sv{\includegraphics[width=\textwidth]{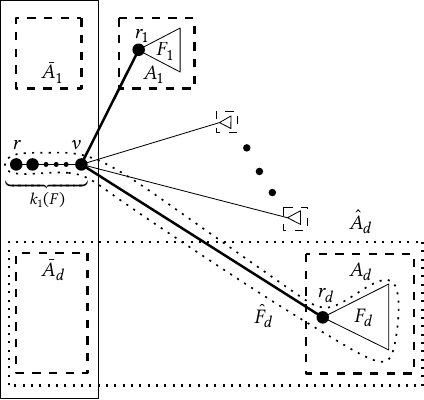}}
  \caption{The situation of Lemma~\ref{lem:decomposition}: a $\td$-decomposition $F$ of $G_P(A)$ pictured in the matrix $A$, the decomposition into smaller blocks $\bar{A}_1, \dots, \bar{A}_d, A_1, \dots, A_d$ derived from $F$ and their $\td$-decompositions $F_1, \dots, F_d$, and a $\td$-decomposition $\hat{F}_d$ of $G_P(\hat{A}_d)$ (Lemma~\ref{lem:fhati}).}
  \label{fig:decomposition}
  \end{subfigure}
  \caption{Illustration of Definitions~\ref{def:tree-depth} and~\ref{def:topheight} (part~\ref{fig:td}) and Lemmas~\ref{lem:decomposition} and~\ref{lem:fhati} (part~\ref{fig:decomposition}).} \label{fig:figs}
\end{figure}

\begin{definition}[Block-structured Matrix]\label{def:primal-decomp}
Let $A \in \Z^{m \times n}$ and $F$ be a $\td$-decomposition of $G_P(A)$.
We say that $A$ is \emph{block-structured along $F$} if either $\ttd(F) = 1$, or if $\ttd(F) > 1$ and the following holds.
Let $v$ be the first non-degenerate vertex in $F$ on a path from the root, $r_1, \dots, r_d$ be the children of $v$, $F_i$ be the subtree of $F$ rooted in $r_i$, and $n_i \df |V(F_i)|$, for $i \in [d]$, and
\begin{align}
	A=
	\left(\begin{array}{ccccc}
		\bar{A}_1 & A_1     &  &\\
		\vdots &      &  \ddots & \\
		\bar{A}_d&      &   & A_d
	\end{array}\right), \tag{block-structure} \label{eq:block-structure}
\end{align}
where, for $i \in [d]$, $\bar{A}_i \in \Z^{m_i \times k_1(F)}$ where $k_1(F)$ is the first level height of $F$ and $m_i \in \N$, $A_i \in \Z^{m_i \times n_i}$, and $A_i$ is block-structured along $F_i$.
Note that $\ttd(F_i) \leq \ttd(F)-1$, $\height(F_i) \leq \height(F)-k_1(F)$, for $i \in [d]$.
\end{definition}
Whenever $A$ and $F$ are given, we will assume throughout this paper that $A$ is block-structured along $F$.
The following Lemma shows that this is without loss of generality as we can always efficiently put $A$ in this format.

\begin{lemma}[Primal Decomposition] \label{lem:decomposition}
  Let $A \in \Z^{m \times n}$, $G_P(A)$, and a $\td$-decomposition $F$ of $G_P(A)$ be given, where $n,m \geq 1$.
  There exists an algorithm which in time $\Oh(n)$ permutes the rows and columns of $A$ such that the resulting matrix $A'$ is block-structured along $F$.
\end{lemma}
\begin{proof}[{Proof of Lemma~\ref{lem:decomposition}}]
Denote by $r$ the root of $F$.
Take any root-leaf path in $F$ and let $v$ be the first non-degenerate vertex (i.e., a leaf or a vertex with at least two children) on this path; note that possibly $v=r$.
Denote by $P$ the path from $r$ to $v$ and observe that $k_1(F)$ is the number of vertices of $P$.
Let $d \df \deg(v)$ be the number of children of $v$ and denote them $r_1, \dots, r_d$, and let $F_1, \dots, F_d$ be the subtrees of $F$ rooted in $r_1, \dots, r_d$.
Clearly $\ttd(F_i) \leq \ttd(F)-1$ for each $i \in [d]$ since $v$ is a non-degenerate vertex.
For a set of column indices $S \subseteq [n]$ denote by $A_{\bullet,S}$ the submatrix of $A$ consisting of exactly the columns indexed by $S$.
For each $i \in [d]$, we obtain $A_i$ from $A_{\bullet, V(F_i)}$ by deleting zero rows, and we obtain $\bar{A}_i$ from $A_{\bullet,V(P)}$ by only keeping rows which are non-zero in $A_{\bullet,V(F_i)}$.
For every row of $A$ whose support is contained in $V(P)$, append its restriction to $V(P)$ to $\bar{A}_1$, and append a zero row to $A_1$.
If $A$ has $\zeta$ zero rows, then append $\zeta$ zero rows to $\bar{A}_1$ and $A_1$; this accounts for zero rows of $A$.
Now, apply the same procedure recursively to $F_1, \dots, F_d$; the base case is when $F_i$ has $\ttd(F_i) = 1$ and $A_i$ is already block-structured by definition.

To finish the proof we need to argue that $A$ has the form~\eqref{eq:block-structure}, in particular, that there is no overlap between the blocks $A_1, \dots, A_d$.
This follows simply from the fact that by the definition of treedepth there are no edges between any two $u \in F_i$, $w \in F_j$ for $i \neq j$, and thus, by definition of $G_P(A)$, there is no row containing a non-zero at indices $u$ and $v$, see Figure~\ref{fig:decomposition}.
\end{proof}

Note that given an~\eqref{IP}, the primal decomposition naturally partitions the right hand side $\veb = (\veb^1, \dots, \veb^d)$ according to the rows of $A_1, \dots, A_d$, and each object of length $n$ (such as bounds $\vel, \veu$, a solution $\vex$, any step $\veg$, or the objective function $f$) into $d+1$ objects according to the columns of $\bar{A}_1, A_1, \dots, A_d$.
For example, we write $\vex = (\vex^0, \vex^1, \dots, \vex^d)$.

By considering the transpose of $A$ we get an analogous notion and a corollary for the dual case:
\begin{definition}[Block-structured Matrix (dual)] \label{def:dual-decomp}
	Let $A \in \Z^{m \times n}$ and $F$ be a $\td$-decomposition of $G_D(A)$.
	We say that $A$ is \emph{dual block-structured along $F$} if either $\ttd(F) = 1$, or if $\ttd(F) > 1$ and the following holds.
	Let $v$ be the first non-degenerate vertex in $F$ on a path from the root, $r_1, \dots, r_d$ be the children of $v$, $F_i$ be the subtree of $F$ rooted in $r_i$, and $m_i \df |V(F_i)|$, for $i \in [d]$, and
  \begin{align}
	A=
	\left(\begin{array}{cccc}
		\bar{A}_1 & \bar{A}_2 & \cdots & \bar{A}_d\\
		A_1 &      &  &  \\
		& A_2      &  &   \\
		&      & \ddots  &   \\
		&      &         & A_d
	\end{array}\right) \enspace . \tag{dual-block-structure} \label{eq:dual-block-structure}
\end{align}

where $d \in \N$, and for all $i \in [d]$, $\bar{A}_i \in \Z^{k_1(F) \times n_i}$, and $A_i \in \Z^{m_i \times n^i}$, $n_i \in \N$, and $A_i$ is block-structured along $F_i$.
Note that $\ttd(F_i) \leq \ttd(F)-1$, $\height(F_i) \leq \height(F)-k_1(F)$, for $i \in [d]$.
\end{definition}

\begin{corollary}[Dual decomposition] \label{cor:dual_decomp}
  Let $A \in \Z^{m \times n}$, $G_D(A)$, and a $\td$-decomposition $F$ of $G_D(A)$ be given, where $n,m \geq 1$.
There exists an algorithm which in time $\Oh(n)$ permutes the rows and columns of $A$ such that the resulting matrix $A'$ is dual block-structured along $F$.
\end{corollary}

Again, the dual decomposition naturally partitions the right hand side $\veb = (\veb^0, \veb^1, \dots, \veb^d)$ according to the rows of $\bar{A}_1, A_1, \dots, A_d$, and each object of length $n$ into $d$ objects according to the columns of $A_1, \dots, A_d$.

\begin{lemma} \label{lem:fhati}
Let $A \in \Z^{n \times m}$, a $\td$-decomposition $F$ of $G_P(A)$ (or $G_D(A)$), and $\bar{A}_i, A_i, F_i$, for all $i \in [d]$, be as in Definitions~\ref{def:primal-decomp} (or~\ref{def:dual-decomp}).
Let $\hat{A}_i \df (\bar{A}_i~A_i)$ (or $\hat{A}_i \df \left(\begin{smallmatrix} \bar{A}_i \\ A_i \end{smallmatrix}\right)$, respectively) and let $\hat{F}_i$ be obtained from $F_i$ by appending a path on $k_1(F)$ new vertices to the root of $F_i$, and the other endpoint of the path is the new root.
Then $\hat{F}_i$ is a $\td$-decomposition of $\hat{A}_i$, $\ttd(\hat{F}_i) < \ttd(F)$, and $\height(\hat{F}_i) \leq \height(F)$.
\end{lemma}
\begin{proof}
Consider Figure~\ref{fig:decomposition}.
The construction of $\hat{F}_i$ can be equivalently described as taking $F$ and deleting all $F_j$, $j \neq i$.
Thus, $\hat{F}_i$ has the claimed properties, in particular $\ttd(\hat{F}_i) < \ttd(F)$ because $v$ was non-degenerate in $F$ but is degenerate in $\hat{F}_i$.
The dual case follows by transposition.
\end{proof}

\subsection{Solving Augmentation IP} \label{sec:augip}
Our goal now is to show that \eqref{AugIP} can be solved quickly when the largest absolute value of a coefficient in $A$, $\|A\|_\infty$, and the primal or dual treedepth $\td_P(A)$ or $\td_D(A)$, respectively, are small.
Together with Corollary~\ref{cor:lambda_oracle}, this implies Theorem~\ref{thm:main}.

To that end, we need two key ingredients.
The first are algorithms solving \eqref{AugIP} quickly when $\|A\|_\infty$ and $\td_P(A)$ or $\td_D(A)$ are small and when restricted to solutions of small $\ell_\infty$- or $\ell_1$-norm, respectively.
The second are theorems showing that this is in fact sufficient because the elements of $\G(A)$ have bounded $\ell_\infty$- and $\ell_1$-norms, respectively.

More specifically, denote by $B_\infty(\rho)$ and $B_1(\rho)$ the $\ell_\infty$- and $\ell_1$-norm balls, respectively, of appropriate dimension and of radius $\rho$, centered at the origin.
Let
\[ g_\infty(A) \df \max_{\veg \in \G(A)}\|\veg\|_\infty \qquad\text{ and }\qquad g_1(A) \df \max_{\veg \in \G(A)}\|\veg\|_1 \enspace .\]
Observe that if $S \subseteq S'$, then an $S'\best$ solution is certainly also an $S\best$ solution.
Since $\G(A) \subseteq B_\infty\left(g_\infty(A)\right)$ and $\G(A) \subseteq B_1\left(g_1(A)\right)$, it follows that a $B_\infty(g_\infty(A))\best$ solution or a $B_1(g_1(A))\best$ solution of~\eqref{AugIP} is also a $\G(A)\best$ solution.
This implies that solving~\eqref{AugIP} roughly amounts to solving an instance of~\eqref{IP} with an additional norm bound.

The plan is as follows.
In Sections~\ref{sec:augip:primal} and~\ref{sec:augip:dual} we will prove that~\eqref{AugIP} can be solved quickly when $\td_P(A)$ and $g_\infty(A)$ or $\td_D(A)$ and $g_1(A)$ are small, respectively.
In Sections~\ref{sec:norms:primal} and~\ref{sec:norms:dual} we show that $g_\infty(A)$ and $g_1(A)$ are small if $\|A\|_\infty$ and either $\td_P(A)$ or $\td_D(A)$ are small, respectively.
Theorem~\ref{thm:main} then follows.

\subsubsection{Primal Treedepth}\label{sec:augip:primal}
\begin{lemma}[Primal Lemma] \label{lem:primal}
Problem~\eqref{AugIP} can be solved in time $\td_P(A)^2 (2g_\infty(A)+1)^{\td_P(A)} n$.
\end{lemma}
\begin{proof}
Let $F$ be an optimal $\td$-decomposition of $G_P(A)$.
The proof proceeds by induction on $\ttd(F) \leq \td_P(A)$.
For that, we prove a slightly more general claim:
\begin{claim*}
  Given $\rho \in \N$, there is an algorithm running in time $\td_P(A)^2 (2\rho+1)^{\td_P(A)} n$ which solves
 \begin{equation*}
 B_\infty(\rho)\best \{f(\veg) \mid A\veg = \veb, \, \vel \leq \veg \leq \veu, \, \veg \in \Z^n\} 
 \end{equation*}
  for any separable-convex function $f$.
\end{claim*}
The statement of the Lemma is obtained by the following substitution.
For a given~\eqref{AugIP} instance $(\vex, \lambda)$, solve the auxiliary problem above with $\rho \df g_\infty(A)$, $f(\veg) \df f(\vex + \lambda \veg)$, $\veb \df \vezero$, $\vel \df \ceil*{\frac{\vel- \vex}{\lambda}}$, and $\veu \df \floor*{\frac{\veu-\vex}{\lambda}}$.
If $f$ of~\eqref{IP} was separable convex, then the newly defined $f$ is also separable convex.
The returned solution is a solution of~\eqref{AugIP} because $\G(A) \subseteq B_\infty\left(g_\infty(A)\right)$.

As the base case, if $\ttd(F) = 1$, then $F$ is a path, meaning that $A$ has $\td_P(A)$ columns.
An optimal solution is found simply by enumerating all $(2\rho+1)^{\td_P(A)}$ integer vectors
$\veg \in [-\rho,\rho]^{\td_P(A)} \cap [\vel, \veu]$, for each checking $A\veg=\veb$ and evaluating $f$, and returning the best feasible one.
Since the number of rows of $A$ is at most its number of columns, which is $\td_P(A)$, checking whether $A\veg=\veb$ takes time at most $\td_P(A)^2$ for each $\veg$.

As the induction step, we assume $A$ is block-structured along $F$ (otherwise apply Lemma~\ref{lem:decomposition}), hence we have matrices $\bar{A}_1, \dots, \bar{A}_d, A_1, \dots, A_d$ for some $d$ and $\td$-decompositions $F_1, \dots, F_d$ for $G_P(A_1), \dots, G_P(A_d)$, respectively, with, for each $i \in [d]$, $\bar{A}_i$ having $k_1(F)$ columns, $F_i$ having $\ttd(F_i) < \ttd(F)$, and $\td_P(A_i) \leq \td_P(A) - k_1(F)$.
Now iterate over all vectors $\veg^0 \in \Z^{k_1(F)}$ in $[-\rho, \rho]^{k_1(F)} \cap [\vel^0, \veu^0]$ and for each use the algorithm which exists by induction to compute $d$ vectors $\veg^i$, $i \in [d]$, such that $\veg^i$ is a solution to
\begin{equation}
B_\infty(\rho)\best \{f(\veg^i) \mid A_i\veg^i = -\bar{A}_i \veg^0 +  \veb^i, \, \vel^i \leq \veg^i \leq \veu^i, \, \veg^i \in \Z^{n_i}\} \enspace . \label{eq:primal_subproblem}
\end{equation}
Finally return the vector $(\veg^0, \dots, \veg^d)$ which minimizes $\sum_{i=0}^d f(\veg^i)$.
If $\veg^i$ is undefined for some $i \in [d]$ because the subproblem~\eqref{eq:primal_subproblem} has no solution, report that the problem has no solution.

Let $k \df \td_P(A) - k_1(F)$.
There are $(2\rho+1)^{k_1(F)}$ choices of $\veg^0$, and computing the solution $(\veg^1, \dots, \veg^d)$ for each takes time at most $\sum_{i=1}^d k^2 (2\rho+1)^{k} n_i = k^2 (2\rho+1)^{k} n$.
For each choice we also need to compute the product $-\bar{A}_i \veg^0$, which is possible in time $k_1(F)\cdot \td_P(A)$ since the number of rows of $\bar{A}_i$ is at most $\td_P(A)$.
The total time needed is thus $(2\rho+1)^{k_1(F)} \cdot \left(\td_P(A) \cdot k_1(F) + k^2 (2\rho+1)^{k}\right) n \leq \td_P(A)^2 (2\rho+1)^{\td_P(A)} n$.
\end{proof}

\lv{
	\begin{remark}
	A claim analogous to the above but for the more permissive parameter treewidth follows from Freuder's algorithm~\cite{Freu}.
	However, the proof above is simpler and will be later extended in ways which would not be possible if we used treewidth.
	\end{remark}
}

\subsubsection{Dual Treedepth}\label{sec:augip:dual}
\begin{lemma}[Dual Lemma] \label{lem:dual}
	Problem~\eqref{AugIP} can be solved in time $(2\|A\|_\infty g_1(A)+1)^{\Oh(\td_D(A))} n$.
\end{lemma}
\begin{proof}
	We solve an auxiliary problem analogous to the one in Lemma~\ref{lem:primal}: given $\rho \in \N$ and a separable convex function $f$, solve
	\[B_1(\rho)\best \{f(\veg) \mid A\veg = \veb, \, \vel \leq \veg \leq \veu, \, \veg \in \Z^n\} \enspace . \]
	The lemma then follows by the same substitution described at the beginning of the proof of Lemma~\ref{lem:primal}.
	We assume that $\|\veb\|_\infty \leq \rho \|A\|_\infty$ since otherwise there is no solution within $B_1(\rho)$.
	
	Let $F$ be an optimal $\td$-decomposition of $G_D(A)$.
	We define the algorithm recursively over $\ttd(F)$.
	If $\ttd(F) \geq 2$, we assume $A$ is dual block-structured along $F$ (otherwise apply Corollary~\ref{cor:dual_decomp}) and we have, for every $i \in [d]$, matrices $A_i, \bar{A}_i, \hat{A}_i$ and a tree $\hat{F}_i$ (see Lemma~\ref{lem:fhati}) with the claimed properties, and a corresponding partitioning of $\veb, \vel, \veu, \veg$ and $f$.
	If $\ttd(F) = 1$, let $d \df n$ and $\hat{A}_i \df A_{\bullet,i}$, for all $i \in [d]$, be the columns of $A$, and let $\veb^1, \dots, \veb^n$ be empty vectors (i.e., vectors of dimension zero).
	
	The crucial observation is that for every solution $\veg$ of $A\veg = \veb$ with $\|\veg\|_1 \leq \rho$ and each $i \in [d]$, both $\bar{A}_i \veg^i$ and $\sum_{j=1}^i \bar{A}_j \veg^j$ belong to $R \df [-\rho\|A\|_\infty, \rho \|A\|_\infty]^{k_1(F)}$.
	For every $i \in [d]$ and every $\ver \in R$, solve
	\begin{equation}
	B_1(\rho)\best \{f^i(\veg^i) \mid \hat{A}_i\veg^i = (\ver, \veb^i), \, \vel^i \leq \veg^i \leq \veu^i, \, \veg^i \in \Z^{n_i}\} \enspace . \label{eq:dual_subproblem}
	\end{equation}
	In the base case when $\hat{A}_i$ has only one column, we simply enumerate all $\veg^i \in [\vel^i, \veu^i] \cap [-\rho,\rho]$, check whether the equality constraints are satisfied, and return the best feasible choice.
	If $\ttd(F) > 1$, then we use recursion to solve~\eqref{eq:dual_subproblem}.
	The recursive call is well-defined, since, for all $i \in [d]$, $\ttd(\hat{F}_i) < \ttd(F)$ and $\hat{F}_i$ is a $\td$-decomposition of $G_D(\hat{A}_i)$.
	Next, we show how to ``glue'' these solutions together.
	
	Let $\ver \in R$ and denote by $\veg^i_{\ver}$ a solution to the subproblem~\eqref{eq:dual_subproblem}; by slight abuse of notation, when the subproblem has no solution, we define $f^i(\veg^i_{\ver}) \df +\infty$.
	Now we need to find such $\ver_1, \dots, \ver_d \in R$ that $\sum_{i=1}^d \ver_i= \veb^0$ and $\sum_{i=1}^d f^i(\veg^i_{\ver_i})$ is minimized.
	This is actually a form of the $(\min,+)$\hy convolution problem, a fact which we will use later.
	For now it suffices to say that this problem can be easily solved using dynamic programming in $d$ stages: our DP table $D$ shall have an entry $D(i,\ver)$ for $i \in [d]$ and $\ver \in R$ whose meaning is the minimum $\sum_{j=1}^i f^j(\veg^j_{\ver_j})$ where $\sum_{j=1}^i \ver_j = \ver$.
	To compute $D$, set $D(0, \ver) \df 0$ for $\ver = \vezero$ and $D(0, \ver) \df +\infty$ otherwise, and for $i \in [d]$, set
	$$D(i, \ver) \df \min_{\substack{\ver', \ver'' \in R:\\\ver'+\ver'' = \ver}} D(i-1, \ver') + f^i(\veg^i_{\ver''}) \enspace .$$
	The value of the solution is $D(d, \veb^0)$ and the solution $\veg = (\veg^1, \dots, \veg^d)$ itself can be computed easily with a bit more bookkeeping in the table $D$.
	If $D(d,\veb^0) = +\infty$, report that the problem has no solution.
	Another important observation is this: in the DP above we computed the solution of the auxiliary problem not only for the right hand side $\veb$, but for \emph{all} right hand sides of the form $(\ver, \veb^1, \dots, \veb^d)$ where $\ver \in R$ and $\veb^1,\dots,\veb^d$ are fixed.
	We store all of these intermediate results in an array (an approach also known as ``memoization'').
	When the algorithm asks for solutions of such instances, we simply retrieve them from the array of intermediate results instead of recomputing them. This is important for the complexity analysis we will describe now.
	
	The recursion tree has $\ttd(F)$ levels, which we number $1, \dots, \ttd(F)$, with level $1$ being the base of the recursion.
	Let us compute the time required at each level.
	In the base case $\ttd(F)=1$, recall that the matrix $\hat{A}_i$ in subproblem~\eqref{eq:dual_subproblem} is a single column with $\height(F)$ rows, and solving~\eqref{eq:dual_subproblem} amounts to trying at most $2\rho+1$ feasible valuations of $\veg^i$ (which is a scalar variable) satisfying $\vel^i \leq \veg^i \leq \veu^i$ and returning the best feasible one.
	Since there are $n$ columns in total, computing the solutions of~\eqref{eq:dual_subproblem} takes time $(2\rho+1)n$.
	Let $N_1$ be the number of leaves of $F$, and let $\alpha_j$, $j \in [N_1]$, denote the number of columns corresponding to the $j$-th leaf.
	``Gluing'' the solutions is done by solving $N_1$ DP instances with $\alpha_1, \dots, \alpha_{N_1}$ stages, where $\sum_{i=1}^{N_1} \alpha_i = n$.
	This takes time $\sum_{i=1}^{N_1} |R|^2 \cdot \alpha_i \leq (2\|A\|_\infty \rho+1)^{\td_D(A)} \cdot n$, since a $\td$-decomposition of each column is a path on $\td_D(A)$ vertices.
	In total, computing the first level of recursion takes time $(2\|A\|_\infty \rho +1)^{\td_D(A)} n$.
	
	Consider a recursion level $\ell \in [2,\ttd(F)]$ and subproblem~\eqref{eq:dual_subproblem}.
	The crucial observation is that when the algorithm asks for the answer to~\eqref{eq:dual_subproblem} for one specific $\ver' \in R$, an answer for \emph{all} $\ver \in R$ is computed; recall that the last step of the DP is to return $D(d,\ver')$ but the table contains an entry $D(d,\ver)$ for all $\ver \in R$.
	Thus the time needed for the computation of all $\veg^i_{\ver}$ has been accounted for in lower levels of the recursion and we only have to account for the DP at the level $\ell$.
	Let $R'$ be the analogue of $R$ for a specific subproblem at level $\ell$, and let $A'$ be the corresponding submatrix of $A$ and $F' \subseteq F$ be a $\td$-decomposition of $G_D(A')$.
	We have that $|R'| \leq (2\|A\|_\infty \rho+1)^{k_1(F')}$, with $k_1(F') \leq \td_D(A)$.
	Note that the levels here are defined bottom up, hence all leaves are at level $1$, and an inner node of $F$ is at level $\ell$ if $\ell-1$ is the largest level of its children; in particular a level does \emph{not} correspond to the distance from the root.
	Let $N_\ell$ be the number of vertices of $F$ at level $\ell$.
	The number of subproblems on level $\ell$ is exactly $N_{\ell}$, so computation of the $\ell$-th level takes time at most $|R|^2 \cdot N_\ell \leq (2\|A\|_\infty \rho +1)^{\td_D(A)} N_\ell$.
	Adding up across all levels we get that the total complexity is at most $\left(n + \sum_{\ell=2}^{\ttd(F)} N_\ell\right) \cdot (2\|A\|_\infty \rho +1)^{\td_D(A)}$ where $\sum_{\ell=2}^{\ttd(F)} N_\ell < n$ since $F$ has $n$ leaves and each level corresponds to a vertex with degree at least $2$.
	The lemma follows.
\end{proof}

\subsection{Bounding Norms} \label{sec:norms}
We begin by using the Steinitz Lemma to obtain a basic bound on $g_1(A)$.

\begin{proposition}[Steinitz~\cite{Steinitz1913}, Sevastjanov, Banaszczyk~\cite{SevastjanovBanasczyk1997}]
\label{prop:steinitz}
  Let $\|\cdot\|$ be any norm, and let $\vex_1, \dots, \vex_n \in \R^d$ be such that $\|\vex_i\| \leq 1$ for $i \in [n]$ and $\sum_{i=1}^n \vex_i = \vezero$.
  Then there exists a permutation $\pi \in S_n$ such that for each $k \in [n]$, $\|\sum_{i=1}^k \vex_{\pi(i)}\| \leq d$.
\end{proposition}

\begin{lemma}[Base bound]
  \label{lem:bound1}
  Let $A \in \Z^{m \times n}$. Then $g_1(A) \leq (2m \|A\|_\infty+1)^m$.
\end{lemma}
\begin{proof}
  Let $\veg \in \G(A)$.
We define a sequence of vectors in the following manner.
If $g_i \geq 0$, we add $g_i$ copies of the $i$-th column of $A$ to the sequence,
if $g_i < 0$ we add $|g_i|$ copies of the negation of column $i$ to the sequence, either way obtaining vectors $\vev_1^i, \dots, \vev_{|g_i|}^i$.

Clearly, the vectors $\vev_j^i$ sum up to $\vezero$ as $\veg \in \ker_{\Z}(A)$ and their $\ell_\infty$-norm is bounded by $\|A\|_\infty$.
Using the Steinitz Lemma, there is a reordering $\veu^1,\dots,\veu^{\|\veg\|_1}$ (i.e., $\vev_j^i = \veu^{\pi(i,j)}$ for some permutation $\pi$)
of this sequence such that
each prefix sum $\vep_k \df \sum_{j=1}^k \veu^j$
is bounded by $m\|A\|_\infty$ in the $l_\infty$-norm. Clearly
\begin{displaymath}
  | \{ \vex \in \Z^m \mid \|\vex\|_{\infty} \leq m\|A\|_\infty \} | = \left( 2m \|A\|_\infty + 1 \right)^m \enspace .
\end{displaymath}
Assume for contradiction that $\|\veg\|_1 > \left( 2m \|A\|_\infty + 1 \right)^m$. Then two of these prefix sums are the same, say, $\vep_{\alpha} = \vep_{\beta}$ with $1 \leq \alpha < \beta \leq \|\veg\|_1$.
Obtain a vector $\veg'$ from the sequence $\veu^1, \dots, \veu^\alpha, \veu^{\beta+1}, \dots, \veu^{\|\veg\|_1}$ as follows: begin with $g'_i \df 0$ for each $i \in [n]$, and for every $\veu^\ell$ in the sequence, set
\[g'_i \df
\begin{cases}
 g'_i + 1 & \text{ if } \pi^{-1}(\ell) = (i,j) \text{ and } g_i \geq 0 \\
 g'_i - 1 & \text{ if } \pi^{-1}(\ell) = (i,j) \text{ and } g_i < 0 \enspace .
 \end{cases}
\]
Similarly obtain $\veg''$ from the sequence $\veu^{\alpha+1} \dots, \veu^{\beta}$.
We have $A\veg'' = \vezero$ because $\vep_\alpha - \vep_\beta = \vezero$ and thus $\veg'' \in \ker_{\Z}(A)$, and thus also $\veg' \in \ker_{\Z}(A)$.
Moreover, both $\veg'$ and $\veg''$ are non-zero and satisfy $\veg', \veg'' \sqsubseteq \veg$.
This is a contradiction with $\sqsubseteq$-minimality of $\veg$, hence $\|\veg\|_1 \leq (2m\|A\|_\infty + 1)^m$, finishing the proof.
\end{proof}

\subsubsection{Norm of Primal Treedepth}\label{sec:norms:primal}
\begin{lemma}[Primal Norm]\label{lem:primal_norm}
	Let $A \in \Z^{m \times n}$, and $F$ be a $\td$-decomposition of $G_P(A)$.
	Then there exists a constant $\alpha \in \N$ such that
	\[
	\stackinset{l}{39pt}{b}{-9pt}{\tiny\rotatebox{33}{$\underbrace{\kern21pt}_{\ttd(F)-1}$}}{%
		$g_\infty(A) \leq 2^{2^{\rdots^{2^{(2\|A\|_\infty)^{2^{\ttd(F)} \cdot \alpha \cdot \td_P(A)^2}}}}}$}
	\]
\end{lemma}
We will need a new and powerful lemma due to Klein~\cite{Klein}.
\begin{proposition}[Klein~\cite{Klein}] \label{prop:klein}
	Let $T_1, \dots, T_n \subseteq \Z^d$ be multisets all belonging to one orthant where all elements $\vet \in T_i$ have bounded size $\|\vet\|_\infty \leq C$ and where
	$$\sum_{\vet \in T_1} \vet = \sum_{\vet \in T_2} \vet = \cdots = \sum_{\vet \in T_n} \vet \enspace .$$
	Then there exists a constant $\alpha \in \N$ and non-empty submultisets $S_1 \subseteq T_1, \dots, S_n \subseteq T_n$ of bounded size $|S_i| \leq (dC)^{dC^{\alpha d^2}}$ such that
	$$\sum_{\ves \in S_1} \ves = \sum_{\ves \in S_2} \ves = \cdots = \sum_{\ves \in S_n} \ves \enspace .$$
\end{proposition}

\begin{proof}[{Proof of Lemma~\ref{lem:primal_norm}}]
	We will proceed by induction on $\ttd(F)$.	
	In the base case when $\ttd(F) = 1$, $G_P(A)$ is a path and thus $A$ has $\td_P(A)$ columns.
	Observe that the number of rows of $A$ is bounded by $\td_P(A)$ as we assume purity.
	By Lemma~\ref{lem:bound1} we then have that
	\[g_\infty(A) \leq \td_P(A) g_1(A) \leq \td_P(A)(2\|A\|_\infty \td_P(A) +1)^{\td_P(A)} \leq 2^{2\alpha \cdot \td_P(A)^2 + \log 2\|A\|_\infty} \enspace .\]
	
	In the inductive step, we assume $A$ is block-structured along $F$ (otherwise apply Lemma~\ref{lem:decomposition}).
	Let $\hat{A}_i = (\bar{A}_i~A_i) \in \Z^{m_i \times k' + n_i}$ and $\hat{F}_i$ as in Lemma~\ref{lem:fhati}, and let $\hat{g}_\infty \df \max_{i \in [d]} g_\infty(\hat{A}_i)$.
	Note that $\td_P(\hat{A}_i) \leq \td_P(A)$.
	Since $\hat{F}_i$ is a $\td$-decomposition of $G_P(\hat{A}_i)$ and $\ttd(\hat{F}_i) < \ttd(F)$, we may apply induction on $\hat{A}_i$, showing
	\begin{equation}
	\stackinset{l}{39pt}{b}{-9pt}{\tiny\rotatebox{33}{$\underbrace{\kern21pt}_{\ttd(F)-2}$}}{%
		$\hat{g}_\infty \leq 2^{2^{\rdots^{2^{(2\|A\|_\infty)^{2^{\ttd(F)} \cdot \alpha \cdot \td_P(A)^2}}}}}$} \label{eq:primal_bound_induction}
	\end{equation}

	Consider $\veg = (\veg^0, \veg^1, \dots, \veg^d) \in \G(A)$.
	For each $i \in [d]$, decompose $(\veg^0, \veg^i) = \sum_{j=1}^{N_i} (\veh_j^0, \veh_j^i)$ with $(\veh_j^0, \veh_j^i) \in \G(\hat{A}_i)$ by the Positive Sum Property (Proposition~\ref{prop:possum}).
	Let $T_i \df \left\{\veh^0_j \ \middle|\ j \in [N_i]\right\}$ and observe that $\max_{\vet \in T_i} \|\vet\|_\infty \leq g_\infty(\hat{A}_i) \leq \hat{g}_\infty$.
	If applying Proposition~\ref{prop:klein} to $T_1, \dots, T_d$ yielded sets $S_1, \dots, S_d$ such that $S_i \subsetneq T_i$ for some $i \in [d]$ then $\veg$ was not $\sqsubseteq$-minimal, a contradiction.
	Let $k_1 \df k_1(F)$.
	Thus Proposition~\ref{prop:klein} implies, for each $i \in [d]$, 
	\[|T_i| \leq (k_1 \hat{g}_\infty)^{k_1 \hat{g}_\infty^{\alpha k_1^2}}
	= 2^{2^{\alpha k_1^2 + \log(k_1 \hat{g}_\infty) + \log \log (k_1 \hat{g}_\infty)}} \leq 2^{2^{2 \alpha k_1^2 + \log \hat{g}_\infty}}\] 
	and $\|(\veg^0, \veg^i)\|_\infty \leq \hat{g}_\infty |T_i|$,
	which in turn means that
	$\|\veg\|_\infty \leq \hat{g}_\infty \max_{i \in [d]} |T_i|$.
	Note that $2^{{2}^{2\alpha k_1^2 + \log \hat{g}_\infty}} \hat{g}_\infty \leq 2^{{2}^{2\alpha k_1^2 + \log \hat{g}_\infty + \log \log \hat{g}_\infty}}$.
	To simplify, let $\zeta \df 2\alpha k_1^2 + \log \hat{g}_\infty + \log \log \hat{g}_\infty$ so that the expression reads $2^{2^\zeta}$.
	Plugging in the bound~\eqref{eq:primal_bound_induction} for $\hat{g}_\infty$ then gives
	\begin{multline*}
	\zeta = 2\alpha k_1^2 + \log \hat{g}_\infty + \log \log \hat{g}_\infty \leq 2\alpha k_1^2 + 2\log \hat{g}_\infty \leq \\
	2\alpha k_1^2 + 2\cdot \stackinset{l}{2pt}{b}{-10pt}{\tiny\rotatebox{29}{$\underbrace{\kern15pt}_{\ttd(F)-3}$}}{%
		$2^{\rdots^{2^{(2\|A\|_\infty)^{2^{\ttd(F)-3} \cdot \alpha \cdot \td_P(A)^2}}}}$}
	\leq \stackinset{l}{2pt}{b}{-10pt}{\tiny\rotatebox{29}{$\underbrace{\kern15pt}_{\ttd(F)-3}$}}{%
		$2^{\rdots^{2^{(2\|A\|_\infty)^{2^{\ttd(F)} \cdot \alpha \cdot \td_P(A)^2}}}}$}, \quad \text{and thus,}
	\end{multline*}
	\[2^{2^{\zeta}} \leq 
	\stackinset{l}{8pt}{b}{-6pt}{\tiny\rotatebox{29}{$\underbrace{\kern8pt}_{\ttd(F)-3}$}}{%
		$2^{2^{2^{{\rdots^{2^{(2\|A\|_\infty)^{2^{\ttd(F)} \cdot \alpha \cdot \td_P(A)^2}}}}}}}$}
	\leq 
	\stackinset{l}{39pt}{b}{-9pt}{\tiny\rotatebox{33}{$\underbrace{\kern21pt}_{\ttd(F)-1}$}}{%
		$g_\infty(A) \leq 2^{2^{\rdots^{2^{(2\|A\|_\infty)^{2^{\ttd(F)} \cdot \alpha \cdot \td_P(A)^2}}}}}$}
	\]
\end{proof}

\subsubsection{Norm of Dual Treedepth}\label{sec:norms:dual}
\begin{lemma}[Dual Norm]\label{lem:dual_norm}
Let $A \in \Z^{m \times n}$, $F$ be a $\td$-decomposition of $G_D(A)$, and let $K \df \max_{P: \text{root-leaf path in } F} \prod_{i=1}^{\ttd(F)} \left(k_i(P)+1\right)$.
Then $g_1(A) \leq (3\|A\|_\infty K)^{K-1}$.
\end{lemma}
\begin{proof}
  The proof will proceed by induction over $\ttd(F)$.
  In the base case we have $\ttd(F) = 1$ and thus $G_D(A)$ is a path with $\height(F)$ vertices, meaning $A$ has $\height(F)$ rows.
  Now we use the Base bound of Lemma~\ref{lem:bound1} to get that $g_1(A) \leq (2\|A\|_\infty \height(F)+1)^{\td_D(A)}$, which is at most $(3\|A\|_\infty K)^{K-1}$, where $K = \height(F)+1 = k_1(F)+1$.
  (Note that $k_1(F)=k_1(P)$ for all root-leaf paths $P$ in $F$ since all paths share an identical segment from the root to the first non-degenerate vertex.)

  For the inductive step, assume that the claim holds for all trees of topological height less than $\ttd(F)$.
  Let $\veg \in \G(A)$ and $K' \df \max_{P: \text{root-leaf path in } F} \prod_{i=2}^{\ttd(F)} \left(k_i(P)+1\right)$.
  For each $i \in [d],$ $\veg^i$ has a decomposition into elements $\veg^i_j$ of $\G(A_i)$, and by induction we have $\|\veg^i_j\|_1 \leq g_1(A_i) \leq (3 \|A\|_\infty  K')^{K'-1} =: \hat{g}_1$.
  Construct a sequence of vectors as follows: for each $i \in [d]$ and each $\veg^i_j$ in the decomposition of $\veg^i$, insert $\vev^i_j := \bar{A}_i \veg^i_j$ into the sequence.
  Note that $\|\vev^i_j\|_\infty \leq \|A\|_\infty \hat{g}_1$.
  Denote the resulting sequence $\veu^1, \dots, \veu^N$.

  Applying the Steinitz Lemma (Proposition~\ref{prop:steinitz}) to this sequence, we obtain its permutation $\veu^{\pi(1)}, \dots, \veu^{\pi(N)}$ such that the $\ell_\infty$-norm of each of its prefix sums is bounded by $k_1(F) \|A\|_\infty \hat{g}_1$.
  As in the proof of Lemma~\ref{lem:bound1}, we will prove that no two prefix sums are the same, thus $N \leq (2 k_1(F) \|A\|_\infty \hat{g}_1+1)^{k_1(F)}$ and subsequently $\|\veg\|_1 \leq N \hat{g}_1 \leq \hat{g}_1 (2 k_1(F) \|A\|_\infty \hat{g}_1+1)^{k_1(F)}$.
  Plugging in $\hat{g}_1 = (3 \|A\|_\infty  K')^{K'-1} \leq (3 \|A\|_\infty  K)^{K'-1}$ and simplifying yields
\[
\|\veg\|_1 \leq (3 \|A\|_\infty  K)^{K'-1} \cdot (3 \|A\|_\infty K)^{k_1(F) K'} = (3 \|A\|_\infty K)^{K-1} \enspace .
\]

  Assume to the contrary that some two prefix sums $\vep_{\alpha}$ and $\vep_{\beta}$, for $\alpha < \beta$, are identical.
  Then the sequence $\veu^{\alpha+1}, \dots, \veu^{\beta}$ sums up to zero and we may ``work backward'' from it to obtain an integer vector $\bar{\veg} \sqsubset \veg$, which is a contradiction to $\veg \in \G(A)$.
  Specifically, $\bar{\veg}$ can be obtained by initially setting $\bar{\veg} = \vezero$ and then, for each $\gamma \in [\alpha+1, \beta]$, if $\pi^{-1}(\gamma) = (i,j)$, setting $\bar{\veg}^i := \bar{\veg}^i + \veg^i_j$.
\end{proof}
\begin{remark}
Our definition of $K$ allows us to recover the currently best known upper bounds on $g_1(A)$ from Lemma~\ref{lem:dual_norm}.
Specifically, Knop et al.~\cite[Lemma 10]{KnopPW:2018} show that $g_1(A) \leq (2\|A\|_\infty +1)^{2^{\td_D(A)}-1}$.
This pertains to the worst case when $\ttd(F) = \height(F) = \td_D(A)$.
Then, we have $K=\prod_{i=1}^{\ttd(F)} (k_i(P)+1) = 2^{\td_D(A)}$ and our bound essentially matches theirs.
On the other hand, our bound is better in scenarios when $\ttd(F) < \height(F)$ and $K$ is attained by some path with $k_i(P) > 1$ for some $i \in \ttd(F)$.
A particular example of this are $N$-fold and tree-fold matrices discussed in Section~\ref{sec:apps}.
\end{remark}

\subsection{The Proof}\label{sec:mainrproof}
\begin{proof}[Proof of Theorem~\ref{thm:main}]
We run two algorithms in parallel, terminate when one of them terminates, and return its result.
In the \emph{primal} algorithm, let $G(A) = G_P(A)$, $\td(A) = \td_P(A)$ and $p=\infty$.
In the \emph{dual algorithm}, let $G(A) = G_D(A)$, $\td(A) = \td_D(A)$ and $p=1$.
The description of both algorithms is then identical.

First, run the algorithm of Proposition~\ref{prop:tddecomposition} on $G(A)$ to obtain its optimal $\td$-decomposition.
By Lemmas~\ref{lem:primal_norm} and~\ref{lem:dual_norm} there is a computable function $g'$ such that the maximum $\ell_p$-norm of elements of $\G(A)$ is bounded by $g'(\|A\|_\infty,\td(A))$.
By Lemmas~\ref{lem:primal} and~\ref{lem:dual}, there is a computable function $g''$ such that~\eqref{AugIP} is solvable in time $g''(g'(\td(A), \|A\|_p), \|A\|_p, \td(A))$ and thus in time $g(\|A\|_p, \td(A))$ for some computable function $g$.
Then, solve~\eqref{IP} using the algorithm of Corollary~\ref{cor:lambda_oracle} in the claimed time.
\end{proof}

\section{A General Framework for Improved Complexity} \label{sec:framework}
We develop an oracle-based framework for designing algorithms for~\eqref{IP} with improved complexity bounds.
In Section~\ref{sec:oracles} we introduce the individual parts of the framework.
In Section~\ref{sec:warmup} we provide a warm-up example.
In Section~\ref{sec:master} we develop a ``Master Lemma'' which connects the individual components together and is later (Section~\ref{sec:apps}) used to derive concrete time complexity bounds.
The next five sections deal with the individual components of the framework: feasibility oracles (Section~\ref{sec:feasibility}), handling infinite bounds (Section~\ref{sec:infinite_bounds}), proximity bounds (Section~\ref{sec:proximity}), relaxation oracles (Section~\ref{sec:relaxation}), reducibility bounds (Section~\ref{sec:reducibility}), and augmentation oracles (Section~\ref{sec:almostlinear}).

\subsection{An Oracle-based Approach for Designing Algorithms for IP} \label{sec:oracles}
In this section we develop a framework for designing algorithms for~\eqref{IP} which comprises five ingredients with the following roles:
\begin{description}
\item[Relaxation oracle] is used to obtain a point near an optimum of the fractional relaxation of~\eqref{IP}.
\item[Proximity bound] guarantees a nearby integer optimum and allows shrinking the bounds $\vel$ and $\veu$.
\item[Feasibility oracle] finds an initial feasible integer solution or declares infeasibility.
\item[Reducibility bound] concerns replacing $f$ with an equivalent objective $g$ with smaller values.
\item[Augmentation oracle] is used to augment the initial solution to optimality.
\end{description}
We will later prove a ``Master Lemma'' (Lemma~\ref{lem:stronglypoly_master}) which relates the ingredients and spells out exactly how improvements to individual ingredients improve the overall time complexity.
The main focus of the remainder of this section is then on examining closely these components, improving them, and in some cases showing that they cannot be improved further.
The order in which we will study these ingredients in Sections~\ref{sec:feasibility}--\ref{sec:almostlinear} is different from the one above because of the interdependencies of the proved statements.
Before proving the Master Lemma, we demonstrate the oracles and bounds in Section~\ref{sec:warmup} by constructing a strongly polynomial algorithm for ILP when $A$ is endowed with an~\eqref{AugIP} oracle.

\subsection{Warm-up: Linear Objectives} \label{sec:warmup}
As a warm-up for the Master Lemma and a specific demonstration of a situation where we already have all the components (either in this paper or in the existing literature), we shall prove that~\eqref{ILP} has a strongly-polynomial algorithm with numeric input $\vew, \veb, \vel, \veu$ whenever $A$ is endowed with an oracle solving~\eqref{AugIP}.
Note that we still treat the augmentation oracle as an oracle for the sake of generality, while we realize the rest of the ingredients.
By now the reader is aware of two realizations of an~\eqref{AugIP} oracle: Lemmas~\ref{lem:primal} and~\ref{lem:dual}.

\begin{theorem}\label{thm:oracle}
Problem~\eqref{ILP} with arithmetic input $\la A\ra$ and numeric input $\l \vew,\veb,\vel,\veu \r$, endowed with an oracle solving~\eqref{AugIP}, is solvable in strongly polynomial oracle time.
\end{theorem}
\begin{remark}
	The partition of the input to the arithmetic input $\l A \r$ and the numeric input $\l \vew, \veb, \vel, \veu \r$ is the same as in the classical results for linear programming~\cite{FT,Tar}.
\end{remark}
Together with Theorem~\ref{thm:main}, Theorem~\ref{thm:oracle} immediately yields that:
\begin{repcorollary}{cor:ilp_stronglypoly}
	There exists a computable function $g$ such that problem~\eqref{ILP} can be solved in time \[g(a,d) \poly(n), \qquad \text{where }d \df \min\{\td_P(A), \td_D(A)\} \enspace .\]
\end{repcorollary}
Let $\C(A)$ be the set of {\em circuits} of $A$, which are those
$\vecc\in\ker_{\Z}(A)$ whose support is a circuit of the linear matroid of $A$ and whose entries are coprime.
Let $c_\infty(A):=\max_{\vecc\in\C(A)}\|\vecc\|_\infty$.
It is known that $\C(A) \subseteq \G(A)$~\cite[Definition 3.1]{Onn} and thus $c_\infty(A) \leq g_\infty(A)$.
\begin{proposition}[{Onn~\cite[Lemma 2.17]{Onn}}] \label{prop:circuits}
	For any $\vex \in \ker(A)$, $\vex$ may be written as $\sum_{i=1}^{n'} \lambda_i \veg_i$ where $n' \leq n-r$ with $r \df \mathrm{rank}(A)$, and for all $i \in [n']$, $\lambda_i > 0$, $\veg_i \in \CC(A)$, and $\lambda_i\veg_i \sqsubseteq \vex$, i.e., the sum is sign-compatible.
\end{proposition}

\begin{proof}[Proof of Theorem~\ref{thm:oracle}]
The algorithm which demonstrates the theorem consists of several steps as follows.
\subsubsection*{Step 1: Relaxation oracle and proximity bound \emph{(i.e., reducing $\veb,\vel,\veu$.)}}
Apply the strongly polynomial algorithm of Tardos~\cite{Tar} to the linear programming relaxation $\min\left\{\vew\vey \mid \vey\in\R^n,\, A\vey=\veb, \,\vel\leq \vey\leq \veu\right\}$; the algorithm performs $\poly(\l A \r)$ arithmetic operations.
If the relaxation is infeasible then so is \eqref{ILP} and we are done.
If it is unbounded then \eqref{ILP} is either infeasible or unbounded too, and in this case we set $\vew:= \mathbf{0}$ so that all solutions are optimal, and we proceed as below and terminate at the end of step 3.
Suppose then that we obtain an optimal solution $\vey^*\in\R^n$ to the relaxation,
with round down $\lfloor \vey^*\rfloor\in\Z^n$. 
By Lemma~\ref{lem:bound1} we have $c_\infty(A)\leq g_\infty(A) \leq g_1(A) \leq  (2m \|A\|_\infty + 1)^m$.

We now use the proximity results of \cite{HKW,HS} (see Theorem~\ref{thm:proximity}) which assert that either \eqref{ILP} is
infeasible or it has an optimal solution $\vex^*$ with $\|\vex^*-\vey^*\|_\infty\leq nc_\infty(A)$ and hence $\|\vex^*-\lfloor \vey^*\rfloor\|_\infty < n(2m \|A\|_\infty + 1)^m + 1$, where the ``$+1$'' is due to the round-down of $\vey^*$.
Since both sides are integers, we have $\|\vex^*-\lfloor \vey^*\rfloor\|_\infty \leq n(2m \|A\|_\infty + 1)^m$.
Thus, making the variable transformation $\vex=\vez+\lfloor \vey^*\rfloor$, problem \eqref{ILP} reduces to the following,
\begin{equation*}
\min\,\left\{\vew(\vez+\lfloor \vey^*\rfloor) \mid \vez\in\Z^n\,,\ A(\vez+\lfloor \vey^*\rfloor)=\veb
\,,\ \vel\leq \vez+\lfloor \vey^*\rfloor\leq \veu\,,\ \|\vez\|_\infty\leq n(2m \|A\|_\infty + 1)^m\right\},
\end{equation*}
which is equivalent to the program
\begin{equation}\label{IP1}
\min\,\left\{\vew\vez \mid \vez\in\Z^n\,,\ A\vez=\bar \veb\,,\ \bar \vel\leq \vez\leq \bar \veu\right\}, \quad \text{where}
\end{equation}
\[\bar \veb:=\veb-A\lfloor \vey^*\rfloor,\ \
\bar l_i:=\max\left\{l_i-\lfloor y^*_i\rfloor,-n(2m \|A\|_\infty + 1)^m\right\},\ \
\bar u_i:=\min\left\{u_i-\lfloor y^*_i\rfloor,n(2m \|A\|_\infty + 1)^m\right\}\ \enspace .\]

If some $\bar l_i>\bar u_i$ then \eqref{IP1} is infeasible
and hence so is \eqref{ILP}, so we may assume that
\[-n(2m \|A\|_\infty + 1)^m\leq\bar l_i\leq\bar u_i\leq n(2m \|A\|_\infty + 1)^m, \quad \mbox{for all }i \in [n] \enspace .\]

This implies that for every point $\vez$ feasible for \eqref{IP1},
$\|A\vez\|_\infty\leq n^2\|A\|_\infty(2m \|A\|_\infty + 1)^m$ holds and so we may assume that
$\|\bar \veb\|_\infty\leq n^2\|A\|_\infty(2m \|A\|_\infty + 1)^m$ else there is no feasible solution.
By $m \leq n$ (see Proposition~\ref{prop:purification}) we have
\[\|\bar \veb, \bar \vel, \bar \veu\|_\infty \leq 2^{\Oh(n\log n)}\|A\|_\infty^{\Oh(n)}
\ \mbox{and hence}\ \l\bar \veb,\bar \vel,\bar \veu\r\ \mbox{is polynomial in}\ \l A\r\ .\]

\subsubsection*{Step 2: Feasibility oracle}
The next step is to find an integer solution to the system of equations $A \vez = \bar{\veb}$, and then to use this solution in an auxiliary problem with relaxed bounds to find an initial feasible solution to~\eqref{IP1}.
This is exactly the purpose of Lemma~\ref{lem:initial}.
Crucially, its bound on the number of calls to an \eqref{AugIP} oracle and the time to compute an integral solution of $A \vez = \bar{\veb}$ only depends on $\l A, \bar{\veb}, \bar{\vel}, \bar{\veu} \r$ and \emph{not} on the objective function $\vew$.
In the following subsections we will show that the general Lemma~\ref{lem:initial} may be sometimes replaced with a faster approach.

\subsubsection*{Step 3: Reducibility bound \emph{(i.e., reducing $\vew$)}.}
Let $N:=2n(2m\|A\|_\infty+1)^m$.
Now apply the strongly polynomial algorithm of Frank and Tardos~\cite{FT}, which on arithmetic input $n,\l N\r$ and numeric input $\l \vew\r$, outputs $\bar \vew\in\Z^n$ with $\|\bar \vew\|_\infty \leq 2^{\Oh(n^3)}N^{\Oh(n^2)}$ such that $\sign(\vew\vex)=\sign(\bar \vew\vex)$ for all $\vex\in\Z^n$ with $\|\vex\|_1<N$.
Since $\l N\r=1+\ceil{\log N}=\Oh(n\log n+ n\log \|A\|_\infty)$ is polynomial in $\l A\r$, this algorithm is also strongly polynomial in our original input. Now, for every two points $\vex,\vez$ feasible in \eqref{IP1} we have $\|\vex-\vez\|_1<2n(2m\|A\|_\infty+1)^m=N$, so that for
any two such points we have $\vew\vex\leq \vew\vez$ if and only if $\bar \vew\vex\leq \bar \vew\vez$,
and therefore we can replace~\eqref{IP1} by the equivalent program
\begin{equation}\label{IP3}
\min\,\left\{\bar \vew\vez\ \mid \vez\in\Z^n,\, A\vez=\bar \veb,\, \bar \vel\leq \vez\leq \bar \veu\right\}, \quad \text{where}
\end{equation}
\[
\|\bar \vew\|_\infty=2^{\Oh(n^3\log n)}\|A\|_\infty^{\Oh(n^3)}
\ \mbox{and hence}\ \l\bar \vew,\bar \veb,\bar \vel,\bar \veu\r\ \mbox{is polynomial in}\ \l A\r.
\]

\subsubsection*{Step 4: Augmentation oracle.}
Starting from the point $\vez$ which is feasible in \eqref{IP3} and using the \eqref{AugIP} oracle, we can solve program~\eqref{IP3} using Lemma~\ref{lem:lambda_oracle_initi} in polynomial time and in a number of arithmetic operations and oracle queries which is polynomial in $n$ and in $\log (f_{\max})$, which is bounded by
$\log\left(n\|\bar \vew\|_\infty\|\bar \veu-\bar \vel\|_\infty\right)$,
which is polynomial in $\l A\r$, and hence strongly polynomially.
\end{proof}

\subsection{Ingredient Definitions and the Master Lemma} \label{sec:master}
Now we shall construct a framework which brings to the fore individual components of the last proof.
Let us introduce the necessary notions, provide some examples, and prove the Master Lemma.
Generally we use the convention that the objects related to an oracle are denoted by its first letter written in the caligraphic font.
The oracle itself is denoted by the letter with the word ``oracle'' displayed over it, and the letter itself denotes a function bounding the time required to realize the oracle.
For example, a relaxation oracle is denoted $\RRR$ and the time it takes to solve the fractional relaxation of~\eqref{IP} to $\epsilon$-accuracy is $\RR(\II,\epsilon)$.

\subsubsection{Relaxation}
The first step in the proof of Theorem~\ref{thm:oracle} is to solve the relaxation of~\eqref{IP}:
\begin{equation} \label{relax}
  \min\left\{f(\vex) \, \mid A\vex=\veb,\, \vel\leq\vex\leq\veu,\, \vex\in\R^{n}\right\} \tag{P} \enspace .
\end{equation}
In the context of non-linear functions we run into the possibility of irrational optima.
Hochbaum and Shanthikumar~\cite[Section 1.2]{HS} argue in favor of the notion of an \emph{$\epsilon$-accurate} optimum, which is a solution of~\eqref{relax} close to some optimum in terms of distance, not necessarily objective value.
Moreover, they show that under reasonable assumptions on the objective such an optimum is also close in terms of objective value.
\begin{definition}[{$\epsilon$-accuracy~\cite{HS}}]
Let $\vex_\epsilon$ be a feasible solution of~\eqref{relax}.
We say that $\vex_\epsilon$ is an \emph{$\epsilon$-accurate solution} if there exists an optimum $\vex^*$ of~\eqref{relax} with $\|\vex^* - \vex_\epsilon\|_\infty \leq \epsilon$.
\end{definition}
\begin{definition}[Approximate relaxation oracle]
An \emph{approximate relaxation oracle $\RRR$} for a matrix $A$ is one that, queried on an instance of~\eqref{relax} with a constraint matrix $A$ and $\epsilon \in \R_{\geq 0}$, returns an $\epsilon$-accurate solution of~\eqref{relax}, or correctly reports that~\eqref{relax} is unbounded or infeasible.
\end{definition}

Such an oracle can be realized in weakly polynomial time by Chubanov's algorithm:
\begin{proposition}[Chubanov~\cite{Chubanov:2016}]
For each separable convex function $f$, an approximate relaxation oracle $\RRR$ for $A$ with error $\epsilon >0$ may be realized in $\poly(n, \la A, \veb, \vel, \veu, 1/\epsilon \ra)$ arithmetic operations.
\end{proposition}
\begin{remark}
A closer inspection of Chubanov's algorithm~\cite[Theorem 12]{Chubanov:2016} reveals that its dependence on $n$ is roughly $n^4 \log n + Tn\log n$ where $T$ is the time needed to solve an auxiliary linear program.
\end{remark}
While this is not a strongly polynomial algorithm, to the best of our knowledge this is the only algorithm which provides an approximate relaxation oracle for non-linear functions.
In particular, it is not clear whether the ellipsoid method can be used to return an $\epsilon$-accurate solution (instead of a solution approximating the objective).

The time complexity of current realizations of the relaxation oracle $\RRR$ dominates the overall dependence on $n$ for our algorithms for~\eqref{IP}.
We thus ask in which cases, particularly when $\td_P(A)$ and $\td_D(A)$ are small, can the dependence on $n$ be reduced?
Regarding strongly polynomial algorithms, the famous algorithm for linear programming of Tardos can be rephrased as
\begin{proposition}[Tardos~{\cite{Tar}}] \label{prop:Tardos}
For each $f(\vex) = \vew \vex$, an approximate relaxation oracle $\RRR$ for $A$ may be realized in $\poly(n, \l A \r)$ arithmetic operations, even with $\epsilon=0$.
\end{proposition}
Granot and Skorin-Kapov~\cite{GranotS:1990} partially extend Tardos' algorithm to quadratic programming:
\begin{proposition}[{Granot and Skorin-Kapov~\cite[Theorem 3.5]{GranotS:1990}}]
For each $f(\vex) = \frac{1}{2}\vex^\intercal D \vex + \vew \vex$, an approximate relaxation oracle $\RRR$ for $A$ may be realized in $\poly(n, \Delta)$ arithmetic operations even with $\epsilon=0$, where $\Delta$ is the maximal absolute subdeterminant of $(D,A^{\transpose}, -I)$.
\end{proposition}

\subsubsection{Proximity}
After (approximately) solving the relaxation~\eqref{relax} we would like to relate its solution to the optimum of~\eqref{IP} so as to reduce the bounds $\vel, \veu$ and subsequently the right hand side $\veb$.
\begin{definition}[Proximity bound]
Let $1 \leq p \leq +\infty$. We say that~$(A,f)$ has an \emph{$\ell_p$-proximity bound $\PP_p(A, f)$} if for any optimum $\vex^* \in \R^n$ of~\eqref{relax}, there exists an optimum $\vez^* \in \Z^n$ of~\eqref{IP} with
$$\|\vex^* - \vez^*\|_p \leq \PP_p(A, f) \enspace .$$
\end{definition}

\begin{example*}
In the proof of Theorem~\ref{thm:oracle} we have used the fact that for any separable convex function $f$, $\PP_\infty(A, f) \leq n g_\infty(A)$ due to Hemmecke, Köppe and Weismantel~\cite{HKW}.
\end{example*}
In Section~\ref{sec:proximity} we will extend the result of~\cite{HKW} to any $\ell_p$-norm.
Note that obviously any $p$-norm proximity bound implies an $\ell_\infty$-norm proximity bound, i.e., $\PP_\infty(A,f) \leq \PP_p(A,f)$ for any $1 \leq p < \infty$.

\subsubsection{Feasibility}
After obtaining an instance of~\eqref{IP} with reduced bounds and right hand side, we need to find an initial feasible integer solution $\vex_0$ in order to start the augmentation procedure.
\begin{definition}[Feasibility oracle]
  A \emph{feasibility oracle $\FFF$} for a matrix $A$ is one which, given $\veb, \vel, \veu$, either computes a feasible integer solution $\vex_0$ with $\la \vex_0 \ra \leq \poly(\la \veb \ra)$ or declares~\eqref{IP} infeasible.
\end{definition}

\begin{example*}
Lemma~\ref{lem:lambda_oracle_initi} shows that fast solvability of \eqref{AugIP} implies a feasibility oracle for $A$, however, with a super-quadratic overhead due to computing the Hermite normal form of $A$, which might dominate the overall time complexity.
\end{example*}
In Section~\ref{sec:feasibility} we show faster feasibility oracles for IPs with bounded $\td_P(A)$ and $\td_D(A)$.

\subsubsection{Reducibility of $f$}
With an initial solution $\vex_0$ at hand it remains to reduce the convergence dependence on the values of $f$.
To that end we seek to replace $f$ with an ``equivalent'' objective.
Actually, we show that because our algorithm only requires $f$ to be represented by a comparison oracle, it is sufficient to guarantee the \emph{existence} of a smaller but ``equivalent'' objective and there is no need to compute it.
Let us define the terms precisely:
\begin{definition}[Equivalent objective] \label{def:f_equivalence}
Let $f, g: \R^n \to \R$ be functions such that $\forall \vex \in \Z^n:\, f(\vex), g(\vex) \in \Z$ and let $\DD \subseteq \Z^n$.
We say that $g$ and $f$ are equivalent on $\DD$ if
\begin{equation} \label{eq:f_equivalence}
\forall \vex, \vey \in \DD: \, f(\vex) \geq f(\vey) \Leftrightarrow g(\vex) \geq g(\vey) \enspace .
\end{equation}
\end{definition}
Naming this notion an ``equivalence'' is justified by the following lemma:
\begin{lemma}
\label{lem:optimality}
Let $\II_f=(A,f,\veb,\vel,\veu)$ and $\II_g=(A,g,\veb,\vel,\veu)$ be instances of~\eqref{IP}, where $f$ and $g$ are equivalent on $[\vel,\veu]$.
 Then $\vex \in [\vel,\veu]$ is an optimum of $\II_f$ if and only if it is an optimum of $\II_g$.
\end{lemma}
\begin{proof}
By Definition~\ref{def:f_equivalence}, we have
\begin{alignat*}{3}
\vex \in [\vel,\veu] \text{ is optimal for } \II_f &\Leftrightarrow f(\vex) \leq f(\vez) \quad \forall \vez \in [\vel,\veu],\, A\vez=\veb \\
&\Leftrightarrow g(\vex) \leq g(\vez) \quad \forall \vez \in [\vel,\veu],\, A\vez=\veb &
\Leftrightarrow \vex \in [\vel,\veu] \text{ is optimal for } \II_g. \qedhere
\end{alignat*}
\end{proof}
\begin{definition}[$\rho$-reducibility] \label{def:reducibility}
Let $\rho: \N \to \N$. 
We say that a linear or separable convex function $f: \R^n \to \R$ is \emph{$\rho$-reducible} if, for every $N \in \N$, there exists a linear or separable convex function $g: [-N,N]^n \to \Z$, respectively, which is equivalent to $f$ on $[-N,N]^n$ and $g_{\max}^{[-N,N]^{n}} \leq \rho(N)$.
\end{definition}

The famous result of Frank and Tardos we used in the proof of Theorem~\ref{thm:oracle} can be rephrased as:
\begin{proposition}[{\cite{FT}}] \label{prop:FT}
Every linear function $f: \R^n \to \R$ is $\left(2^{n^3}N^{n^2}\right)$-reducible.
\end{proposition}

\begin{remark}
The definition of $\rho$-reducibility speaks of equivalence on the $[-N,N]^n$ box, but the fact that $[-N,N]^n$ is centered around the origin (or that it contains it at all) will be of no importance to us for the following reason.
Let $[\vel, \veu] \subseteq \Z^{n}$ be a box with largest side $N \df \|\veu-\vel\|_\infty$.
Our positive statements about reducibility (Theorem~\ref{thm:lin_red} and Corollary~\ref{cor:sepconvex_red}) hold for \emph{all} linear or separable convex functions, respectively.
Given this, notice that a separable convex function $f$ on $[\vel, \veu]$ is isomorphic to another separable convex function $f'$ on $[-N,N]^n$ (more precisely, on some $[\vel', \veu'] \subseteq [-N,N]^n$ with $u'_i - l'_i = u_i - l_i$), hence we are guaranteed a reduced $g$ on $[\vel, \veu]$ by the existence of a reduced $g'$ on $[-N,N]^n$.
We use the fact that separable convex functions are closed under translation.
For linear functions, we may just observe that any $f(\vex) = \vew \vex$ is already isomorphic to itself on $[\vel, \veu]$ and $[-N,N]^n$.
\end{remark}

\subsubsection{Augmentation}
\begin{definition}[Augmentation oracle]
	An \emph{augmentation oracle} $\AAA$ for an integer matrix $A$ is one that, queried on an instance of~\eqref{IP} whose constraint matrix is $A$, and a feasible solution $\vex$, either returns a feasible augmenting step $\veg \in \Z^n$, or correctly declares $\vex$ optimal.
\end{definition}

Let $\AAA$ be an augmentation oracle.
The \emph{$\AAA$-augmentation procedure} for~\eqref{IP} with a given feasible solution $\vex_0$ works as follows. Let $i \df 0$.
\begin{enumerate}
	\item \label{generic_step1}Query $\AAA$ on~\eqref{IP} and $\vex_i$.
	If $\AAA$ declares $\vex_i$ optimal, return it.
	\item Otherwise, $\AAA$ returns an augmenting step $\veh_i$ for $\vex_i$. Set $\vex_{i+1} := \vex_i + \veh_i$, $i \df i+1$, and go to \ref{generic_step1}.
\end{enumerate}
%

\begin{example*}
	We have seen that a halfling oracle, i.e., one which returns a halfling, converges in
	$3n \log \left(f_{\max}\right)$ iterations (Lemma~\ref{lem:halfling}) and can be realized by solving $(\log \|\veu - \vel\|_\infty) + 1$ instances of \eqref{AugIP}.
	A Graver-best oracle, i.e., one which returns a Graver-best step, converges in $1.5n \log f_{\max}$ steps.
	However, it is not in general known to be realizable faster than by solving \eqref{AugIP} for all possible $\lambda \in [-\|\veu-\vel\|_\infty, \|\veu-\vel\|_\infty]$, thus much slower than a halfling oracle.
\end{example*}

We denote by $\AA\left(\|\veu-\vel\|_\infty, f_{\max}^{[\vel,\veu]}\right)$ the time it takes to realize one call to $\AAA$, and denote by $\AAap\left(\|\veu-\vel\|_\infty, f_{\max}^{[\vel,\veu]}\right)$ the number of operations needed to realize the $\AAA$-augmentation procedure.
If $\sigma$ is the number of iterations in which $\AAA$ converges, then clearly $\AAap\left(\|\veu-\vel\|_\infty, f_{\max}^{[\vel,\veu]}\right) \leq \AA(\|\veu-\vel\|_\infty) \cdot \sigma$.
However, in Section~\ref{sec:almostlinear} we will show that, by the fact that $\AAA$ is always queried on a point $\vex_{i}$ obtained from $\vex_{i-1}$ based on $\AAA$'s answer (which is not arbitrary but under our control), the $\AAA$-augmentation procedure may be realizable significantly faster.

A neat property of the halfling augmentation procedure is that its convergence guarantee (Lemma~\ref{lem:halfling}) only depends on $f$ ``non-constructively''.
Specifically, for every function $g$ equivalent to $f$ on $[\vel, \veu]$ we can bound the number of iterations by $3n\log (g^{[\vel, \veu]}_{\max})$ instead of $3n\log (f^{[\vel, \veu]}_{\max})$.
We define this property formally now:

\begin{proposition}[Convergence with equivalent objective] \label{prop:converge_comparison}
	If an $\AAA$ only requires $f$ to be given by a comparison oracle, then for any~\eqref{IP} instance
	\[
	\AAap\left(\|\veu-\vel\|_\infty, f^{[\vel,\veu]}_{\max} \right) \leq \min_{\substack{g: \, \Z^n \to \Z\\ g \text{ is equivalent to $f$ on $[\vel,\veu]$}}} \AAap\left(\|\veu-\vel\|_\infty, g^{[\vel,\veu]}_{\max}\right) \enspace . \qedhere
	\]
\end{proposition}

\begin{example*}
	Since in all of the previous exposition we have only required $f$ to be given by a comparison oracle, Proposition~\ref{prop:converge_comparison} applies.
	In particular this is true of the halfling augmentation oracle for any separable convex function $f$, and in the proof of Lemma~\ref{lem:halfling}, $f$ may be replaced with any equivalent $g$, without $g$ being required for the execution of the algorithm.
\end{example*}

\subsubsection{Master Lemma}
\begin{definition}[An $(\AAA,\RRR,\PP_p,\FFF, \rho)$-equipped tuple]
Let $A \in \Z^{m \times n}$ be a matrix and $f: \R^n \to \R$ be a separable convex function.
We say that $(A,f)$ is \emph{$(\AAA,\RRR,\PP_p,\FFF, \rho)$-equipped} if the tuple $(\AAA,\RRR,\PP_p,\FFF, \rho)$ is such that, for any instance $\II=(A, f, \veb, \vel, \veu)$ of~\eqref{IP},
\begin{itemize}
\item $\AAA$ is an augmentation oracle which realizes the $\AAA$-augmentation procedure in time $\AAap(\|\veu-\vel\|_\infty, f_{\max}^{[\vel, \veu]})$ and accepts $f$ represented by a comparison oracle,
\item $\RRR$ is an approximate relaxation oracle which is realizable in $\mathcal{R}(\II, \epsilon)$ arithmetic operations,
\item $\mathcal{P}_p(A,f)$ for some $p \geq 1$ is a proximity bound for $(A,f)$,
\item $\FFF$ is a feasibility oracle realizable in $\mathcal{F}(\|\veu-\vel\|_\infty, \|\veb\|_\infty, \|A\|_\infty)$ arithmetic operations,
\item $\rho: \N \to \N$ is a function such that $f$ is $\rho$-reducible.
\end{itemize}
\end{definition}

\begin{lemma}[Master lemma] \label{lem:stronglypoly_master}
Let $\II=(A,f,\veb, \vel, \veu)$ be an instance of~\eqref{IP} such that $(A,f)$ is $(\AAA,\RRR,\PP_p,\FFF, \rho)$-equipped.
Let $N = 4\PP_p(A,f)$.
Then $\II$ is solvable in time
  \begin{numcases}{\min}
    \mathcal{R}(\II,N/4) + \mathcal{F}(N, N \|A\|_\infty n, \|A\|_\infty) + \AAap(N,\rho(N/2)) \label{lem:stronglypoly_master:1}\\
    \mathcal{F}(\|\veu-\vel\|_\infty, \|\veb\|_\infty, \|A\|_\infty) + \AAap\left(\|\veu-\vel\|_\infty, f_{\max}^{[\veu, \vel]}\right) \enspace . \label{lem:stronglypoly_master:2}
  \end{numcases}
\end{lemma}
Obviously, part~\eqref{lem:stronglypoly_master:2} of the Lemma is obtained simply by first calling $\FFF$ and then using the $\AAA$-augmentation procedure, yet we include this statement for future reference.
Moreover, part~\eqref{lem:stronglypoly_master:1} is useful not only for obtaining strongly polynomial algorithms, but also algorithms whose time complexity is weakly polynomial in the ``$\RR(\II,N/4)$'' part and strongly polynomial otherwise.

Before we proceed with the proof, let us demonstrate the lemma by deriving Theorem~\ref{thm:oracle}.
Let $A$ be an integer matrix and let $f(\vex) = \vew \vex$ be a linear function.
If~\eqref{AugIP} is solvable in $\AA(\|\veu - \vel\|_\infty)$ arithmetic operations, then $(A,f)$ form a well equipped tuple with the following parameters:
\begin{itemize}
\item $\AAap(\|\veu-\vel\|_\infty, f^{[\vel, \veu]}_{\max}) = 3n\log f^{[\vel, \veu]}_{\max} \cdot \AA(\|\veu-\vel\|_\infty)$ by Lemma~\ref{lem:halfling},
\item $\RR(A,f,\veb, \vel, \veu, \epsilon) = \RR(A,f,\veb, \vel, \veu, 0)= \poly(n, \l A \r)$ by Proposition~\ref{prop:Tardos},
\item $\PP_\infty(A,f) \leq g_\infty(A) \leq (2m\|A\|_\infty+1)^m$ by Theorem~\ref{thm:proximity},
\item $\FF(\|\veu - \vel\|_\infty, \|\veb\|_\infty, \|A\|_\infty) = 3n\AA(\|A,\veb,\veu-\vel\|_\infty) + \Oh(n^{\omega})$ by Lemma~\ref{lem:initial},
\item $\rho(N) \leq 2^{n^3} N^{n^2}$ by Proposition~\ref{prop:FT}.
\end{itemize}
Plugging in, we obtain
\begin{itemize}
\item $N = 4\PP_\infty(A,f) \leq 4(2m\|A\|_\infty+1)^m \leq 4(2n\|A\|_\infty+1)^n$ since $m \leq n$,
\item $\RR(\II,N/4) = \poly(n, \l A \r)$,
\item $\FF(N, N\|A\|_\infty n, \|A\|_\infty) = 3n \log (N) \AA(N) + \Oh(n^{\omega})$,
\item $\AAap(N, \rho(N/2)) \leq 3n \log(2^{n^3} N^{n^2}) \AA(N) = \Oh(n^4 + n^3 \log N) \AA(N)$.
\end{itemize}
and thus,
\begin{align*}
\RR(\II,N/4) + \FF(N, N\|A\|_\infty n) + \AAap(N,\rho(N/2)) \leq& \poly(n, \l A \r) + 3n\log (N) \AA(N) +  \\
&+\poly(\l n, N\|A\|_\infty n \r) + \Oh(n^4 + n^3 \log N)\AA(N).
\end{align*}
Since the encoding length of $N$ is polynomial in $n$, the whole time complexity is strongly polynomial.
\begin{proof}[Proof of Lemma~\ref{lem:stronglypoly_master}]
We generalize the proof of Theorem~\ref{thm:oracle}.
First, use the relaxation oracle $\RRR$ on $\II$ with $\epsilon=N/4=P_\infty(A,f)$, in time $\mathcal{R}(\II,N/4)$ obtaining an $N/4$-accurate solution $\vex_\epsilon$.
By triangle inequality, there is some integer optimum $\vez^*$ of~\eqref{IP} with $\|\vex_\epsilon - \vez^*\|_\infty \leq N/4 + \PP_\infty(A,f) = N/2$.

We proceed with the variable transformation
$$\bar{\veb} \df \veb - A \floor{\vex_\epsilon}, \,\, \bar{l}_i \df \max \left\{l_i-\floor{(x_\epsilon)_i}, -N/2 \right\}, \,\, \bar{u}_i \df \max \left\{u_i - \floor{(x_\epsilon)_i}, N/2\right\},$$
obtaining the reduced instance~\eqref{IP1}.
If some $\bar{l}_i > \bar{u}_i$ then~\eqref{IP1} is infeasible and so is~\eqref{IP}, so we may assume that $\|\bar{\veu}-\bar{\vel}\|_\infty \leq N$.
We also see that $\|\veb\|_\infty \leq n \|A\|_\infty N$ since any feasible solution $\vez$ of~\eqref{IP1} must satisfy $\|A\vez\|_\infty \leq 2n\|A\|_\infty (N/2)$ by the previous observation.

Now we apply the feasibility oracle $\FFF$ to the reduced instance~\eqref{IP1} which finds a solution $\vex_0$ of~\eqref{IP1} or declares it infeasible, meaning also~\eqref{IP} is infeasible, in time $\mathcal{F}(N, N \|A\|_\infty n, \|A\|_\infty)$.

Finally, $\|\bar{\veu}-\bar{\vel}\|_\infty \leq N$ implies $[\bar{\vel}, \bar{\veu}] \subseteq [-N/2, N/2]^n$ and by $\rho$-reducibility of $f$ there exists a function $g$ equivalent to $f$ on $[\bar{\vel}, \bar{\veu}]$ such that $g^{[\bar{\vel}, \bar{\veu}]}_{\max} \leq \rho(N/2)$.
The $\AAA$-augmentation procedure is then realizable in time $\AAap(N,\rho(N/2))$.
\end{proof}

\subsection{Feasibility Oracles} \label{sec:feasibility}
When working with matrices with small treedepth (or more generally pathwidth and treewidth), Gaussian elimination can be solved more efficiently, as was shown by Fomin et al.~\cite{FominLSPW:2018}.
\begin{proposition}[{Bounded treedepth purification~\cite[Theorem 1.2 ]{FominLSPW:2018}}]\label{prop:tw_pure}
Given $A \in \Z^{m \times n}$, a $\td$-decomposition $F$ of $G_P(A)$ or $G_D(A)$, and $\veb \in \Z^m$, in time $\Oh(\height(F)^2 (n+m))$ one can either declare $A\vex=\veb$ infeasible, or return a pure equivalent subsystem $A'\vex = \veb'$.
\end{proposition}
We note here that Fomin et al.~\cite{FominLSPW:2018} in fact prove a stronger statement which replaces a $\td$-decomposition with a path decomposition and replaces $G_P(A)$ and $G_D(A)$ with the incidence graph, whose treedepth (and also pathwidth and treewidth) are bounded from above by $\td_P(A)$ and $\td_D(A)$ (see Lemma~\ref{lem:inc_prim_tw}).
They also show that the more general tree decomposition may be used at the cost of increasing the dependence on the parameter from quadratic to cubic.
Because of Proposition~\ref{prop:tw_pure} we again assume that an~\eqref{IP} under consideration has pure $A\vex=\veb$, and do not explicitly account for the cost of purifying the system as it is dominated by the time complexity of other steps.
We also assume that the lower and upper bounds are finite, and discuss how to handle the case when they are not later in Section~\ref{sec:infinite_bounds}.

\medskip

Next, we show how to avoid the overhead of computing HNF in Lemma~\ref{lem:initial} for~\eqref{IP} instances with small $\td_P(A)$ or $\td_D(A)$.
We do this similarly to phase I of the simplex algorithm, via an auxiliary~\eqref{IP} instance with a constraint matrix $A_I \df (A~I) \in \Z^{m \times (n+m)}$.
We will use the notion of a centered instance:
\begin{definition}[Centered instance]
An~\eqref{IP} instance is \emph{centered} if $\vezero \in [\vel,\veu]$.
\end{definition}
Is is easy to see that any instance can be transformed into a centered instance by a simple translation, hence we will omit the proof of the following lemma.
For an~\eqref{IP} instance $\II$, let $\Sol(\II) \df \{\vex \in \Z^n \mid A\vex=\veb, \, \vel \leq \vex \leq \veu\}$ denote the set of feasible solutions of $\II$.
\begin{lemma}[Equivalent centered instance] \label{lem:centered}
Let an~\eqref{IP} instance $\II$ be given, and $\vev \in \Z^n$.
Define an~\eqref{IP} instance $\bar{\II}=(A,\bar{f},\bar{\veb}, \bar{\vel}, \bar{\veu})$ by
$$\bar{\veb} \df \veb - A\vev, \, \bar{\vel} \df \vel - \vev, \, \bar{\veu} \df \veu - \vev, \, \bar{f}(\vex) \df f(\vex -\vev) \enspace .$$
The translation $\tau (\vex) = \vex + \vev$ is a bijection from $\Sol(\bar{\II})$ to $\Sol(\II)$.
Moreover, $\vex$ is an optimal solution of $\bar{\II}$ if and only if $\tau(\vex)$ is an optimal solution of $\II$. 
If $\vev \in [\vel,\veu]$, then $\bar{\II}$ is a centered instance.\qedhere
\end{lemma}
For simplicity, we will usually choose $\vev = \vel$.
When beneficial, we will move to the centered instance $\bar{\II}$, recovering an optimum of $\II$ eventually.

\begin{lemma}[$A_I$ feasibility instance] \label{lem:feas_instance}
For a centered instance of~\eqref{IP}, define $\vel', \veu', \vew' \in \Z^m$ as
\begin{equation*}
w'_i \df \sign(b'_i), \, l'_i \df \min\{0,b_i\}, \, u'_i \df \max \{0, b_i\} \quad \forall i \in [m],
\end{equation*}
and let $\bar{\vel} \df (\vel, \vel')$, $\bar{\veu} \df (\veu, \veu')$.
The vector $(\vezero, \veb)$ is feasible for the instance
	\begin{equation}
\min \{\vew'\vex' \mid A_I(\vex, \vex') = \veb, \, (\vel, \vel') \leq (\vex, \vex') \leq (\veu, \veu'), \, (\vex, \vex') \in \Z^{n+m}\}, \tag{$A_I$-feas IP}\label{eq:auxiliary_feasibility}
\end{equation}
and this instance has an optimum of the form $(\vex, \vezero)$ (of value zero) if and only if~\eqref{IP} is feasible, and then $\vex$ is a feasible solution of~\eqref{IP}.
\end{lemma}
\begin{proof}
The vector $(\vezero, \veb)$ is feasible for~\eqref{eq:auxiliary_feasibility} by definition.
Moreover, a solution $(\vex, \vex')$ of~\eqref{eq:auxiliary_feasibility} is optimal with value $0$ if and only if $\vex' = \vezero$, in which case $\vex$ is feasible for~\eqref{IP}.
\end{proof}



\begin{lemma}\label{lem:feas_A_I}
	Let $A \in \Z^{m \times n}$, $\veb \in \Z^m$, $\vel,\veu \in \Z^n$ define an IP feasibility instance.
	Let $\AAA_I$ be an augmentation oracle for $A_I$ which, given $\bar{\veb} \in \Z^m$ and $\bar{\vel}, \bar{\veu} \in \Z^n$ and a linear objective $f(\vex) = \vew \vex$, realizes the $\AAA_I$-augmentation procedure in $\AAap_I\left(\|\bar{\veu}-\bar{\vel}\|_\infty, f_{\max}^{[\vel, \veu]}\right)$ arithmetic operations.
	Then a feasibility oracle for $A$ is realizable in
	$\AAap_I\left(\|\veu-\vel,\veb - A\vev\|_\infty,\|\veb - A \vev \|_1\right)$	
	arithmetic operations, where $\vev \in [\vel,\veu]$ is any vector.
	
	If the initial instance is centered, the feasibility oracle can be realized in $\AAap_I\left(\|\veu-\vel,\veb\|_\infty,\|\veb\|_1\right)$
	arithmetic operations.
\end{lemma}
\begin{proof}
	If~\eqref{IP} is not centered, translate the instance using Lemma~\ref{lem:centered} with $\tau(\vex) = \vex + \vel$.
	We obtain vectors $\bar{\veb} = \veb - A \vel$, $\bar{\vel} = \vezero$, $\bar{\veu} = \veu - \vel$.
	Then, solve the instance~\eqref{eq:auxiliary_feasibility} of Lemma~\ref{lem:feas_instance} by $\AAA_I$, reporting~\eqref{IP} as infeasible if the optimum has value more than $0$ and otherwise returning $\vez$ as its feasible solution.
	Observe that $\vew' \bar{\veb} = \|\bar{\veb}\|_1$ and for any $(\vex,\vex')$ feasible for~\eqref{eq:auxiliary_feasibility}, $\vew' \vex'$ is non-negative.
	Using the left-hand side bound of Lemma~\ref{lem:halfling} thus gives that the auxiliary instance~\eqref{eq:auxiliary_feasibility} is solved by $\AAA_I$ in time 
	$\AAap_I\left(\|\bar{\veu}-\bar{\vel},\bar{\veb}\|_\infty,\|\bar{\veb}\|_1\right) = 
	\AAap_I\left(\|\veu-\vel,\veb - A\vel\|_\infty,\|\veb - A \vel \|_1\right)$.
	For a centered instance, observe that $\bar{\veb} = \veb$.
\end{proof}


\begin{lemma} \label{lem:feas_td}
	Let $A \in \Z^{m \times n}$ and let $F_D$ be a $\td$-decomposition of $G_D(A)$. Then
	\begin{equation}
	G_D(A_I) = G_D(A),  \quad  \text{ in particular, } \td_D(A_I) = \td_D(A), \quad G_D(A_I) \subseteq \cl(F_D) \enspace . \label{eq:A_I_dual}
	\end{equation}
	
	Let an instance of~\eqref{IP} and a $\td$-decomposition $F_P$ of $G_P(A)$ be given, and assume $A\vex=\veb$ is pure.
	Then there is an algorithm which runs in time $\Oh(|F_P|)$ and computes rooted trees $F'_P$, $F''_P$, with $G_P(A_I) \subseteq \cl(F'_P)$ and $G_P(A_I) \subseteq \cl(F''_P)$, respectively, satisfying
	\begin{align}
	\height(F'_P) &\leq \height(F_P)+1, && \ttd(F'_P) \leq \ttd(F_P)+1, \label{eq:A_I_primal:1}\\
	 \height(F''_P) &\leq 2\height(F_P), && \ttd(F''_P) = \ttd(F) \enspace .\label{eq:A_I_primal:2} 
	\end{align}
\end{lemma}
\begin{proof}
	Concerning~\eqref{eq:A_I_dual}, by definition of $G_D(A)$ an edge corresponds to a column with at least two non-zero entries.
	Thus adding to any matrix any number of columns which have at most one non-zero entry does not change its dual graph.
	Hence $G_D(A) = G_D(A_I)$ and the rest follows.
	
	Let us now prove~\eqref{eq:A_I_primal:1}-\eqref{eq:A_I_primal:2}.
	For each leaf $i \in [n]$ of $F_P$, denote by $P(i)$ the path from the root of $F_P$ to $i$.
	Define $\ell(i) \df \left|\left\{A_{j,\bullet} \,\middle|\, \suppo(A_{j,\bullet}) \subseteq V(P(i))  \wedge \left(\forall i' < i\colon\suppo(A_{j,\bullet}) \not\subseteq V(P(i')) \right)  \right\}\right|$ to be the number of rows of $A$ whose support lies in $V(P(i))$ but not in $V(P(i'))$ for $i' < i$.
	For all non-leaf vertices define $\ell(i) \df 0$.
	If $\max_{i \in [n]} \ell(i) > \height(F_P)$ then $A$ contains a submatrix with $\height(F_P)$ columns and at least $\height(F_P)+1$ rows, which induces a subsystem of $A\vex=\veb$ which is not pure, contradicting purity of $A\vex=\veb$.
	Hence there are at most $\height(F_P)$ rows of $A$ whose support contains $V(P(i))$ for every leaf $i$ of $F_P$.
	
	To obtain $F'_P$, to each $i \in [n]$ attach $\ell(i)$ new leaves.
	To obtain $F''_P$, to each $i \in [n]$ attach a path on $\ell(i)$ new vertices.
	It is easy to verify that we have added $m$ new vertices in both cases, corresponding to the columns of $I$, in such a way that $G_P(A_I) \subseteq \cl(F'_P)$ and $G_P(A_I) \subseteq \cl(F''_P)$.
	
	Regarding $F'_P$, because we have attached new vertices only to leaves (for all other vertices $\ell(i)=0$), the height of $F'_P$ is at most one more than the height of $F_P$.
	Regarding $F''_P$, because $\ell(i) \leq \height(F_P)$, we have $\height(F''_P) \leq 2\height(F_P)$, and we have not increased the number of non-degenerate vertices on any path, so $\ttd(F''_P) = \ttd(F_P)$. (Specifically, on every root-leaf path we have made its old leaf degenerate, and the new leaf is a new non-degenerate vertex, while all other new vertices are degenerate.)
\end{proof}
In Section~\ref{sec:almostlinear} we will need a few more observations about the tree $F''_P$ in the previous Lemma.
First, consider again the proof of the Primal Decomposition Lemma (Lemma~\ref{lem:decomposition}).
The next proposition easily follows from it:
\begin{proposition}
Let $A$, $F$, $A_i$ and $F_i$, $i \in [d]$, be as in Definition~\ref{def:primal-decomp}.
Let $P \subseteq F$ be a path on $k_1(F)$ vertices containing the root of $F$.
Then $\{F_i\}_{i \in [d]}$ are the connected components of $F \setminus P$. \qedhere
\end{proposition}
Then, the following is easily derived from the proof of Lemma~\ref{lem:feas_td}:
\begin{proposition} \label{prop:AIstructure}
Let $A$, $F$, and $F''_P$ be as in Lemma~\ref{lem:feas_td}.
For any level $\ell \in [\ttd(F)-1]$, any block $E$ of the block decomposition of $A_I$ along $F''_P$ is of the form $(E'_\ell~I)$.
(A block at level $1$ is one of the blocks $A_i$; a block at level $2$ is a block of the decomposition of $A_i$ along $F_i$, and so on.)
\end{proposition}
\begin{proof}
By the construction of $F''_P$, on any root-leaf path, all vertices belonging to $[n]$ (corresponding to the columns of $A$) come before vertices belonging to $[n+1,n+m]$ (corresponding to the columns of $I$).
Thus, if $\ttd(F) \geq 2$, the vertices from $[n+1,n+m]$ are never contained in the path $P$ from the root to the first non-degenerate vertex, and do not become part of the $\bar{A}_i$ block.
\end{proof}

Combining Lemmas~\ref{lem:feas_A_I} and~\ref{lem:feas_td} then immediately yields that for matrices with bounded $\td_P$ and $\td_D$, ``feasibility is as easy as optimization'', where optimization is problem~\eqref{IP} when an initial solution is provided.
This means that it is sufficient to focus on improving algorithms realizing the augmentation procedure and any time complexity improvements there will translate to improvements for~\eqref{IP}.
\begin{corollary}[Feasibility as easy as optimization for $\td_P$ and $\td_D$] \label{cor:feas_td}
	If there is an algorithm which, given an instance of~\eqref{IP}, a $\td$-decomposition $F_P$ of $G_P(A)$ (or, $F_D$ of $G_D(A)$), and a feasible solution $\vex_0$ of~\eqref{IP}, solves the given instance in time 
	\begin{align*}
	&T\df T_P\left(k_1(F_P),\dots,k_{\ttd(F_P)}(F_P),n,\|A\|_\infty,\|\veu-\vel\|_\infty, \|\veb\|_\infty, f_{\max}^{[\vel,\veu]}\right), \\
	&\left(\text{or, } T\df T_D\left(k_1(F_D),\dots,k_{\ttd(F_D)}(F_D),n,\allowbreak\|A\|_\infty,\|\veu-\vel\|_\infty, \|\veb\|_\infty, f_{\max}^{[\vel,\veu]}\right)\text{, respectively}\right),
	\end{align*}
	then there is an algorithm which, given an instance of~\eqref{IP} and a $\td$-decomposition $F_P$ of $G_P(A)$ (or, $F_D$ of $G_D(A)$), solves it in time
	\begin{align*}
	&T_P\left(k_1(F_P),\dots,k_{\ttd(F_P)-1}, k_{\ttd(F_P)}(F_P)+\height(F_P),n+m,\|A\|_\infty, \|\veu-\vel,\veb\|_\infty, \|\veb\|_\infty, \|\veb\|_1\right)
	+ T, \\
	&\left(\text{or, }T_D\left(k_1(F_D),\dots,k_{\ttd(F_D)}(F_D),n+m,\|A\|_\infty, \|\veu-\vel,\veb\|_\infty, \|\veb\|_\infty, \|\veb\|_1\right) +T \text{, respectively}\right) \enspace.
	\end{align*}
	If $A$ has the form $A=(A'~I)$, then~\eqref{IP} can be solved in time
	\[
	2T_P\left(k_1(F_P),\dots,k_{\ttd(F_P)}(F_P),n,\|A\|_\infty, \|\veu-\vel,\veb\|_\infty, \|\veb\|_\infty, \max\left\{\|\veb\|_1,f_{\max}^{[\vel,\veu]}\right\}\right) \enspace .
	\]
\end{corollary}
\begin{proof}
We only argue the last part of the statement since the rest is clear.
If $A$ has the form $(A'~I)$, then we can solve feasibility using an auxiliary instance of the form~\eqref{eq:auxiliary_feasibility} where $A$ is replaced by $A'$, hence~\eqref{eq:auxiliary_feasibility} is itself an~\eqref{IP} instance again with a constraint matrix $A$.
Thus, we require time $T_P\left(k_1(F_P),\dots,k_{\ttd(F_P)}(F_P),n,\|A\|_\infty,\|\veu-\vel,\veb\|_\infty, \|\veb\|_\infty, \|\veb\|_1\right)$ to solve this instance and find an initial feasible solution, and then time $T_P\left(k_1(F_P),\dots,k_{\ttd(F_P)}(F_P),n,\|A\|_\infty,\|\veu-\vel\|_\infty, \|\veb\|_\infty, f_{\max}^{[\vel,\veu]}\right)$ to solve~\eqref{IP} to optimality.
The bound then follows from $\|\veu-\vel\|_\infty, \|\veb\|_\infty \leq \|\veu-\vel,\veb\|_\infty$ and $\|\veb\|_1, f_{\max}^{[\vel,\veu]} \leq \max\{\|\veb\|_1, f_{\max}^{[\vel,\veu]}\}$.
\end{proof}

\subsection{Handling Infinite Bounds} \label{sec:infinite_bounds}
So far we have assumed that the bounds $\vel, \veu$ are finite.
In the following we will discuss several aspects of when and how this can (and cannot) be attained.
As the situation differs rather drastically for linear and non-linear objectives, we treat these cases separately. First, we will discuss the feasibility problem, as the objective only comes into play once we have an initial feasible solution.
\begin{lemma}[\cite{JansenLR:2018}]
\label{lem:feas-inf-bounds-bounded}
Any feasible instance of~\eqref{IP} has a feasible solution satisfying $\norm{\vex_{0}}_{1} \leq \norm{\veb - A \vev}_{1} g_1 (A_I)$, for any finite $\vev \in [\vel,\veu]$.
\end{lemma}
\begin{proof}
If necessary, we change to a centered instance by setting $\veb' \df \veb - A \vev$, and $\vel' \df \vel - \vev$, $\veu' \df \veu - \vev$.
Then, we set up the auxiliary IP according to Lemma~\ref{lem:feas_A_I} with objective $\vew$.
Decomposing an optimum $\binom{\vex_{0}}{\ve{0}} - \binom{\ve{0}}{\veb'} = \sum_{i=1}^{2n} \alpha_i \veg_i$ for some $\veg_i \in \G(A_I)$ conformal, we may assume that $\vew \veg_i < 1$; otherwise, omitting $\alpha_i \veg_i$ leads to an optimum of smaller $\ell_1$-norm.
But this implies that $\sum_{i=1}^{2n} \alpha_i \leq \norm{\veb^\prime}_1$, in turn showing that $\norm{\vex_0}_1 \leq \norm{\veb^\prime}_1 \cdot g_1(A_I)$.
%
\end{proof}
Now we can simply use artificial bounds for solving feasibility.
\begin{corollary}[\cite{JansenLR:2018}]
Let $A \in \Z^{m \times n},\veb \in \Z^m, \vel,\veu \in (\Z \cup {\pm \infty})^n$ define an IP feasibility instance.
Let $\AAA_I$ be an augmentation oracle for $A_I$ which, given $\bar{\veb} \in \Z^m$ and $\bar{\vel}, \bar{\veu} \in \Z^n$ and a linear objective $f(\vex) = \vew \vex$, realizes the $\AAA_I$-augmentation procedure in $\AAap_I\left(\|\bar{\veu}-\bar{\vel}\|_\infty, f_{\max}^{[\bar{\vel}, \bar{\veu}]}\right)$ arithmetic operations.
	Then a feasibility oracle for $A$ with possibly infinite bounds is realizable in	
	\[
	\AAap_I\left( \norm{\veb - A \vev}_1 g_1(A_I),\|\veb - A \vev \|_1\right)
	\]	
	arithmetic operations, where $\vev \in [\vel,\veu]$ is any vector.
	
	If the initial instance is centered, the feasibility oracle can be realized in
	\[
	\AAap_I\left( \norm{\veb}_1 g_1(A_I) ,\|\veb \|_1\right)
	\]
	arithmetic operations.
\end{corollary}

\subsubsection{Linear objectives}
The most standard argument when $f(\vex)$ is a linear function is the following: The relaxation of~\eqref{IP} can be solved in polynomial time, and we have shown several realizations of a feasibility oracle, e.g., Lemmas~\ref{lem:initial} and~\ref{lem:feas_A_I}.
If the IP is feasible and its relaxation is unbounded, there is a rational augmenting step that can be added arbitrarily often.
By scaling this step to integrality, also the IP is unbounded.
Otherwise, a proximity result provides finite bounds of polynomial length which contain some optimum~(see~\cite[Section 6.2]{GLS}).

If we want to avoid solving the relaxation, Jansen et al.~\cite{JansenLR:2018} gave an elegant approach.
First, they argue that computing a single augmenting step is sufficient for deciding unboundedness, given any feasible point $\vex$.
If the IP is bounded, any augmenting step $\veg$ must have a coordinate $i$ such that either $0 < g_i \leq u_i - x_i < \infty$, or $0 > g_i \geq \ell_i - x_i > - \infty$.
Otherwise, we could apply this augmenting step arbitrarily often.
Vice versa, if the IP is unbounded, there is an unbounded improving direction, and due to sign compatibility, an unbounded augmenting step.
Hence, solving the augmentation IP with altered bounds $\vel_i = 0$ ($\veu_i = 0$) whenever $\vel_i$ is finite ($\veu_i$ is finite) lets us conclude whether the IP is unbounded.

If it is bounded, we could proceed as in the case of finite bounds. However, for the halfling procedure, we guess step-lengths $2^k$, $k \in \Z_{\geq 0}$.
But how do we know where to stop? The natural argument would be to guess only $2^k \leq \max \{|x_i - \ell_i|: \ \ell_i > - \infty \} \cup \{|u_i - x_i|: \ u_i < \infty\}$, but this bound is depending on the current point $\vex$, and hence might even increase during the algorithm.
Instead, Jansen et al.\ argue that if the IP is bounded, there has to be an optimal solution with bounded encoding length.
This allows us to introduce artificial bounds similar to the proximity approach, but depending on the encoding size of $\veb,\vel,\veu$.
Let $\vel_{\fin}$ and $\veu_{\fin}$ denote the restriction of $\vel$ and $\veu$ to their finite entries, respectively.

\begin{lemma}[{Infinite bounds and linear objectives~\cite{JansenLR:2018}}] \label{lem:infinite_linear}
	Given an initial solution $\vex_0$ of~\eqref{IP} with a linear objective function $f(\vex) = \vew \vex$, it is possible to decide unboundedness of~\eqref{IP} by solving~\eqref{AugIP} once.
	If the IP is bounded, there exists an optimal solution satisfying $\norm{\vex^* - \vex_0}_1 \leq (n \norm{\vel_{\fin}, \veu_{\fin}}_\infty + \norm{\vex_0}_1) g_1(A)$.
	Moreover, if $\vex_0$ is the initial solution of Lemma~\ref{lem:feas-inf-bounds-bounded} and we can optimize over finite bounds with
	$\AAap (\norm{\veu - \vel}_\infty,\vew (\vex_0 - \vex^*))$ arithmetic operations,
	then we can optimize over possibly infinite bounds with
	$\AAap ( 2 n \norm{\vel_{\fin},\veu_{\fin},\veb}_\infty  \norm{A}_\infty g_1(A_I)g_1(A) ,\vew (\vex_0 - \vex))$ arithmetic operations.
\end{lemma}
\begin{proof}
	To decide unboundedness, solve~\eqref{AugIP} with input $A, f, \vex_0$, $\lambda\df 1$, and with auxiliary bounds $\bar{\vel}, \bar{\veu}$, defined, for each $i \in [n]$, as follows:
	\[
	\bar{l}_i \df 
	\begin{cases} 
	0 & \text{if } l_i > -\infty \\
	-N       & \text{if } l_i = -\infty
	\end{cases} \qquad
	\bar{u}_i \df 
	\begin{cases} 
	0 & \text{if } u_i < +\infty \\
	N       & \text{if } u_i = +\infty
	\end{cases},
	\]
	where $N \geq g_\infty(A)$, for example take $N \df (2m \|A\|_\infty +1)^m$ by Lemma~\ref{lem:bound1}.
	Observe that any solution $\veg$ of this~\eqref{AugIP} instance is an unbounded augmenting direction.
	Hence if there is a solution $\veg$ with $\vew \veg < 0$, we report that~\eqref{IP} is unbounded.

	Otherwise, let $\zeta \df \norm{\vel_{\fin}, \veu_{\fin}}_\infty$, and let $\vex^*$ be an optimum solution closest to $\vex_{0}$.
	We will show that $\vex^*$ has bounded norm, allowing us to replace the infinite bounds.
	
	To this end decompose $\vex^* - \vex_{0} = \sum_{i=1}^{2n} \alpha_i \veg_i$ for some conformal Graver basis elements with negative objective value.
	For each Graver basis element $\veg_i$ however, there has to be one component with either $0 < (\veg_i)_j \leq u_j -(\vex_{0})_j < \infty$, or $0 > (\veg_i)_j \geq l_j -(\vex_{0})_j > - \infty$.
	Otherwise, we found an unbounded augmenting direction.
	This implies that $\sum_{i=1}^{2n} \alpha_i \leq n \zeta + \norm{\vex_{0}}_1$.
	In total, we have that $\norm{\vex^* - \vex_{0}}_1 \leq  (n \zeta + \norm{\vex_{0}}_1) g_1(A)$.
	Choosing the initial solution of Lemma~\ref{lem:feas-inf-bounds-bounded}, we can replace the infinite bounds by 
	\[
	(n \zeta + (\norm{b}_1 + \norm{A}_\infty n \zeta)g_1(A_I)) g_1(A) \leq (2n\zeta \norm{A}_\infty) g_1(A_I)g_1(A).
	\]
\end{proof}

\subsubsection{Separable convex objectives}
The situation is quite different for separable convex functions.
In the most general case, the problem is undecidable (also see~Onn~\cite[Section 1.3.3]{Onn}).
\begin{proposition}[Unbounded convex minimization undecidable]
	For a convex function $f: \R \to \R$ given by an evaluation oracle, the question whether $f$ has a finite minimum is undecidable.
\end{proposition}
\begin{proof}
	Assume for contradiction there is an algorithm deciding whether $f$ has a finite minimum.
	Then we may construct an adversary function $f$ as follows.
	As long as the algorithm is querying the evaluation oracle, we respond with the function $-x$ (i.e., on input $x \in \R$ we return $-x$).
	At some point the algorithm must terminate with an answer.
	If its answer is ``$f$ has a finite minimum'', then we claim that actually $f(x) = -x$ and hence the answer is incorrect.
	On the other hand, if the answer is ``$f$ has no finite minimum'', we let $x_{\max}$ be the largest $x$ on which $f$ was evaluated in the course of the run of the algorithm, and claim that in fact $f$ is a two-piece linear function with
	\[
	f(x) \df 
	\begin{cases} 
	-x & \text{if } x \leq x_{\max} \\
	x - x_{\max}       & x > x_{\max}
	\end{cases}, \\
	\]
	so clearly $-x_{\max}$ is a finite minimum of $f$ and the answer was incorrect.
\end{proof}
Even when $f$ is minimized over a finite interval $[\ell,u]$ but only presented by a comparison oracle, it is not possible to find its minimum in less than $\log |u-\ell|$ steps (by binary search) from basic information theory lower bounds, see~Hochbaum~\cite[Section 3.1]{Hochbaum94}.
In the more powerful algebraic-tree model there is a lower bound of $\log \log |u-\ell|$~\cite[Section 3.2]{Hochbaum94}.

Again, the problem is how to limit the step-lengths we choose for the halfling procedure (Lemma~\ref{lem:halfling}).

A natural modification of this algorithm is to solve~\eqref{AugIP} for $\vex$ and for increasing $\lambda \in \Z_{\geq 0}$ as long as the returned solution $\veg$ has $f(\vex + 2^\lambda \veg) < f(\vex)$.
	However, the objective may become constant at a certain point, meaning that at a certain point, every step-length yields the same improvement and remains feasible.
	As a criterion to stop increasing $\lambda$, we could check whether $\lambda$ and $\lambda+1$ yield the same improvement.
	But one can imagine that some step $\veg$ yields a lot improvement, and for $\lambda \geq 1$ the step $2^\lambda \veg$ yields the same (i.e.\ $f$ becomes constant in that direction), but another step $2^\lambda \veg^\prime$ gives a lot more improvement, but only for $\lambda$ large enough.
	If we stopped at step-length $2$ already, we would not have found $g^\prime$.
	
	For these reasons, it is unclear what is a good selection for an abortion criterion.
    However, simply guessing a bound on $\|\vex_0 - \vex^*\|_\infty$ yields the following.

\begin{lemma}[No a priori bounds] \label{lem:infinite_sepconv}
	Let an instance of~\eqref{IP} with $A \in \Z^{m \times n}$, $\vel, \veu \in (\Z \cup \{\pm\infty\})^n$  be given, and let $\vex_0$ be a feasible solution and $\vex^*$ be any finite optimum.
	If we can optimize over finite  bounds in $\AAap(\norm{\veu-\vel}_\infty, f(\vex_0) - f(\vex^*))$ arithmetic operations,
	then the given instance can be solved in
	$\log(2\norm{\vex^* - \vex_0}_\infty + 1) \AAap (2 \norm{\vex^* - \vex_0}_\infty, f(\vex_0) - f(\vex^*))$ arithmetic operations.
\end{lemma}
\begin{proof}
	In the proof of Lemma~\ref{lem:halfling} we describe that a halfling for $\vex$ can be found by solving~\eqref{AugIP} with step-lengths 
	$2^\lambda$, $0 \leq \lambda \leq \ceil{\|\veu-\vel\|_\infty}$ and then picking the best solution.

	If the bounds are not finite, we ``guess'' a bound $\norm{\vex^* - \vex_0}_\infty \leq B$, and then run the algorithm where we replace each variable bound that is too large by a corresponding term w.r.t.\ $B$.
	For the new instance, we can find a halfling by solving~\eqref{AugIP} with step-lengths $2^{\lambda}$, $0 \leq \lambda \leq \log(B)$.
	If we guessed correctly, after at most 
	$\AAap(B, f(x_0) - f(x^\star))$
	arithmetic operations
	we terminate with an optimal solution for the new instance.
	If our guess was too small, we can still perform an augmenting step at this point;
 instead of continuing, we restart the whole algorithm again at $\vex_0$, but with $2B$ as a guess.
	With this approach, the iterative guessing yields an additional factor of $\log (2\norm{\vex^* - \vex_0}_\infty)$.
\end{proof}

By the lower bound result of Hochbaum~\cite{Hochbaum94}, the $\log (2\|\vex^* - \vex_0\|_\infty)$ factor of Lemma~\ref{lem:infinite_sepconv} cannot be improved if $f$ is given by a comparison oracle, and cannot be improved below $\Omega(\log \log \|\vex^* - \vex_0\|_\infty)$ even in the stronger algebraic-tree model.
\lv{Whether this factor can be attained is an interesting open problem.}

\subsection{Proximity Bounds} \label{sec:proximity}
Here we focus on proximity results.
First, we show that a careful analysis of a proof of Hemmecke, Köppe and Weismantel~\cite{HKW} allows us to extend their theorem to additionally provide an $\ell_1$-norm bound.
Second, in the spirit of Hochbaum and Shanthikumar~\cite{HS}, we show that also for each integer optimum there is a continuous optimum nearby (Theorem~\ref{thm:proximity}).
Third, we turn our attention to a proximity theorem relating the solutions of an instance of~\eqref{IP} and an instance obtained from~\eqref{IP} by ``scaling down'' (Theorem~\ref{thm:scaling_proximity}).
Using this result we obtain a proximity-scaling algorithm for~\eqref{IP} which works by solving a sequence of instances with small bounds (Corollary~\ref{cor:scaling_algo}).
This algorithm is later useful in obtaining a nearly-linear time algorithm for~\eqref{IP} with small $\td_P(A)$ (Theorem~\ref{thm:nearlylinear_primal}).

\subsubsection{Basic Proximity Theorem}

\begin{theorem}[Basic Proximity] \label{thm:proximity}
  Let $\hat{\vex}$ be an optimum of~\eqref{relax} and $\hat{\vez}$ be an optimum of~\eqref{IP}. 
There exist $\vex^* \in \R^n$ and $\vez^* \in \Z^n$ optima of~\eqref{relax} and~\eqref{IP}, respectively, such that
\begin{align*}
\|\hat{\vex} - \vez^*\|_\infty = \|\vex^* - \hat{\vez}\|_\infty \leq n g_\infty(A), & \qquad &\|\hat{\vex} - \vez^*\|_1 = \|\vex^* - \hat{\vez}\|_1 \leq n g_1(A) \enspace .
\end{align*}
\end{theorem}
We will need a small proposition which follows from Proposition~\ref{prop:superadditivity}:
\begin{proposition} \label{prop:prox_superadd}
	Let $\vex, \vey_1, \vey_2 \in \R^n$, $\vey_1, \vey_2$ be from the same orthant, and $f$ be a separable convex function.
	Then
	\[
	f(\vex + \vey_1 + \vey_2) - f(\vex+\vey_1) \geq f(\vex+\vey_2) - f(\vex) \enspace .
	\]
\end{proposition}
\begin{proof}
	Apply Proposition~\ref{prop:superadditivity} with $\vex \df \vex$, $\veg_1 \df \vey_1$, $\veg_2 \df \vey_2$, and $\lambda_1, \lambda_2 \df 1$, to get
	\[
	f(\vex + \vey_1 + \vey_2) - f(\vex) \geq \left(f(\vex + \vey_1) - f(\vex)\right) + \left(f(\vex + \vey_2) - f(\vex)\right) = f(\vex + \vey_1) + f(\vex + \vey_2) - 2f(\vex) \enspace .
	\]
	Adding $2f(\vex)$ to both sides and rearranging then yields the statement.
\end{proof}
\begin{proof}[Proof of Theorem~\ref{thm:proximity}]
By Proposition~\ref{prop:circuits}, we may write $\hat{\vex} - \hat{\vez}$ as a sign-compatible sum $\sum_{i=1}^{n'} \lambda_i \veg_i$ where $n' \leq n-r$ with $r=\textrm{rank}(A)$, and, for all $i \in [n']$, $\veg_i \in \CC(A) \subseteq \G(A)$, $\lambda_i \in \R_{>0}$, and $\lambda_i\veg_i \sqsubseteq \hat{\vex} - \hat{\vez}$.
Write $\hat{\vex} - \hat{\vez} = \sum_{i=1}^{n'} \floor{\lambda_i} \veg_i + \sum_{i=1}^{n'} \{\lambda_i\} \veg_i$, where $\{\lambda\} \df \lambda - \floor{\lambda}$ denotes the fractional part of $\lambda$.
Now define
\begin{align*}
\vex^* &\df \hat{\vez} + \sum_{i=1}^{n'} \{\lambda_i\} \veg_i, & \vez^* &\df \hat{\vez} + \sum_{i=1}^{n'} \floor{\lambda_i} \veg_i = \hat{\vex} - \sum_{i=1}^{n'} \{\lambda_i\} \veg_i \enspace .
\end{align*}
By the fact that both $\hat{\vez}$ and $\hat{\vex}$ lie within the bounds $\vel$ and $\veu$ and that both $\sum_{i=1}^{n'} \{\lambda_i\} \veg_i$ and $\sum_{i=1}^{n'} \floor{\lambda_i} \veg_i$ are conformal to $\hat{\vex} - \hat{\vez}$, we see that both $\vex^*$ and $\vez^*$ also lie within the bounds $\vel$ and $\veu$.
Thus $\vex^*$ is a feasible solution of~\eqref{relax} and $\vez^*$ is a feasible solution of~\eqref{IP}.
We can also write
$$\hat{\vex} - \hat{\vez} = (\vex^* - \hat{\vez}) + (\vez^* - \hat{\vez}),$$
which, using Proposition~\ref{prop:prox_superadd}, gives
\[f(\hat{\vex}) - f(\vex^*) \geq f(\vez^*) - f(\hat{\vez}) \enspace .\]
Since $\hat{\vex}$ is a continuous optimum and $\vex^*$ is a feasible solution to~\eqref{relax}, the left hand side is non-positive, and so is $f(\vez^*) - f(\hat{\vez})$.
But since $\hat{\vez}$ is an integer optimum it must be that $\vez^*$ is another integer optimum and thus $f(\hat{\vez}) = f(\vez^*)$, and subsequently $f(\vex^*) = f(\hat{\vex})$ and thus $\vex^*$ is another continuous optimum.
Let us now compute the proximity.
Since $\hat{\vex} - \vez^* = \sum_{i=1}^{n'} \floor{\lambda_i} \veg_i = \vex^* - \hat{\vez}$, it is sufficient to bound the distance between just one of the pairs $(\hat{\vex}, \vez^*)$ and $(\vex^*, \hat{\vez})$.
The following derivation is invariant under the norm bound on the elements of $\G(A)$, so, in particular, holds for both the $\ell_\infty$- and $\ell_1$-norms:
\[
\|\hat{\vex} - \vez^*\| =\|\sum_{i=1}^{n'} \{\lambda_i\} \veg_i\|
\leq n'\max_{i=1,\dots,n} \|\veg_i\|
\leq n \max_{\veg \in \G(A)} \|\veg\|,
\]
where the first equality follows by definition of $\vez^*$.
When considering the $\ell_\infty$-norm we use in the last line $g_\infty(A)$ and for the $\ell_1$-norm we use in $g_1(A)$, and the claim clearly follows.
\end{proof}

\subsubsection{Scaling and Proximity}
If we want to avoid solving the continuous relaxation of~\eqref{IP}, but still have large box constraints $\norm{\veu-\vel}_\infty$, it can already be helpful to reduce the dependency on $\vel,\veu$ of the algorithm.
For instance, we might expect that $f_{\max } \geq f(\vex_0) - f(\vex^*)$ is at least linear in $\norm{\veu-\vel}_\infty$ for a separable convex function.
The idea is to first optimize with rather large step lengths $\lambda$ only i.e.\ updating $\vex \leftarrow \vex + \lambda \vey$ with $\lambda \geq 2^k$.
This practically replaces $\vel,\veu$ by smaller bounds $\tfrac{\vel}{\lambda}, \tfrac{\veu}{\lambda}$, and we can apply the results of Section~\ref{sec:reducibility} for limiting the objective function values.
Iteratively refining the solution leads to an optimal solution eventually.
The motivation for these results was taken from the techniques in~\cite{HS}.

We lay out the approach more specifically.
Assume for now that $\mathbf{0}$ is a feasible solution, i.e.\ we are concerned with the following IP.
\begin{align}
\label{eq:s-ip-original}
	\min & \left\{ f(\vex) : \ A\vex = \mathbf{0}, \ \vel \leq \vex \leq \veu, \ \vex \in \Z^n \right\}.
\end{align}
Instead of looking for a solution in the lattice $\vex \in \Z^n$, we first look for a solution in the scaled lattice $\vez \in s \Z^n = \{s\vex \mid \vex \in \Z^n\}$, for some $s \in \Z_{\geq 1}$.
\begin{align}
\label{eq:s-ip-lattice}
	\min & \left\{ f(\vex) : \ A\vex = \mathbf{0}, \ \vel \leq \vex \leq \veu, \ \vex \in s \Z^n \right\}. \tag{$s$-scaled IP}
\end{align}
Observe that both systems have the same continuous relaxation.
Hence, an optimal solution $\vex^*$ to the continuous relaxation of~\eqref{eq:s-ip-original} is also an optimal solution to the continuous relaxation of~\eqref{eq:s-ip-lattice}.
With this, we can relate an optimal solution $\vez^*$ of~\eqref{eq:s-ip-original} with an optimal solution $\hat{\vez}$ of~\eqref{eq:s-ip-lattice}.
\begin{theorem}[Scaling proximity] \label{thm:scaling_proximity}
Let an IP~\eqref{eq:s-ip-original} be given, and $s \in \Z_{\geq 1}$.
For every optimal solution $\vez^*$ of~\eqref{eq:s-ip-original}, there exists an optimum solution $\hat{\vez}$ of the scaled IP~\eqref{eq:s-ip-lattice} such that
\begin{align*}
\norm{\vez^* -\hat{\vez}}_1 \leq (s+1) n g_1(A), &\quad \norm{\vez^* -\hat{\vez}}_{\infty} \leq (s+1) n g_\infty(A).
\end{align*}
Vice versa, for every $\hat{\vez}$, there exists a $\vez^*$ within the same distance.
\end{theorem}
\begin{proof}
We will show one direction for the $\ell_1$-norm, as the remaining statements are shown in the same way.
Let $\vez^*$ be an optimum solution of~\eqref{eq:s-ip-original}.
By proximity, there exists an optimum solution $\vex^*$ to the continuous relaxation with
\begin{align}
\label{eq:aux2}
\norm{\vez^* - \vex^*}_1 &\leq n g_1(A).
\end{align}
As the continuous relaxations coincide, $\vex^*$ is also an optimum to the continuous relaxation of~\eqref{eq:s-ip-lattice}.
Substituting $\vex^\prime \df \tfrac{1}{s} \vex$, 
we obtain an objective-value preserving bijection between the solutions of~\eqref{eq:s-ip-lattice} and the solutions of the \emph{$s$-scaled IP}
\begin{align}
\label{eq:s-ip-scaled}
	\min & \left\{ f(s \vex^\prime ) : \ A\vex^\prime = \mathbf{0}, \ \tfrac{\vel}{s} \leq \vex^\prime \leq \tfrac{\veu}{s}, \ \vex^\prime \in \Z^n \right\}.
\end{align}
In particular, $\bar{\vex} \df \tfrac{1}{s} \vex^*$ is an optimum solution to the continuous relaxation of~\eqref{eq:s-ip-scaled}.

Again by proximity, the instance~\eqref{eq:s-ip-scaled} has an optimal solution $\bar{\vez}$ with 
\begin{alignat}{3}
\allowdisplaybreaks
\label{eq:aux3}
\norm{\bar{\vex} - \bar{\vez}}_1 &\leq n g_1(A) & \quad &\Leftrightarrow \quad & \norm{s\bar{\vex} - s\bar{\vez}}_1 &\leq sn g_1(A)
\end{alignat}
Substituting back, $\hat{\vez} \df s \bar{\vez}$ is an optimal solution to~\eqref{eq:s-ip-lattice}.
The claim follows by triangle inequality.
\end{proof}

\paragraph*{Scaling algorithm.}
As a $2^{k}$-scaled instance is also a $2$-scaled instance of a $2^{k-1}$-scaled instance, this gives the following idea for an algorithm to solve a general~\eqref{IP}.
\begin{enumerate}
\item[$(1)$] Find an initial feasible solution $\vex_{0}$, and recenter the instance to~\eqref{eq:s-ip-original}, so that $\mathbf{0}$ is feasible.
Let $k = \log_2 (\max(\norm{\veu}_\infty, \norm{\vel}_\infty)) + 1$, and set $\vex_k \df 0$.
As it is the only feasible solution, $\vex_k$ is an optimal solution of the $2^k$-scaled instance.
\item[$(2)$] If $k > 0$, replace the box constraints of the $2^{k-1}$-scaled instance with the proximity bounds $\norm{\vex_{k} - \vex}_\infty \leq 3ng_1(A)$.
Solve this instance with initial solution $\vex_{k}$ to optimality, and let $\vex_{k-1}$ denote the optimum found.
Update $k \leftarrow k-1$ and repeat.
\item[$(3)$] Output $\vex_0$.
\end{enumerate}
Analyzing the running time, we obtain the following corollary.
\begin{corollary}
\label{cor:scaling_algo}
Given an instance of~\eqref{IP} with finite bounds, we can find an optimum solution by solving $2 \log \norm{\veu - \vel, \veb}_{\infty}$ instances of~\eqref{IP} with right-hand side $\veb = \vezero$, and the lower and upper bounds $\bar{\vel}, \bar{\veu}$ of any instance satisfy $\norm{\bar{\veu}-\bar{\vel}}_{\infty} \leq 6 n g_\infty (A_I)$, $\norm{\bar{\veu}-\bar{\vel}}_{1} \leq 6 n g_1 (A_I)$.
\end{corollary}
\begin{proof}
Given an intial feasible solution, it is clear that we need at most $\log (\norm{\veu-\vel}_\infty)$ iterations, in each of which we apply Theorem~\eqref{thm:scaling_proximity} with $s=2$ to the bounds $\bar{l},\bar{u}$.
\end{proof}

Once we reached an optimal solution in the initial IP~\eqref{IP}, we could continue scaling, and look for a solution in the superlattice $\tfrac12 \Z^n$.
Continuing further, this can be used to find a $2^{-k}$-accurate solution for the continuous relaxation of~\eqref{IP}, as we shall see in the next section.

\subsection{Relaxation Oracles} \label{sec:relaxation}

The following is a simple corollary from the proximity theorem~\ref{thm:proximity}, applied to the $s$-scaled instance~\eqref{eq:s-ip-scaled} with the respective parameters.
\begin{corollary}[Solving~\eqref{relax} by solving~\eqref{eq:s-ip-lattice}]
\label{cor:scaling_relaxation}
Let $\varepsilon > 0$.
Then an optimal solution for~\eqref{eq:s-ip-lattice} is an $\epsilon$-accurate solution of~\eqref{relax} for any $s \leq \tfrac{\epsilon}{n g_\infty(A)}$.
\end{corollary}


This corollary immediately motivates an algorithm for finding an $\eps$-accurate solution for the continuous problem~\eqref{relax}.
However, for analysing the convergence of the halfling-augmentation, we require that the restriction of $f$ to $\Z^n$ takes integer values.
When applying the algorithm to a finer lattice $2^{-k} \Z^n$, we would need the assumption that $f$ also takes integer values on $2^{-k} \Z^n$.
Since $f$ is separable convex, this would lead to a huge term $f_{\max}$ in the running time estimate.
Instead, we will replace the function $f$ with a function $f^\prime$ equivalent to $f$ on the corresponding refinement of the lattice $\Z^n$, which also makes it unnecessary to assume integrality on the lattice.
We refer to the next section for the details.
\begin{theorem}\label{thm:scaling_relaxation}
Let $\varepsilon > 0$.
We can find an $\epsilon$-accurate solution to~\eqref{relax} by solving
\[
\Oh \left( n^3 g_1(A) (\log (n g_1(A))^2 \left[ \log (\norm{\veu - \vel}_\infty + 1) + \log \tfrac{n g_{\infty}(A)}{\varepsilon} \right] \right)
\]
many augmentation IPs, each with polynomial lower and upper bounds.
\end{theorem}
\begin{proof}
Let $k \df 2 \log (\norm{\veu - \vel}_\infty + 1)$.
For $s = 2^k$, the only feasible solution to~\eqref{eq:s-ip-lattice} is $\mathbf{0}$.
Applying Theorem~\ref{thm:scaling_proximity}, together with Corollary~\ref{cor:sepconvex_red}, we can find an optimum solution for the IP with lattice $\tfrac{s}{2} \Z^n$ by solving $\Oh (n^3 g_1(A) (\log (n g_1(A))^2)$ many augmenting IPs.
We update $k \leftarrow k-1$, and repeat the procedure.
Once $2^k \leq \tfrac{\varepsilon}{2 n g_{\infty}(A)}$, we stop with an $\epsilon$-accurate solution.
This yields $2 \log (\norm{\veu - \vel}_\infty + 1) + \log \tfrac{2 n g_{\infty}(A)}{\varepsilon}$ iterations, and the theorem is proven.
\end{proof}

\subsection{Reducibility Bounds} \label{sec:reducibility}
Thanks to the fact that our augmentation oracles only require $f$ to be represented by a comparison oracle, it is sufficient to prove the \emph{existence} of a small equivalent objective in order to improve our time complexity bounds (see~Proposition~\ref{prop:converge_comparison}).
Here, we show two such existence bounds, first for linear functions, and then for separable convex functions.
We also give asymptotically matching lower bounds in both cases.

\subsubsection*{Linear functions.}
Our main result here is the following:
\begin{theorem}[Linear reducibility] \label{thm:lin_red}
	\label{lem:ordering}
	Let $\vew \in \R^n$ and $f(\vex) = \vew \vex$.
	Then $f$ is $n (4nN)^{n}$-reducible.
\end{theorem}
This stands in contrast with Proposition~\ref{prop:FT} which gives a constructive but much worse bound of $(nN)^{\Oh(n^3)}$.
Replacing Proposition~\ref{prop:FT} with Theorem~\ref{thm:lin_red} in our strongly-polynomial algorithms thus shaves off a factor of $n^2$.
We first need a few definitions and a technical lemma.
A \emph{convex polyhedral cone} is a set $C = \{\sum_{\vex \in S} \alpha_{\vex} \vex \mid \alpha_{\vex} \geq 0 \, \forall \vex \in S\}$ for some finite set $S$.
The set $S$ is called a \emph{set of generators} of $C$, and we also write $\cone(S)$ for the cone generated by $S$. As we are only interested in convex polyhedral cones, we will write cone for brevity.

It is commonly known that $C$ can equivalently be described by a matrix $A$ as $C = \{\vex \mid A\vex \geq \vezero\}$.
For a finite set $S \subseteq \R^n$, the \emph{span of $S$} is defined as $\spann(S) \df \{\sum_{\vex \in S} \alpha_{\vex} \vex \mid \alpha_{\vex} \in \R, \, \forall \vex \in S\}$.
\begin{lemma}
\label{lem:facet-generating}
Let $C = \{\vex \mid A\vex \geq \vezero\} \subseteq \R^n$ be a cone with matrix $A \in \Z^{m \times n}$. Then there exists a finite set $S$ of integral vectors such that $C = \cone(S)$ and $\|\vev\|_1 \leq (2n\|A\|_\infty)^{n-1}$ for all $\vev \in S$.
\end{lemma}
\begin{proof}
We will use the following statement of Minkowski.
\begin{proposition}[Minkowski~\cite{minkowski1896geometrie}] 
Let $C = \{\vex \in \R^n \mid A\vex \geq \vezero\}$ be a polyhedral cone.
Let $S^\prime$ be the set of all possible solutions to any of the systems $M\vey = \veb'$ where $M$ consist of $n$ linearly independent rows of the matrix $\binom{A}{I}$ and $\veb^\prime =\pm \vece_j$ for the $j$-th canonic unit vector, $j=1,\dots,n$.
Then there is a subset $S \subseteq S^\prime$ such that $C = \cone (S)$.
\end{proposition}
Let $S$ be the set of generators described in this proposition and let $\vev \in S$ be one of those generators.
We will replace $\vev$ by an integral vector $\vev^\prime = \lambda \vev$, $\lambda > 0$ of the desired length.

Let $M$ be the submatrix of $\binom{A}{I}$ and $k \in [n]$ the index such that $\vev$ is the unique solution to $M\vex = \veb^\prime$ with either $b^\prime = \vece_k$ or $b^\prime = - \vece_k$.
Let $\vea$ be the $k$-th row of $M$ and let $M^\prime \in \Z^{(n-1) \times n}$ be the matrix $M$ without this row.
As $\spann (\vev) = \ker M^\prime$, we can replace $\vev$ by a Graver basis element $\vev^\prime \in \G(M')$ with $(\vea \vev^\prime)(\vea \vev) > 0$ and the set $S\setminus \{\vev\} \cup \{\vev^\prime\}$ still generates $C$.
As Graver basis elements come in antipodal pairs $(\vev^\prime,-\vev^\prime)$ (i.e., $\vev' \in \G(M')$ implies $-\vev' \in \G(M')$), there exists such a $\vev^\prime$ (note that $\vea \vev^\prime =0$ implies $\vev^\prime = \vezero$).
But by Lemma~\ref{lem:bound1}, the length of such a $\vev^\prime$ is bounded by $(2(n-1)\|A\|_\infty + 1)^{n-1} \leq (2n\|A\|_\infty)^{n-1}$.
Replacing every element in $S$ yields a set of generators $S^\prime$ as desired.
\end{proof}

\begin{proof}[Proof of Theorem~\ref{thm:lin_red}]
By Definitions~\ref{def:f_equivalence} and~\ref{def:reducibility} it is sufficient to show that there exists a vector $\vew' \in \Z^n$ with $\|\vew'\|_1 \leq n (4nN)^{n-1}$ such that if $g(\vex) = \vew' \vex$, then $g$ and $f$ are equivalent on $\DD \df [-N,N]^n$, see~Property~\eqref{eq:f_equivalence} of Definition~\ref{def:f_equivalence}.
The bound $\|\vew'\|_1 \leq n (4nN)^{n-1}$ then implies $g_{\max}^{[-N,N]^n} \leq n (4nN)^{n}$ and the statement of the Lemma.
We will define a cone $C$, show that any integral vector $\vez$ in the relative interior of $C$ fulfills Property~\eqref{eq:f_equivalence}, and then provide a sufficiently small integral vector in $C$.

For any two points $\veu,\vev \in \DD$ with $\vew \veu \geq \vew \vev$, consider the halfspace
\[
H(\veu,\vev) := \{\vez \mid (\veu-\vev) \vez \geq 0\}.
\]
We define the cone
\[
	\mathcal{C} := \bigcap_{\substack{
		(\veu,\vev) \in \DD \times \DD \\
		\vew \veu \geq \vew \vev}} H(\veu,\vev).
\]
Obviously, $\vew \in \mathcal{C}$.
Let $\vez$ be any point in the relative interior of $\mathcal{C}$. We will show that $\vez$ fulfils Property~\eqref{eq:f_equivalence}.
The implication ``$\Rightarrow$'' follows from the definition of the cone. Similarly, for the implication ``$\Leftarrow$'' the only interesting case is $\veu \vez = \vev \vez$ (as any pair $\veu,\vev$ induces at least one inequality). But as $\vez$ was chosen in the relative interior, both inequalities 
$\veu \vez \geq \vev \vez$ and $\veu \vez \leq \vev \vez$ must be present in the description of the cone, and hence $\veu \vew = \vev \vew$.


Hence, for any $\vez$ in the relative interior, Property~\eqref{eq:f_equivalence} is satisfied.
Moreover, $\mathcal{C}$ is described by an integral matrix $A$ with $\|A\|_\infty \leq 2N$, i.e., $\mathcal{C} = \{\vex \mid A\vex \geq \vezero\}$.
By Lemma~\ref{lem:facet-generating}, each generator of $\mathcal{C}$ is an integral vector of $\ell_1$-norm at most $(4nN)^{n-1}$.
Let $S$ be the set of all these vectors, and let $S^\prime$ be any maximal subset of linearly independent vectors.

Define $\vew' = \sum_{\vev \in S^\prime} \vev$.
It is left to show that $\vew'$ is in the relative interior of $\mathcal{C}$.
For this, it is sufficient to show that whenever $\veu \vew' = \vev \vew'$ for any two points $\veu, \vev$, then $\mathcal{C} \subseteq \{\vex \mid (\veu-\vev) \vex = 0\}$.
Let $\veu,\vev \in \DD$ be two vectors such that $\veu \vew' = \vev \vew'$. As $(\veu-\vev) \vex \geq 0$ or $(\veu-\vev) \vex \leq 0$ is feasible for $\mathcal{C}$, this implies that $(\veu-\vev) \vez = 0$ for all $\vez \in S^\prime$. But as $S^\prime$ was chosen to be maximal, this means that actually $\mathcal{C} \subseteq \ker (\veu-\vev)$, hence $\vew'$ is contained in the relative interior.
Thus $\vew'$ is as desired.
\end{proof}

\begin{theorem}[Linear lower bound]\label{thm:lin_red_lowerbound}
Let $\DD = [-N,N]^n$, $\vew = \left(1, (nN)^1, (nN)^2, \dots, (nN)^{n-1}\right) \in \Z^n$, and $f(\vex) = \vew \vex$.
There does not exist any $\vew' \in \Z^n$ with $\|\vew'\|_1 < (nN)^{n-1}$ such that $g(\vex) = \vew' \vex$ is equivalent to $f$ on $\DD$, hence no $g$ equivalent to $f$ with $g_{\max}^{[-N,N]^n} \leq N(nN)^{n-1}$.
\end{theorem}
\begin{proof}
For each $i \in [n-1]$ consider the two points $\vev^i, \veu^i \in \DD$, with $v^i_i = nN$ and $u^i_{i+1} = 1$ and all remaining coordinates zero.
Since $\vew \vev^i = \vew \veu^i$, any $g(\vex) = \vew'\vex$ satisfying~\eqref{eq:f_equivalence} must also satisfy $nN w'_i = w'_{i+1}$.
Thus
\begin{equation} \label{eq:equiv_chain}
(nN)^{n-1} w'_1 = (nN)^{n-2} w'_2 = \cdots = nN w'_{n-1} = w'_{n} \enspace .
\end{equation}
Any $\vew'$ for which $g(\vex) = \vew' \vex$ is equivalent with $f$ on $\DD$ must be nonzero, and since $\vew'$ is integer, we have that $w'_{1} \geq 1$, otherwise all terms of~\eqref{eq:equiv_chain} are $0$. 
We conclude that $\|\vew'\|_1 \geq w'_n \geq (nN)^{n-1}$.
\end{proof}

\subsubsection*{Separable convex functions.}
Turning our attention to separable convex functions, we shall provide a reducibility bound and also an algorithm constructing an equivalent function, although with a worse bound, similarly to the linear case where Theorem~\ref{thm:lin_red} and Proposition~\ref{prop:FT} contrast:
\begin{corollary} \label{cor:sepconvex_red}
	Let $f\colon \R^n \to \R$ be a separable convex function. Then, for every $N \in \N$,
	there exist separable convex functions $g_1, g_2$ which are equivalent to $f$ on $[-N,N]^n$ and satisfy $\forall \vex \in [-N,N]^n\colon g_1(\vex), g_2(\vex) \in \Z$, and
	\begin{enumerate}
		\item \label{it:sepconvex_red:1}$g_1$ satisfies $(g_1)_{\max}^{[-N,N]^n} \leq (n^2N)^{n(2N+1)+1}$,
		\item \label{it:sepconvex_red:2}$g_2$ is computable in strongly polynomial time and satisfies $(g_2)_{\max}^{[-N,N]^n} \leq 2^{\Oh((nN)^3)} (nN)^{\Oh((nN)^2)}$.
	\end{enumerate}
\end{corollary}
Interestingly, while for linear functions the reducibility bound does not depend exponentially on $N$, this is unavoidably not the case for separable convex functions, as we will show later (Theorem~\ref{thm:sepconvex_lb}).

We handle the separable convex case by reducing it to the linear case so that we can apply our existence bound (Theorem~\ref{thm:lin_red}) as well as the constructive algorithm of Frank and Tardos (Proposition~\ref{prop:FT}); a result similar to Corollary~\ref{cor:sepconvex_red}, part~\ref{it:sepconvex_red:1}, was proven by De Loera et al.~\cite{banff}.
We differ in exploiting the connection between linear and separable convex function, which allows us to additionally use Proposition~\ref{prop:FT} to obtain Corollary~\ref{cor:sepconvex_red}, part~\ref{it:sepconvex_red:2}.

Let $M \df n(2N+1)$, and for any separable function $f$ defined on the box $[-N,N]^n$, let $\mu(f) \df \left(f_1(-N), f_1(-N+1), \dots, f_1(N), f_2(-N), \dots, f_2(N), \dots, f_n(-N), \dots, f_n(N)\right) \in \R^{M}$.
The vector $\mu(f)$ can be divided into $n$ bricks of $2N+1$ values, so, as usual, we denote the index of the brick with a superscript and the index within a brick with a subscript.
Note that the natural inverse $\mu^{-1}(\vew)$ defines for any vector $\vew \in \R^M$ a separable function $f$ by setting, for each $i \in [n], j \in [-N,N]$, $f_i(j) \df w^i_j$.
Furthermore, define a mapping $\eta: [-N,N]^n \to \{0,1\}^M$ as $\vey \df \eta(\vex)$ where $y^i_{x_i} \df 1$ for each $i \in [n]$ and all other coordinates of $\vey$ are zero.
Again, the natural inverse is defined for all points $\vey \in \{0,1\}^M$ which have exactly one $1$ in each brick by setting $\vex \df \eta^{-1}(\vey)$ as $x_i = \sum_{j=-N}^N j \cdot y^i_j$ for each $i \in [n]$.
Note that for any $\vez \in [-N,N]^n$ and any separable function $f$, $f(\vez) = \mu(f) \eta(\vez)$.
\begin{lemma}\label{lem:sepconvex_lin_equiv}
	Let $n \in \N$, $N \in \N$, $n \geq 4$, $f$ be a separable convex function, $\vew \df \mu(f)$, and $f'(\vey) \df \vew \vey$.
	If $\vew' \in \Z^M$ is such that $g'(\vey) \df \vew' \vey$ is equivalent with $f'$ on $B_1(n) = \{\vey \in \Z^M \mid \|\vey\|_1 \leq n\}$, then $g \df \mu^{-1}(\vew')$ is equivalent with $f$ on $[-N,N]^n$ and $g$ is separable convex.
\end{lemma}
\begin{proof}
	There are two things to prove: that $g$ is equivalent to $f$, and that $g$ is convex.
	The equivalence follows from the equivalence of $f'$ and $g'$, because for any $\vex, \vex' \in [-N,N]^n$,
	\begin{equation*}
	f(\vex) \geq f(\vex') \underset{(1)}{\Longleftrightarrow} \vew \eta(\vex) \geq \vew \eta(\vex') \underset{(2)}{\Longleftrightarrow}
	\vew'\eta(\vex) \geq \vew'\eta(\vex') \underset{(3)}{\Longleftrightarrow} g(\vex) \geq g(\vex'),
	\end{equation*}
	where $(1)$ and $(3)$ are from the fact that $f(\vez) = \vew \eta(\vez)$ and $g(\vez) = \vew' \eta(\vez)$ for all $\vez \in [-N,N]^n$, respectively, and $(2)$ is from the equivalence of $f'$ and $g'$.
	Showing convexity of $g$ reduces to showing convexity of each univariate function $g^i$.
	Fix $i \in [n]$.
	It is known that a function $h$ of one variable is convex on $[-N,N]$ if, for any $x,y,z \in [-N,N]$ with $y=x+2$ and $z=x+1$, the inequality $h(x) + h(y) \geq 2h(z)$ holds.
	Consider the two points $\veu, \vev \in \Z^{2N+1}$ defined as follows: 
	\[
	u_k \df \begin{cases}0 & k\neq x,y\\ 1 & k=x,y\end{cases} \quad \text{and} \quad v_k \df \begin{cases}0 & k\neq z\\ 2 & k=z\end{cases} \enspace .
	\]
	Then $\vew \veu \geq \vew \vev \Leftrightarrow f_i(x) + f_i(y) \geq 2f_i(z)$ and similarly $\vew' \veu \geq \vew' \vev \Leftrightarrow g_i(x) + g_i(y) \geq 2g_i(z)$.
	Since by equivalence of $f'$ and $g'$ we have $\vew \veu \geq \vew \vev \Leftrightarrow \vew' \veu \geq \vew' \vev$ the chain of equivalences is complete and $g_i$ is convex if and only if $f_i$ is, thus $g$ is convex.
\end{proof}
Using Theorem~\ref{thm:lin_red} and Proposition~\ref{prop:FT} it is easy to prove Corollary~\ref{cor:sepconvex_red}:
\begin{proof}[Proof of Corollary~\ref{cor:sepconvex_red}]
	We will apply Theorem~\ref{thm:lin_red} and Proposition~\ref{prop:FT} to $\mu(f)$, which has $\bar{n} \df n(2N+1)$ dimensions, and, by Lemma~\ref{lem:sepconvex_lin_equiv}, it suffices to obtain a linear function equivalent to $\mu(f)$ on $B_1(n) \subseteq B_\infty(n) = [-n,n]^{\bar{n}}$, hence $\bar{N} \df n$.
	For part~\ref{it:sepconvex_red:1}, we have
	\[
	(g_1)_{\max}^{[-N,N]^n} \leq \bar{n}(4\bar{n}\bar{N})^{\bar{n}} \leq (n(2N+1))(4\cdot n(2N+1) \cdot n)^{n(2N+1)} \leq (n^2N)^{n(2N+1)+1},
	\]
	whereas for part~\ref{it:sepconvex_red:2} we have
	\[
	(g_2)_{\max}^{[-N,N]^n} \leq 2^{\Oh(\bar{n}^3)} \bar{N}^{\Oh(\bar{n}^2)} \leq 2^{\Oh((nN)^3)} n^{\Oh((nN)^2)} \enspace . \qedhere
	\]
\end{proof}
It remains to show that Corollary~\ref{cor:sepconvex_red}, part~\ref{it:sepconvex_red:1}, cannot be asymptotically improved:
\begin{theorem}[Separable convex lower bound] \label{thm:sepconvex_lb}
	There exists a separable convex function $f\colon \R^n \to \R$ which is not $(c^{nN})$-reducible for any $c < \phi$, where $\phi = \frac{1+\sqrt{5}}{2} \approx 1.61083$ is the golden ratio.
\end{theorem}
\begin{proof}
	Let $F_i$ be the $i$-th Fibonacci number, where $F_0, F_1 \df 1$ and $F_i \df F_{i-1} + F_{i-2}$ for all $i \geq 2$.
	Define a function $f\colon \N^n \to \N$ by setting, for each $i \in [n]$, $j \in \N$, 
	\[
	f_i(j) \df
	\begin{cases}
	0 & \text{if $j=0$}, \\
	F_k & \text{where } k \df (i-1) + (j-1)n \enspace .
	\end{cases}
	\]
	For any $k \in \N$, let $i(k) \df 1+\floor*{\frac{k}{n}}$ and $j(k) \df 1+(k \mod n)$.
	Construct a sequence of points $\veu_k, \vev_k$ for all $k \in \Z$ as follows.
	Begin with all-zero vectors, and set $(\veu_k)_{i(k)} \df j(k)$, $(\vev_k)_{i(k-1)} \df j(k-1)$, and $(\vev_k)_{i(k-2)} \df j(k-2)$, i.e., $\suppo(\veu_k) = \{i(k)\}$ and $\suppo(\vev_k) = \{i(k-1), i(k-2)\}$.
	Note that 
	\[
	F_k \underset{(1)}{=} f(\veu_k) \underset{(2)}{=} f(\vev_k) \underset{(3)}{=} f_{i(k-1)}(j(k-1)) + f_{i(k-2)}(j(k-2)) \underset{(4)}{=} F_{k-1} + F_{k-2},
	\]
	where $(1)$ and $(4)$ hold by the definition of $f$, $(3)$ is due to separability of $f$, and $(2)$ follows from $F_k = F_{k-1} + F_{k-2}$.
	Also note that, because, for any $k \geq 2$, $F_k - F_{k-1} = F_{k-2}$ and the Fibonacci sequence is growing, the function $f$ is convex.
	
	Now consider any function $g$ equivalent to $f$ on $[-N,N]^n$ with integral values.
	By equivalence with $f$, $g$ must satisfy, for all $k \in \N$, the equality $g(\veu_k) = g(\vev_k)$, and hence $g(\veu_k) = F_k \cdot g^1(1)$.
	Because $g^1(1) > g^1(0)$ must hold since $f_1(1) > f_1(0)$, and, w.l.o.g. we may assume $g_1(0) = 0$, and $g$ is integral, we have $g_1(1) \geq 1$, and hence $g(\veu_k) \geq F_k$.
	By standard bounds $F_k \approx \phi^k$ and thus for large enough value of $k$, $F_k \geq c^k$, and $g_n(N) \geq f_n(N) = F_{nN-1} > c^{nN-1}$, concluding the proof.
\end{proof}
\begin{remark}
	Note that we have proven an even stronger result: there is no \emph{separable} (not necessarily convex) function which is integral on integral points and equivalent to $f$ with largest value smaller than $c^{nN}$ for $c < \phi$.
\end{remark}

\subsection{Augmentation Oracles} \label{sec:almostlinear}
\subsubsection{Primal Treedepth}
The main result of this section is the following:
\begin{theorem}[Nearly linear $\td_P(A)$] \label{thm:nearlylinear_primal}
	Let $F$ be a $\td$-decomposition of $G_P(A)$.
	There is a computable function $g$ such that~\eqref{IP} can be solved in time
	\[g(\|A\|_\infty, \td_P(A)) \cdot n^{1+o(1)} \log \|\veu-\vel\|_\infty (\log f_{\max})^{\ttd(F)-1} \enspace .\]
	When $f(\vex) = \vew \vex$, then~\eqref{IP} can be solved in time
	\[g(\|A\|_\infty, \td_P(A)) \cdot n^{1+o(1)} \log \|\veu-\vel\|_\infty \cdot (\|\vew\|_\infty)^{o(1)}  \enspace .\]
	The function $g$ is a tower of exponentials of height $\Oh(\ttd(F))$ (i.e., $g_\infty(A)^{2\td_P(A)}$, where $g_\infty(A)$ is bounded by Lemma~\ref{lem:primal_norm}).
	The $n^{1+o(1)}$ term is more specifically $n (\log n)^{\ttd(F)+1}$.
\end{theorem}
Let us first describe the algorithm realizing Theorem~\ref{thm:nearlylinear_primal} and then gradually prove its correctness and the complexity bound given in the Theorem.

\paragraph*{Algorithm}
The algorithm starts by applying the scaling algorithm (Corollary~\ref{cor:scaling_algo}), hence solving the instance of~\eqref{IP} by solving $\log \|\veu-\vel\|_\infty$ auxiliary instances~\eqref{eq:s-ip-scaled} whose constraint matrix has the form $A' \df A_I=(A~I)$ and with smaller right hand side and bounds.
The scaling algorithm thus forms an outer loop of the algorithm.
Let the auxiliary instance of the scaling algorithm be 
\begin{equation}
\min \tilde{f}(\tilde{\vex}):\, A'\tilde{\vex}=\tilde{\veb}, \, \tilde{\vel} \leq \tilde{\vex} \leq \tilde{\veu}, \, \tilde{\vex} \in \Z^{n+m} \enspace . \label{eq:primal_aux}
\end{equation}
As usual, we assume (by Lemma~\ref{lem:decomposition}) that $A'$ is block-structured along $F$ and we have, for each $i \in [d]$, matrices $\bar{A}'_i, A'_i$ and a $\td$-decomposition $F_i$ of $G_P(A_i)$.
This instance is solved by an augmentation procedure which forms an inner loop and is defined recursively over $\ttd(F)$ as follows:

\begin{enumerate}
	\item \label{it:primalalgo:2}Compute an initial feasible solution $\tilde{\vex}_0$ of~\eqref{eq:primal_aux} using the algorithm we shall describe (recall that by Corollary~\ref{cor:feas_td} ``feasibility is as easy as optimization'').
	\item \label{it:primalalgo:3}Perform the augmentation procedure where an augmenting step $\veh$ for $\tilde{\vex}$ is computed as follows.
	Let $\Gamma_2 \df  \{1,2,4,\dots, 2^{\ceil{\log\|\tilde{\veu}-\tilde{\vel}\|}}\}$.
	For each $(\lambda, \veg^0) \in \Gamma_2 \times \left( [-g_\infty(A'), g_\infty(A')]^{k_1(F)} \cap [\tilde{\vel}^0, \tilde{\veu}^0] \right)$, solve, for each $i \in [d]$,
	\begin{equation}
	\min \tilde{f}^i(\veh^i):\, A'_i \veh^i = \tilde{\veb}^i - \bar{A}'_i \lambda \veg^0,\, \tilde{\vel}^i \leq \tilde{\vex}^i + \veh^i \leq \tilde{\veu}^i, \, \veh^i \in \Z^{n_i}, \label{eq:primal_aux_subproblem}
	\end{equation}
	and let $\veh \df (\lambda \veg^0, \veh^1, \dots, \veh^d)$ be a minimizer of $f(\tilde{\vex} + \veh)$ among all choices of $(\lambda, \veg^0)$.
	If for a given pair $(\lambda, \veg^0)$ any subproblem~\eqref{eq:primal_aux_subproblem} is infeasible, disregard this pair.
	\begin{enumerate}
		\item If $\ttd(F) > 1$, then~\eqref{eq:primal_aux_subproblem} is solved by a recursive call because $\ttd(F_i) < \ttd(F)$.
		\item \label{it:t:primalalgo:3b} Otherwise $\ttd(F) = 1$ and~\eqref{eq:primal_aux_subproblem} is solved by the scaling algorithm (Corollary~\ref{cor:scaling_algo}), with each of its auxiliary instances being solved by an augmentation procedure where an augmenting step $\veh$ for $\tilde{\vex}$ is obtained by enumerating all pairs $(\lambda, \veg) \in \Gamma_2 \times \left([-g_\infty(A'), g_\infty(A')]^{n+m} \cap [\tilde{\vel}, \tilde{\veu}]\right)$ and taking $\veh \df \lambda \veg$ which satisfies $A'\veh = \vezero$ and minimizes $f(\tilde{\vex} + \lambda \veg)$.
	\end{enumerate}
\end{enumerate}

\begin{lemma}
The algorithm described above returns the optimum of~\eqref{IP}.
\end{lemma}
\begin{proof}
The correctness of the outer loop follows from the correctness of the scaling algorithm as shown in Corollary~\ref{cor:scaling_algo}.
The correctness of the inner loop as implemented by the recursive algorithm follows from the fact that an augmenting step $\veh$ computed in point~\ref{it:primalalgo:3} is always 
a halfling 
and thus augmentation only stops when a global optimum has been reached.
\end{proof}
Thus it remains to prove the complexity bounds of Theorem~\ref{thm:nearlylinear_primal}, for which we need to
\begin{enumerate}
	\item bound the number of iterations of the augmentation procedure defined in point~\ref{it:primalalgo:3} required to reach the optimum by $g(\td_P(A), \|A\|_\infty) \log f_{\max}$, for some computable function $g$ (Lemma~\ref{lem:primal_speedup}),
	\item bound the time required to solve the leaf instances of the recursive algorithm (point \ref{it:t:primalalgo:3b}) by a function independent of $f_{\max}$ (Lemma~\ref{lem:fixeddim_nofmax}),
	\item bound the complexity of the inner loop (points~\ref{it:primalalgo:2}--\ref{it:primalalgo:3}) using the previous two claims (Lemma~\ref{lem:nearlylinear_primal_aux}),
	\item bound the overall complexity of the algorithm, i.e., use the bounds of the scaling algorithm together with Lemma~\ref{lem:nearlylinear_primal_aux} to prove the complexity bound of Theorem~\ref{thm:nearlylinear_primal}. 
\end{enumerate}
The following lemma shows that the augmenting steps computed in point~\ref{it:primalalgo:3} of the algorithm decrease the optimality gap rapidly.
\begin{lemma} \label{lem:primal_speedup}
	Let $A \in \Z^{m \times n}$, $k \in \N$, $k \leq n$, and $\vex$ be a feasible solution of~\eqref{IP}.
	For each $\lambda \in \Gamma_2 = \{1,2,4,8,\dots\}$ let $S_\lambda \df \{\lambda \veh' \mid \veh' \in [-g_\infty(A), g_\infty(A)]^{k} \} \times \Z^{n-k}$ and let $\veh_\lambda$  be a solution of
	\begin{equation} 
	S_\lambda\best \{f(\vex + \veh) \mid A\veh = \vezero, \, \vel \leq \vex + \veh \leq \veu, \, \veh \in \Z^n\} \enspace . \tag{$S_\lambda\best$}\label{eq:lem_primal_speedup_subproblem}
	\end{equation}
	Let $\veh^*$ be a minimizer of $f(\vex + \veh_{\lambda})$ over $\lambda \in \Gamma_2$, and let $\vex^*$ be any optimum of~\eqref{IP}.
	Then
	\begin{equation} 
	f(\vex) - f(\vex + \veh^*)  \geq \frac{1}{2(2 g_\infty(A) + 1)^{k}} \left(f(\vex) - f(\vex^*) \right)\enspace . \label{eq:lem_primal_speedup}
	\end{equation}
\end{lemma}
\begin{proof}
	Let $\G' = \{\veg' \in \Z^k \mid (\veg', \veg'') \in \G(A)\}$ and observe that $\G' \subseteq [-g_\infty(A), g_\infty(A)]^{k}$ and thus $|\G'| \leq (2g_\infty(A)+1)^{k}$.
	By the Positive Sum Property (Proposition~\ref{prop:possum}) we have that $\vex^* - \vex = \sum_{i=1}^{n'} \alpha_i \veg_i$ for some $n' \leq 2n-2$, $\alpha_i \in \N$, and $\veg_i \in \G(A)$.
	Rewriting each element $\veg_i$ into its first $k$ coordinates and its remaining $n-k$ coordinates as $\veg_i = (\veg'_i, \veg''_i)$, we may rearrange the decomposition by grouping its summands by the first component as
	$$\vex^* - \vex = \sum_{\veg' \in \G'} \sum_{j} \beta_j (\veg', \veg''_j)$$
	and denote, for each $\veg' \in \G'$, $\veh_{\veg'} \df \sum_{j} \beta_j (\veg', \veg''_j)$.
	Now by separable convex superadditivity (Proposition~\ref{prop:superadditivity}) and an averaging argument, we get that there must exist a $\veg' \in \G'$ such that
	\begin{equation}
	f(\vex) - f(\vex + \veh_{\veg'}) \geq \frac{1}{|\G'|} \left(f(\vex) - f(\vex^*)\right) \enspace .
	\label{eq:primal_g_prime}
	\end{equation}
	
	This would be sufficient if our claim was made with $\veh^*$ as the minimum over $\lambda \in \N$, however, we now need to deal with the fact that we are taking $\lambda \in \Gamma_2 = \{1,2,4,8,\dots\}$.
	Let $\veg' \in \G'$ satisfy~\eqref{eq:primal_g_prime}, let $N \df \sum_{j} \beta_j$ from the definition of $\veh_{\veg'}$ and let $N' \df 2^{\floor{\log N}}$ be the nearest smaller integer power of $2$.
	Again applying Proposition~\ref{prop:superadditivity} we get that
	\[
	f(\vex) - f(\vex + \veh_{\veg'}) \geq \sum_{j} \beta_j \left(f(\vex)-f(\vex + (\veg', \veg''_j)) \right) \enspace .
	\]
	By an averaging argument there exist numbers $\beta'_j \leq \beta_j$ such that $\sum_j \beta'_j = N'$ and $\veh'_{\veg'} \df \sum_j \beta'_j (\veg', \veg''_j)$ satisfies
	\[
	f(\vex) - f(\vex + \veh'_{\veg'}) \geq \frac{N'}{N}\left(f(\vex) - f(\vex + \veh_{\veg'}) \right) \geq \frac{1}{2} \frac{1}{|\G'|}\left(f(\vex)-f(\vex^*) \right) \enspace .
	\]
	Let $\lambda \G' \df \{\lambda \veg' \mid \veg' \in \G'\}$.
	By $|\G'| \leq (2 g_\infty(A) + 1)^{k}$, there exists a $\lambda \in \Gamma_2$ and $(\veh', \veh'') \in \lambda \G' \times \Z^{n - k}$ which satisfies~\eqref{eq:lem_primal_speedup} and since $\lambda\G' \times \Z^{n - k} \subseteq S_\lambda$, any solution of~\eqref{eq:lem_primal_speedup_subproblem} satisfies~\eqref{eq:lem_primal_speedup}.
\end{proof}
Let us now show that an instance of~\eqref{IP} in small dimension can be solved in time independent of $f_{\max}$ by a combination of the scaling algorithm and a reducibility bound.
This bounds the complexity of solving the leaf instances in point~\eqref{it:t:primalalgo:3b} of the algorithm.
\begin{lemma}[] \label{lem:fixeddim_nofmax}
	Problem~\eqref{IP} can be solved in time $(3m \|A\|_\infty)^{3n^2 + m} \cdot \log (\|\veu-\vel\|_\infty)$.
\end{lemma}
\begin{proof}
	We will combine the scaling algorithm (Corollary~\ref{cor:scaling_algo}) with the separable convex reducibility bound (Corollary~\ref{cor:sepconvex_red}).
	The scaling algorithm solves a sequence of instances with a constraint matrix $A_I$, so we begin by considering the norm of its Graver elements.
	By Lemma~\ref{lem:bound1} we have $g_\infty(A_I) \leq (3m\|A\|_\infty)^m$ and thus the bounds $\bar{\vel}_i, \bar{\veu}_i$ of each auxiliary instance satisfy $\|\bar{\veu}_i  - \bar{\vel}_i\|_\infty \leq 2(n+m)(3m\|A\|_\infty)^m$, and, possibly after centering (Lemma~\ref{lem:centered}), also satisfy $[\bar{\vel}_i, \bar{\veu}_i] \subseteq [-N,N]^{n+m}$ with $N \df (n+m)(3m\|A\|_\infty)^m$.
	By Lemma~\ref{lem:primal},~\eqref{AugIP} can be solved in time $\left((3m\|A\|_\infty)^{m}\right)^{n+m} \leq \left(3m\|A\|_\infty\right)^{2n^2}$.
	Using~\eqref{AugIP} to realize the halfling augmentation procedure (Corollary~\ref{cor:lambda_oracle}) thus takes time $\left(3m\|A\|_\infty\right)^{2n^2} \log (\|\bar{\veu}_i-\bar{\vel}_i\|_\infty) \log (f_{\max}) \leq \left(3m\|A\|_\infty\right)^{2n^2+1} \log (f_{\max})$, and applying Corollary~\ref{cor:sepconvex_red} shows that we can replace $f_{\max}$ by
	\[
	((n+m)^2N)^{(n+m)(2N+1)+1} \leq \left(4n^2 \cdot 3n(3m\|A\|_\infty)^{m}\right)^{2n \cdot (3n)(3m\|A\|_\infty)^m} \leq \left(12n^3\cdot 3m\|A\|_\infty\right)^{6n^3(3m\|A\|_\infty)^m},
	\]
	and hence
	\[
	\log (f_{\max}) = 6n^3 \cdot (3m\|A\|_\infty)^m \cdot \log(12n^3 \cdot 3m\|A\|_\infty)
	\leq 6n^4 \cdot (3m\|A\|_\infty)^m
	\]
	obtaining a time bound of 
	\[\left(3m\|A\|_\infty\right)^{2n^2+1} \cdot \left( 6n^4 \cdot (3m\|A\|_\infty)^m\right) \leq 6n^4\cdot (3m \|A\|_\infty)^{2n^2 + m+1} \leq (3m \|A\|_\infty)^{3n^2 + m} \enspace .\]
	The number of instances which are solved in the course of the scaling algorithm is at most $\log \|\veu-\vel\|_\infty$, finishing the proof.
\end{proof}
The main technical statement bounds the complexity of the recursive algorithm of point~\ref{it:primalalgo:3}.
\begin{lemma} \label{lem:nearlylinear_primal_aux}
	Let $\II$ be an instance of~\eqref{IP} with a constraint matrix $A'$, let $A\df A'_I = (A'~I)$, $F$ be the $\td$-decomposition $F''_P$ of $G_P(A)$ from Lemma~\ref{lem:feas_td}, $k_{\ttd(F)} \df k_{\ttd(F)}(F)$, $P \df (3k_{\ttd(F)} \|A\|_\infty)^{4k_{\ttd(F)}^2} \cdot \log(\|\veu-\vel\|_\infty)$, $\tilde{L}_1 \df 1+\log (\|\veu-\vel, \veb\|_\infty)$, $\tilde{L}_2 \df \log(\max\{f_{\max}, \|\veb\|_1\})$.
	Then any \eqref{IP} instance $\II'$ with a constraint matrix $A$ is solvable in time $P \cdot (3g_\infty(A))^{2\height(F)} \cdot n (\tilde{L}_1 \cdot \tilde{L}_2)^{\ttd(F)-1}$.
\end{lemma}
\begin{proof}
	We will again use an inductive claim with two parts.
	\begin{claim*}
		Let $\ttd(F) \geq 2$, $k_1 \df k_1(F)$, and $k' \df 2\height(F) - k_1(F)$.
		\begin{enumerate}[label=(\alph*)]
			\item \label{it:nearlylinear_primal:claim1} There is an algorithm solving problem~\eqref{eq:lem_primal_speedup_subproblem} in time \[P \cdot (2g_\infty(A)+1)^{k_1}\cdot  (3g_\infty(A))^{2(\height(F)-k_1)} n (\tilde{L}_1 \cdot \tilde{L}_2)^{\ttd(F)-2} \enspace .\]
			\item \label{it:nearlylinear_primal:claim2} An $\AAA$ which converges in $\frac{2}{3} (2 g_\infty(A)+1)^{k_1} \log \left(f_{\max}\right)$ steps is realizable in time \[P \cdot (2g_\infty(A)+1)^{k_1}\cdot  (3g_\infty(A))^{2(\height(F)-k_1)} \cdot n \cdot \tilde{L}_1^{\ttd(F)-1} \tilde{L}_2^{\ttd(F)-2} \enspace .\]
		\end{enumerate}
	\end{claim*}
	Let us first give an outline of the proof and then provide the details.
	The proof proceeds by induction on $\ttd(F)$, where the algorithm for the base case $\ttd(F)=2$ uses a recursive call to the algorithm we exhibited in Lemma~\ref{lem:fixeddim_nofmax} (i.e., the algorithm of point~\ref{it:t:primalalgo:3b}). The structure of the proof is
	\[\text{Claim, part~\ref{it:nearlylinear_primal:claim1}} \underset{(1)}{\implies} \text{Claim, part~\ref{it:nearlylinear_primal:claim2}} \underset{(2)}{\implies} \text{Lemma~\ref{lem:nearlylinear_primal_aux}} \enspace \]
	Implication $(1)$ is proved as follows.
	By Lemma~\ref{lem:primal_speedup}, the $\AAA$ can be realized by solving subproblem~\eqref{eq:lem_primal_speedup_subproblem} for each $\lambda \in \Gamma_2$ and picking the best solution, which is possible in the time claimed by part~\ref{it:nearlylinear_primal:claim1} multiplied by $\log \|\veu-\vel\|_\infty$, which is the bound of part~\ref{it:nearlylinear_primal:claim2}; the convergence rate analysis which results in the factor $\frac{2}{3}$ is identical to that of Lemma~\ref{lem:halfling}.
	Implication $(2)$ follows because the $\AAA$-augmentation procedure can be realized by $\frac{2}{3} (2 g_\infty(A)+1)^{k_1(F)} \log \left(f_{\max}\right)$ calls to $\AAA$ (notice that the subproblem objective $f^i$ satisfies $f^i_{\max} \leq f_{\max}$ so our use of $f_{\max}$ in this bound is justified), and then by the last part of Corollary~\ref{cor:feas_td} (``feasibility as easy as optimization''), problem~\eqref{IP} can be solved in the claimed time because $A$ has the required form $(A'~I)$.
	Thus the main work is to prove part~\ref{it:nearlylinear_primal:claim1} of the Claim.
	
	We will describe the algorithm which realizes part~\ref{it:nearlylinear_primal:claim1} recursively and then analyze its complexity.
	As usual, we assume $A$ is block-structured along $F$, thus we have matrices $\bar{A}_1, \dots, \bar{A}_d, A_1, \dots, A_d$ and $\td$-decompositions $F_1, \dots, F_d$ of $G_P(A_1), \dots, G_P(A_d)$, where $d \in \N$, with $\bar{A}_i$ having $k_1$ columns and $\ttd(F_i) < \ttd(F)$ for each $i \in [d]$.
	Moreover, by Proposition~\ref{prop:AIstructure}, each $A_i$ is itself of the form $A_i = (A_i'~I)$.
	We will prove part~\ref{it:nearlylinear_primal:claim1} of the Claim by solving subproblems involving $\td$-decompositions with topological height less than $\ttd(F)$.
	Those are solved either by a recursive call to Lemma~\ref{lem:nearlylinear_primal_aux} if $\ttd(F_i) \geq 2$, or by an application of Lemma~\ref{lem:fixeddim_nofmax} if $\ttd(F_i) = 1$.
	
	Given a $\lambda \in \Gamma_2$, iterate over all $\veg^0 \in \left([-g_\infty(A),g_\infty(A)]^{k_1} \cap [\vel^0, \veu^0]\right)$ and for each use the algorithm for smaller topological height to compute $d$ vectors $\veh^i$, $i \in [d]$, such that $\veh^i$ is an optimum of
	\begin{equation}
	\min \{f^i(\vex^i + \veh) \mid A_i \veh = \veb^i - \bar{A}_i \lambda \veg^0, \, \vel^i \leq \vex^i + \veh \leq \veu^i, \, \veh^i \in \Z^{n^i}\} \enspace . \label{eq:slambdabest_subprob}
	\end{equation}
	Finally return the vector $(\lambda \veg^0, \veh^1 \dots, \veh^d)$ which minimizes $f^0(\vex^0 + \lambda \veg^0) + \sum_{i=1}^d f^i(\vex^i + \veh^i)$ over all choices of $\veg^0$.
	If for some $\veg^0$ at least one subproblem~\eqref{eq:slambdabest_subprob} is infeasible, disregard this $\veg^0$.
	If~\eqref{eq:slambdabest_subprob} is infeasible for all $\veg^0$, return that the problem is infeasible.
	
	Let us compute the complexity of this procedure.
	For each $i \in [d]$, $\height(F_i) + k_1 \leq \height(F)$ and $f^i_{\max} \leq f_{\max}$.
	There are at most $(2g_\infty(A)+1)^{k_1}$ choices of $\veg^0$, and computing the solution $(\veg^1, \dots, \veg^d)$ for each $\veg^0$ takes time at most \[\sum_{i=1}^d P \cdot (3g_\infty(A))^{2\height(F_i)} \cdot n_i (\tilde{L}_1 \cdot \tilde{L}_2)^{\ttd(F_i)-1} \leq P \cdot (3g_\infty(A))^{2(\height(F)-k_1)} n (\tilde{L}_1 \cdot \tilde{L}_2)^{\ttd(F)-2} \enspace .\]
	Summing over all choices of $\veg^0$, we get that the time complexity is at most
	\begin{equation}
	(2g_\infty(A)+1)^{k_1} \cdot P \cdot (3g_\infty(A))^{2(\height(F)-k_1)} n (\tilde{L}_1 \cdot \tilde{L}_2)^{\ttd(F)-2} \enspace . \label{eq:primal:claim1bound}
	\end{equation}
	This finishes the proof of part~\ref{it:nearlylinear_primal:claim1} of the Claim.
	
	Let us move on to part~\ref{it:nearlylinear_primal:claim2}.
	One call of the oracle $\AAA$ is realized by solving~\eqref{eq:slambdabest_subprob} for each $\lambda \in \Gamma_2$.
	Because $|\Gamma_2| \leq 1+\log\|\veu-\vel\|_\infty \leq \tilde{L}_1$, the time required to realize one call of the oracle is the term~\eqref{eq:primal:claim1bound} multiplied by $\tilde{L}_1$, which is the bound of part~\ref{it:nearlylinear_primal:claim2} of the Claim.
	To prove the convergence bound, Lemma~\ref{lem:primal_speedup} with $k \df k_1$ shows that an augmenting step $\veh$ decreases the gap $f(\vex) - f(\vex^*)$ by a multiplicative factor of $\frac{1}{2g_\infty(A)^{k_1}}$.
	Repeating the analysis of Lemma~\ref{lem:halfling} thus shows that the number of iterations until $f(\vex) - f(\vex^*) < 1$ is $\frac{2}{3} (2 g_\infty(A)+1)^{k_1} \log \left(f_{\max}\right) \leq \frac{2}{3} (2 g_\infty(A)+1)^{k_1} \log \left(f_{\max}\right)$.
	This concludes the proof of part~\ref{it:nearlylinear_primal:claim2} of the Claim.
	
	With $\AAA$ at hand, we can solve~\eqref{IP} in time
	\[
	\frac{2}{3}(2 g_\infty(A)+1)^{2k_1} P \cdot (3g_\infty(A))^{2(\height(F)-k_1)} n (\tilde{L}_1 \cdot \tilde{L}_2)^{\ttd(F)-1} \enspace .
	\]
	Handling feasibility using Corollary~\ref{cor:feas_td} then requires at most the same time by the definition of $\tilde{L}_1$ and $\tilde{L}_2$: the feasibility instance~\eqref{eq:auxiliary_feasibility} of Lemma~\ref{lem:feas_instance} has lower and upper bounds bounded by $\max\{\|\veu-\vel\|_\infty,\|\veb\|_\infty\} \leq 2^{\tilde{L}_1}$, and its objective function has largest value $\|\veb\|_1 \leq 2^{\log \|\veb\|_1} \leq 2^{\tilde{L}_2}$.
	Thus, we may bound the total complexity of realizing~\eqref{IP} as
	\[
	2\cdot \frac{2}{3}(2 g_\infty(A)+1)^{2k_1} P \cdot (3g_\infty(A))^{2(\height(F)-k_1)} n (\tilde{L}_1 \cdot \tilde{L}_2)^{\ttd(F)-1} 
	\leq (3g_\infty(A))^{2\height(F)} n (\tilde{L}_1\cdot  \tilde{L}_2)^{\ttd(F)-1} \enspace . \qedhere
	\]
\end{proof}

\begin{proof}[Proof of Theorem~\ref{thm:nearlylinear_primal}]
	For the proof we need a standard bound~\cite[Exercise 3.18]{CyganFKLMPPS:2015}:
	\begin{equation}
	(\log n)^k \leq 2^{k^2/2} \cdot 2^{{\log \log n}^2 /2} = 2^{k^2/2} n^{\oh(1)} \enspace . \label{eq:lognk}
	\end{equation}
	Let $F$ be an optimal $\td$-decomposition of $A$, let $A_I \df (A~I)$ and $F_I \df F''_P$ be a $\td$-decomposition of $G_P(A_I)$ from Lemma~\ref{lem:feas_td} such that $G_P(A_I) \subseteq \cl(F_I)$ and $\ttd(F_I) = \ttd(F_P)$.
	We solve~\eqref{IP} using the scaling algorithm, which involves solving $2\log \|\veu-\vel\|_\infty+2$ auxiliary instances~\eqref{eq:s-ip-scaled}.
	Note that the first $1+\log \|\veu-\vel\|_\infty$ of these are of the form~\eqref{eq:auxiliary_feasibility} and have the constraint matrix $A_I$.
	In our complexity estimates we include such an instance (feasibility or optimization) of the scaling algorithm which dominates the given bound.
	Each instance has bounds $\bar{\vel}_i, \bar{\veu}_i$ satisfying $\|\bar{\veu}_i - \bar{\vel}_i\|_\infty \leq 8ng_\infty(A_I) =: 2N$.
	By Lemma~\ref{lem:feas_td}, $\td_P(A_I) \leq \td_P(A)+1$ and by Lemma~\ref{lem:primal_norm} we see that $N \leq 4n \cdot g'(\|A\|_\infty, \td_P(A))$, where $g'$ is a tower of exponentials of height $\Oh(\ttd(F))$ (see~Lemma~\ref{lem:primal_norm}).
	
	Solving each auxiliary instance using Lemma~\ref{lem:nearlylinear_primal_aux} and using $F_I$ as a $\td$-decomposition of $A_I$ takes time $P \cdot (3g_\infty(A_I))^{2\height(F_I))} n \left(\log \|\bar{\veu}_i-\bar{\vel}_i\|_\infty \cdot \log f_{\max}\right)^{\ttd(F_I)-1}$.
	Since $\height(F_I) \leq 2\height(F) = 2\td_P(A)$, we may bound $(3g_\infty(A_I))^{2\height(F_I)}$ by $g''(\|A\|_\infty, \td_P(A))$ for a computable function $g''$.
	Moreover, since $\|\bar{\veu}_i-\bar{\vel}_i\|_\infty \leq 2N \leq 8ng'(\|A\|_\infty, \td_P(A))$ and by $\ttd(F_I) = \ttd(F)$, we have that 
	\[(\log \|\bar{\veu}_i-\bar{\vel}_i\|_\infty)^{\ttd(F_I)-1} \leq \left(\log 8ng'(\|A\|_\infty, \td_P(A))\right)^{\ttd(F)-1} \enspace .
	\]
	By the bound~\eqref{eq:lognk},
	\[
	\left(\log 8ng'(\|A\|_\infty, \td_P(A))\right)^{\ttd(F)-1} \leq 2^{\ttd(F)^2} \left(n g'(\|A\|_\infty, \td_P(A)\right)^{\oh(1)} \enspace .
	\]
	Moreover, when $f$ is a linear function, centering the instance (Lemma~\ref{lem:centered}) means that we may subtract the constant $\vew \vev$ and for each instance solved by the scaling algorithm bound
	\[
	(\log f_{\max})^{\ttd(F)-1} = \log(\|\vew\|_\infty N)^{\ttd(F)-1} \leq 2^{\ttd(F)^2} \left(\|\vew\|_\infty n g'(\|A\|_\infty, \td_P(A)\right)^{\oh(1)} \enspace .
	\]
	Setting $g(\|A\|_\infty, \td_P(A)) \df P \cdot g'(\|A\|_\infty, \td_P(A))$
	yields the claimed bounds of the Theorem,
	and correctness follows from Corollary~\ref{cor:scaling_algo}.
\end{proof}

\subsubsection{Dual Treedepth}
Recall that the \FPT solvability of~\eqref{IP} parameterized by $\td_D(A)+\|A\|_\infty$ follows from Lemma~\ref{lem:dual} which says that~\eqref{AugIP} can be solved in time $(\|A\|_\infty \cdot g_1(A))^{\Oh(\td_D(A))} n$.
In this subsection we will speed up this result using the observation that the halfling augmentation procedure solves a sequence of~\eqref{AugIP} instances which have very similar lower and upper bounds because each augmenting step has small support.
Our main goal is to prove the following:
\begin{theorem}[Nearly linear $\td_D(A)$]\label{thm:nearlylinear_dual}
	There is an algorithm realizing the halfling augmentation procedure for~\eqref{IP} in time $\AAap(\|\veu-\vel\|_\infty, f_{\max}) \leq (\|A\|_\infty g_1(A))^{\Oh(\td_D(A))} n \log n \log (f_{\max}) \log(\|\veu-\vel\|_\infty)$.
\end{theorem}
Specifically, let $\vex \in \Z^n$ be a feasible solution to~\eqref{IP}, $\lambda \in \N$ be a step length, and $\veh \in \ker_{\Z}(A)$ with $\mathopen|\suppo(\veh)| \leq \sigma$.
Then the~\eqref{AugIP} instance $(\vex, \lambda)$ and the~\eqref{AugIP} instance $(\vex + \veh, \lambda)$ are identical up to at most $\sigma$ coordinates of their lower and upper bounds and we call them $\sigma$-similar, using the following definition.
Two instances $(\lambda_1, \vex_1)$ and $(\lambda_2, \vex_2)$ of~\eqref{AugIP} are defined as \emph{$\sigma$-similar} if $\lambda_1 = \lambda_2$ and $\mathopen|\suppo(\vex_1 - \vex_2)| \leq \sigma$.
We shall construct a data structure called a ``convolution tree'' which maintains a representation of an~\eqref{AugIP} instance, takes linear time to initialize, and takes time roughly $\sigma \log n$ to update to represent a $\sigma$-similar~\eqref{AugIP} instance.

\paragraph*{Convolution Tree}
We will need the following notion:
\begin{definition}[Convolution]
	Given a set $R \subseteq \Z^\delta$ and tuples $\vealpha=\left(\alpha_{\ver}\right)_{\ver \in R}, \vebeta = \left(\beta_{\ver}\right)_{\ver \in R} \in \left(\Z \cup \{+\infty\}\right)^R$, a tuple $\vegamma = \left(\gamma_{\ver}\right)_{\ver \in R} \in \left(\Z \cup \{+\infty\}\right)^R$ is the \emph{convolution of $\vealpha$ and $\vebeta$}, denoted $\vegamma = \convol(\vealpha, \vebeta)$, if
	\[\gamma_{\ver} = \min_{\substack{\ver', \ver'' \in R\\\ver' + \ver'' = \ver}} \alpha_{\ver'} + \beta_{\ver''} \qquad \forall \ver \in R \enspace .\]
	A tuple of pairs $w(\vegamma)_{\vealpha, \vebeta} \in (R\times R)^R$ is called a \emph{witness of $\vegamma$ w.r.t. $\vealpha, \vebeta$} if
	\[
	w(\vegamma)_{\vealpha, \vebeta}(\ver) = (\ver',\ver'') \Leftrightarrow \gamma_{\ver} = \alpha_{\ver'} + \beta_{\ver''} \qquad \forall \ver \in R \enspace .
	\]
\end{definition}

We note that, as in the rest of this paper, the separable convex functions which appear in the following definition are represented by comparison oracles.
\begin{definition}[Convolution Tree] \label{def:convol_tree}
	Let $A \in \Z^{m \times n}$, $F$ be a $\td$-decomposition of $G_D(A)$, $R \df [-\rho \|A\|_\infty, \rho \|A\|_\infty]^{k_1(F)}$ and $R' \df R \times [0,g_1(A)]$.
	A \emph{convolution tree} is a data structure $\T$ which stores two vectors $\vel_{\T}, \veu_{\T} \in (\Z \cup \{\pm \infty\})^n$ and a separable convex function $f_{\T}\colon \R^n \to \R$, and we call $\vel_{\T}, \veu_{\T}, f_{\T}$ the \emph{state of $\T$}.
	A convolution tree  $\T$ supports the following operations:
	\begin{enumerate}
		\item \textsc{Init}$(\vel_{\T}, \veu_{\T}, f_{\T})$ initializes $\T$ to be in state $\vel_{\T}, \veu_{\T}, f_{\T}$.
		\item \textsc{Update}$(i,l_i,u_i,f_i)$ is defined for $i \in [n]$, $l_i,u_i \in \Z$ and a univariate convex function $f_i: \R \to \R$.
		Calling \textsc{Update}$(i, l_i, u_i, f_i)$ changes the $i$-th coordinates of the state into $l_i, u_i$ and $f_i$, i.e.,
		if $\vel_{\T}, \veu_{\T}, f_{\T}$ is the state of $\T$ before calling \textsc{Update}$(i, l_i, u_i, f_i)$ and $\vel'_{\T}, \veu'_{\T}, f'_{\T}$ is the state of $\T$ afterwards, then for all $j \in [n] \setminus \{i\}$, $(\vel'_{\T})_j=(\vel_{\T})_j$, $(\veu'_{\T})_j=(\veu_{\T})_j$, and $(f'_{\T})_j=(f_{\T})_j$, and $(\vel'_{\T})_i = l_i$, $(\veu'_{\T})_i = u_i$, and $(f'_{\T})_i = f_i$.
		\item for each $\sigma \in [n]$, $\sigma$-\textsc{Update}$(U,\vel_U, \veu_U, f_U)$ is defined for $U \subseteq [n]$ with $|U|=\sigma$, $\vel_U, \veu_U \in \Z^U$ and a separable convex function $f_U: \R^U \to \R$. Calling $\sigma$-\textsc{Update}$(U,\vel_U, \veu_U, f_U)$ is equivalent to calling \textsc{Update}$(i,(\vel_U)_i,(\veu_U)_i,(f_U)_i)$ for each $i \in U$ (in increasing order of indices).
		\item \label{it:conv_tree_veG}\textsc{Query} returns a sequence $\veG \in (\Z^{n} \cup \{\undffd\})^{R'}$ where  $\veg_{\ver,\rho} \df (\veG)_{\ver,\rho}$ is a solution of
		\begin{equation}
		\min \left\{f(\veg) \mid A\veg = (\ver, \vezero), \, \vel_{\T} \leq \veg \leq \veu_{\T}, \, \veg \in \Z^n, \|\veg\|_1 = \rho\right\}, \label{eq:convol_tree_subproblem}
		\end{equation}
		for each $\ver'=(\ver,\rho) \in R'$, with $\vezero$ the $\left(m-k_1(F)\right)$-dimensional zero vector, and $\vel_{\T}, \veu_{\T}, f_{\T}$ the current state of $\T$.
		If~\eqref{eq:convol_tree_subproblem} has no solution, we define its solution to be $\undffd$.
	\end{enumerate}
\end{definition}
In Definition~\ref{def:convol_tree}, the set $R$ represents the set of possible right hand sides of~\eqref{eq:convol_tree_subproblem}, and the extra coordinate in the set $R'$ serves to represent the set of considered $\ell_1$-norms of solutions of~\eqref{eq:convol_tree_subproblem}.
Note that we allow a zero $\ell_1$-norm.

\begin{lemma}[Convolution Tree Lemma] \label{lem:convol_tree}
	Let $A,F$, and $\T$ be as in Definition~\ref{def:convol_tree}.
	\begin{enumerate}
		\item \textsc{Init}$(\vel_{\T}, \veu_{\T}, f_{\T})$ can be realized in time $(2\|A\|_\infty \cdot g_1(A) + 1)^{2\height(F)+1} \cdot 2n$,
		\item \textsc{Update}$(i,l_i,u_i,f_i)$ can be realized in time $(2\|A\|_\infty \cdot g_1(A) + 1)^{2\height(F)+1} \Oh(\ttd(\T) \log n)$
		\item $\sigma$-\textsc{Update}$(U,\vel_U,\veu_U,f_U)$ can be realized in time $\sigma \cdot (2\|A\|_\infty \cdot g_1(A) + 1)^{2\height(F)+1} \Oh(\ttd(\T) \log n)$,
		\item \textsc{Query} can be realized in constant time.
	\end{enumerate}
\end{lemma}
Given Lemma~\ref{lem:convol_tree} whose proof we postpone, we prove Theorem~\ref{thm:nearlylinear_dual} as follows.
\begin{proof}[Proof of Theorem~\ref{thm:nearlylinear_dual}]
	Let $\Gamma_2 = \{1, 2, 4, \dots, 2^{\ceil{\log \|\veu-\vel\|_\infty}}\}$, let $\vex_0 \in \Z^n$ be a given initial solution, and set $\veh_0 \df \vezero$ and $U_0 := \emptyset$.
	In each iteration $i \geq 1$, $\veh_i$ denotes a halfling for $\vex_{i-1}$ with $\|\veh_i\|_1 \leq g_1(A)$, $U_i = \suppo(\veh_i)$ denotes the set of non-zero coordinates of $\veh_i$, and $\vex_i = \vex_{i-1} + \veh_i$.
	Moreover, we represent $\veh_i$ compactly as coordinate-value pairs, which means that its encoding length does not depend on the dimension $n$.
	This also means that, given $\veh_i$, computing $\vex_i = \vex_{i-1} + \veh_i$ only takes $\Oh(\mathopen|\suppo(\veh_i)|)$ arithmetic operations because it suffices to change the coordinates $U_i$ of $\vex_{i-1}$.
	
	At the beginning of iteration $i \in \N$ we distinguish two cases.
	If $i=0$, then, for each $\lambda \in \Gamma_2$, consider the substitution described in the proof of Lemma~\ref{lem:primal}, i.e.,
	$\vel_\lambda \df \ceil{(\vel - \vex_i)/\lambda}$, $\veu_\lambda \df \floor{(\veu - \vex_i)/\lambda}$, and $f_{\lambda}(\veg) \df f(\vex_i + \lambda \veg)$.
	Let $\T_\lambda$ be a new convolution tree, and call \textsc{Init}$(\vel_\lambda, \veu_\lambda, f_\lambda)$ on $\T_\lambda$.
	Otherwise, when $i \geq 1$, for each $\lambda \in \Gamma_2$, let $U \df U_{i-1}$, $\sigma := |U|$, $\vel_U$, $\veu_U$ and $f_U$ be a restriction of $\vel_\lambda \df \ceil{(\vel - \vex_i)/\lambda}$, $\veu_\lambda \df \floor{(\veu - \vex_i)/\lambda}$, and $f_{\lambda}(\veg) \df f(\vex_i + \lambda \veg)$ to the coordinates $U$, respectively, and call $\sigma$-\textsc{Update}$(U, \vel_U, \veu_U, f_U)$ on $\T_\lambda$.
	Note that an evaluation oracle for $f_{\lambda}$ is easily constructed from an evaluation oracle for $f$: when queried on $\veg$, return $f(\vex_i + \lambda \veg)$.
	Also note that computing $\vel_U, \veu_U$ and $f_U$ can be done with $\Oh(\sigma)$ arithmetic operations.
	
	Then, we obtain a halfling for $\vex_i$ as follows.
	For each $\lambda \in \Gamma_2$, query $\T_\lambda$ obtaining a sequence $\veG_\lambda$.
	For each $(\ver,\rho) \in R'$ denote by $\veg_{\ver,\rho,\lambda} \df (\veG_\lambda)_{\ver,\rho}$.
	Let $\rho^*_{(i,\lambda)} \df \arg \min_{\rho' \in [0,g_1(A)]} f(\vex_i + \lambda \veg_{\vezero,\rho',\lambda})$, and let $\veh_{i, \lambda} \df \veg_{\vezero, \rho^*_{(i,\lambda)}, \lambda}$.
	The idea of this definition is that augmenting step pairs $(\lambda, \veg)$ for $\vex_{i}$ with $\ell_1$-norm at most $g_1(A)$ correspond to solutions of~\eqref{eq:convol_tree_subproblem} with $\ver=\vezero$, the best one has $\ell_1$-norm exactly $\rho^*_{(i,\lambda)}$, and 
	it is specifically $\veh_{i, \lambda}$.
	Thus $\veh_{i, \lambda}$ is a solution of the $(\vex_i,\lambda)$ instance of~\eqref{AugIP}.
	Let $\lambda_i^* \in \arg \min_{\lambda \in \Gamma_2} f(\vex_i + \lambda \veh_{i,\lambda})$ (in particular, we break ties in the ``$\arg\min$'' arbitrarily) and let $\veh_i = \lambda_i^* \veh_{i, \lambda_i^*}$.
	By Lemma~\ref{lem:powersoftwo}, $\veh_i$ is a halfling for $\vex_i$.
	If $f(\vex_i + \veh_i) < f(\vex_i)$, set $\vex_{i+1} := \vex_i + \veh_i$ and $i:= i+1$, and otherwise return $\vex_i$ as optimal.
	
	Let us compute the complexity.
	We are maintaining $|\Gamma_2| \leq \log \|\veu-\vel\|_\infty + 1$ convolution trees, with each taking $(\|A\|_\infty \cdot g_1(A))^{\Oh(\td_D(A))} n$ time to initialize in iteration $i=0$.
By Lemma~\ref{lem:halfling} the number of iterations is $n' \leq 3n \log (f_{\max})$.
	Because for each $i \in [0,n']$ it holds that $|\suppo(\vex_{i+1} - \vex_{i})| = |\suppo(\veh_i)| \leq g_1(A)$ due to the ``$\|\veg\|_1 = \rho$'' constraint in problem~\eqref{eq:convol_tree_subproblem} and the fact that $\rho \leq g_1(A)$, updating one convolution tree takes time $(\|A\|_\infty \cdot g_1(A))^{\Oh(\td_D(A))} \sigma \log n$ with $\sigma = g_1(A)$.
	Because $n' \leq 3n \log (f_{\max})$ we update each tree at most $3n  \log (f_{\max})$ times, in total taking time
	\begin{multline*}
	\log \|\veu-\vel\|_\infty \cdot \left(\|A\|_\infty \cdot g_1(A)\right)^{\Oh(\td_D(A))} \left(n + 3n\cdot \log (f_{\max}) \cdot g_1(A) \cdot \log n\right) \\
	\qquad \qquad \leq \left(\|A\|_\infty \cdot g_1(A)\right)^{\Oh(\td_D(A))} n \log n \log (f_{\max}) \log \|\veu-\vel\|_\infty \enspace . \qedhere
	\end{multline*}
\end{proof}

\begin{proof}[Proof of Lemma~\ref{lem:convol_tree}]
	The proof is similar to the proof of Lemma~\ref{lem:dual}.
	We define $\T$ recursively over $\ttd(F)$, then describe how the operations are realized, and finally analyze the time complexity.
	If $\ttd(F) \geq 2$, we assume $A$ is dual block-structured along $F$ (otherwise apply Corollary~\ref{cor:dual_decomp}) and we have, for every $i \in [d]$, matrices $A_i, \bar{A}_i, \hat{A}_i$ and a tree $\hat{F}_i$ (¨see Lemma~\ref{lem:fhati}) with the claimed properties, and a corresponding partitioning of $\veb, \vel, \veu, \veg$ and $f$.
	If $\ttd(F) = 1$, let $d \df n$, $\bar{A}_i \df A_{\bullet, i}$ for each $i \in [d]$, and $A_1, \dots, A_d$ be empty.
	For $i,j \in [d]$, $i \leq j$, denote by $A[i,j]$ the submatrix of $A$ induced by the rows and columns of the blocks $\bar{A}_i, \dots, \bar{A}_j$, $\vel_{\T}[i,j] \df (\vel^i_{\T}, \dots, \vel^j_{\T})$, $\veu_{\T}[i,j] \df (\veu^i_{\T}, \dots, \veu^j_{\T})$, and $f_{\T}[i,j]$ be the restriction of $f_{\T}$ to the coordinates of $\vel_{\T}[i,j]$.
	Observe that $\hat{A}_i = A[i,i] = \left(\begin{smallmatrix}\bar{A}_i \\ A_i\end{smallmatrix}\right)$.
	
	We obtain $\T$ by first defining a convolution tree $\T_i$ for each $i \in [d]$, where $\T_i$ is a convolution tree for the matrix $\hat{A}_i$ and for $\vel_{\T}^i, \veu_{\T}^i$ and $f_{\T}^i$, and then constructing a binary tree $T$ whose leaves are the $\T_i$'s and which is used to join the results of the $\T_i$'s in order to realize the operations of $\T$.
	Since, for each $i \in [d]$, $\T_i$ is supposed to be a convolution tree for a matrix $\hat{A}_i$ with $\ttd(\hat{F}_i) < \ttd(F)$, if $\ttd(F) \geq 2$, we will construct $\T_i$ by a recursive application of the procedure which will be described in the text starting from the next paragraph.
	Now we describe how to obtain $\T_i$ if $\ttd(F)=1$, i.e., when $\hat{A}_i$ is just one column.
	When $\T_i$ is initialized or updated, we construct the sequence $\veG_i$ by using the following procedure.
	For each $(\ver, \rho) \in R'$, defining $\veg^i$ to be such $\veg^i \in \left(\{-\rho, \rho\} \cap [\vel_{\T}^i, \veu_{\T}^i]\right)$ which satisfies $\hat{A}_i \veg^i = (\ver, \vezero)$ and minimizes $f_{\T}^i(\veg^i)$, and returning either $\veg^i$ if it is defined or $\undffd$ if no such $\veg^i$ exists.
	Notice that $\veg^i, \vel^i_{\T}$ and $\veu^i_{\T}$ are scalars.
	
	We say that a rooted binary tree is \emph{full} if each vertex has $0$ or $2$ children, and that it is \emph{balanced} if its height is at most $\log |T| + 2$.
	It is easy to see that for any number $h$ there exists a rooted balanced full binary tree with $h$ leaves (hence of height at most $3+\log h$).
	Let $T$ be a rooted balanced full binary tree with $d$ leaves labeled by the singletons $\{1\}, \dots, \{d\}$ whose internal vertices are labeled as follows: if $u \in T$ has children $v,w$, then $u=v \cup w$; hence the root $r$ satisfies $r=[d]$ and the labels are subsets of consecutive indices $1, \dots, d$.
	We obtain $\T$ by identifying the leaves of $T$ with the roots of the trees $\T_i$ and will explain how to use $T$ to join the results of all $\T_i$'s in order to realize the operations of $\T$.
	
	To initialize $\T$ with vectors $\vel_\T$ and $\veu_\T$ and a function $f_{\T}$, we shall compute, in a bottom-up fashion, the sequence $\veG$ described in point~\ref{it:conv_tree_veG} of Definition~\ref{def:convol_tree}.
	Let $u$ be a node of $T$ and let $i \df \min u$ and $j \df \max u$ be the leftmost and rightmost leaves of $T_u$ (i.e., the subtree of $T$ rooted at $u$), respectively.
	Moreover, if $u$ is an internal node, let its left and right child be $v$ and $w$, respectively, and let $k \df \max v$ be the rightmost leaf of $T_v$.
	Consider the following auxiliary problem, which is intuitively problem~\eqref{eq:convol_tree_subproblem} restricted to blocks $i$ to $j$: simply append ``$[i,j]$'' to all relevant objects, namely $f_{\T}, \veg, A, \vel_{\T}$ and $\veu_{\T}$:
	\begin{multline}
	\min \big\{f_{\T}[i,j](\veg[i,j]) \mid A[i,j] \veg[i,j] = (\ver, \vezero), \, \vel_{\T}[i,j] \leq \veg[i,j] \leq \veu_{\T}[i,j], \\
	\veg[i,j] \in \Z^{n_i + \cdots + n_j}, \, \|\veg[i,j]\|_1 = \rho \big\}, \label{eq:conv_tree_aux}
	\end{multline}
	where $\vezero$ has dimension $\sum_{\ell=i}^j m_{\ell}$.
	Let $\veg^u_{\ver, \rho}$ be a solution of~\eqref{eq:conv_tree_aux} and let $\veG^u$ be the sequence $\left(\veg^u_{\ver'}\right)_{\ver' \in R'}$.
	
	If $u$ is a leaf of $T$, the sequence $\veG^u$ is obtained by querying $\T_u$ (which was defined previously).
	Otherwise, compute the sequences $\veG^v$ and $\veG^w$ for the children $v,w$ of $u$, respectively.
	Then, we compute a convolution $\vegamma^u$ of sequences, $\zeta \in \{v,w\}$, $\vegamma^\zeta$ obtained from $\veG^\zeta$ by setting $\gamma^\zeta_{\ver'} \df f_{\T}[\min \zeta, \max \zeta](\veG^\zeta_{\ver'})$, where $f_{\T}[\min \zeta, \max \zeta](\undffd) \df +\infty$.
	We also compute a witness $w(\vegamma^u)_{\vegamma^v, \vegamma^w}$ of $\vegamma^u$.
	The desired sequence $\veG^u$ is easily obtained from $w(\vegamma^u)_{\vegamma^v, \vegamma^w}$.
	
	The \textsc{Update} operation is realized as follows.
	Let $i, l_i, u_i, f_i$ be as described in Definition~\ref{def:convol_tree}.
	Let $\iota(i) \in [d]$ be the index of the block containing coordinate $i$, and let $i' \in [n_i]$ be the coordinate of block $\iota(i)$ corresponding to $i$.
	First traverse $T$ downward from the root to leaf $\iota(i)$, call \textsc{Update}$(i', l_i, u_i, f_i)$ on subtree $\T_{\iota(i)}$, and then recompute convolutions $\vegamma^w$ and sequences $\veG^w$ for each vertex $w$ on a root-leaf path between $\iota(i)$ and the root $r$.
	The $\sigma$-\textsc{Update} operation is realized by simply calling \textsc{Update} for each $i \in U$, in increasing order of $i$.
	
	Let us analyze the time complexity.
	The convolution tree $\T$ is composed of $\ttd(F)$ levels of smaller convolution trees, each of which has height $\Oh(\log n)$.
	Specifically, the topmost level $\ttd(F)$ consists of the nodes corresponding to the internal vertices of $T$, and the previous levels are defined analogously by recursion (e.g., level $\ttd(F)-1$ consists of the union of the topmost levels of $\T_i$ over all $i \in [d]$, etc.).
	Thus, $\height(\T) \in \Oh (\ttd(F) \log(n))$.
	
	There are $n$ leaves of $\T$, one for each column of $A$.
	The initialization of leaves takes time at most $n \cdot (2\rho+1)$.
	Let $N_\ell$, $\ell \in [\ttd(F)]$, denote the number of internal nodes at level $\ell$.
	Because a full binary tree has at most as many internal nodes as it has leaves, we see that $\sum_{\ell=1}^{\ttd(F)} N_\ell \leq n$.
	Then, initializing level $\ttd(F)$ amounts to solving $N_{\ttd(F)}$ convolutions and their witnesses, each of which is computable in time $|R'|^2 \leq \left((2g_1(A) + 1)\cdot (2\|A\|_\infty \cdot g_1(A) + 1)\right)^{2\height(F)} \leq (2\|A\|_\infty \cdot g_1(A) + 1)^{2\height(F)+1}$.
	Processing level $\ell \in [\ttd(F)-1]$ amounts to solving $N_\ell$ convolutions with sets $R'(\ell) = R(\ell) \times [0,g_1(A)]$ where $R(\ell)$ is obtained from $R$ by dropping some coordinates, and thus $|R(\ell)| \leq |R|$, $|R'(\ell)| \leq |R'|$, and hence computing one convolution can be done in time $|R'(\ell)|^2 \leq (2\|A\|_\infty \cdot g_1(A) + 1)^{2\height(F)+1}$.
	In total, initialization takes time at most $2n (2\|A\|_\infty \cdot g_1(A) + 1)^{2\height(F)+1}$.
	
	Regarding the \textsc{Update} operation, there is one leaf of $\T$ corresponding to the coordinate $i$ being changed.
	Thus, in order to update the results, we need to recompute all convolutions corresponding to internal nodes along the path from the leaf to the root.
	Because $\height(\T) \leq \Oh(\log n)$, the number of such nodes is $\Oh(\log n)$, with each taking time at most $(2\|A\|_\infty \cdot g_1(A) + 1)^{2\height(F)+1}$.
	The time required by $\sigma$-\textsc{Update} is $\sigma$ times the time required by one \textsc{Update} operation.
\end{proof}

\lv{\begin{remark}
		TODO another observation: there is approximate convolution, could be useful for practice.
\end{remark}}

\subsubsection{Primal and Dual Treewidth} 
We also briefly consider the more permissive graph parameter treewidth:
\begin{definition}[Treewidth]
	A \emph{tree decomposition} of a graph $G=(V,E)$
	is a pair $(T, B)$, where $T$ is a tree and $B$ is a mapping
	$B: V(T) \rightarrow 2^V$ satisfying
	\begin{itemize}
		\item for any $uv \in E$, there exists $a \in V(T)$ such that
		$u, v \in B(a)$,
		\item if $v \in B(a)$ and $v \in B(b)$, then $v \in B(c)$ for all
		$c$ on the path from $a$ to $b$ in $T$.
	\end{itemize}
	We use the convention that the vertices of the tree are called \emph{nodes} and the sets
	$B(a)$ are called \emph{bags}.
	The {\em treewidth $\tw((T, B))$ of a tree decomposition} $(T, B)$ is
	the size of the largest bag of $(T, B)$ minus one.
	The {\em treewidth $\tw(G)$ of a graph} $G$ is the
	minimum treewidth over all possible tree decompositions of $G$.
	A \emph{path decomposition} is a tree decomposition in which $T$ is a path.
\end{definition}
Analogously to treedepth, we denote by $\tw_P(A) = \tw(G_P(A))$ and $\tw_D(A) = \tw(G_D(A))$.
We have that for any graph $G$, $\tw(G) \leq \td(G)$.

The point of this section is to prove Lemmas~\ref{lem:primal_treewidth} and~\ref{lem:dual_treewidth} which show that, given bounds on $g_\infty(A)$ and $g_1(A)$, problem~\eqref{AugIP} is efficiently solvable when $\tw_P(A)$ and $\tw_D(A)$ are small, respectively.
The reason we have focused on the more restrictive treedepth so far is that, in general, even if $\tw_P(A), \tw_D(A)$ and $\|A\|_\infty$ are bounded by a constant, $g_\infty(A)$ can be exponential in $n$ (Lemma~\ref{lem:largeg1}).
We still need a few notions and bounds before we can state and prove the key lemmas.

The \emph{incidence graph of $A$} is $G_I(A) = (V_I, E_I)$ with $V_I \df \{v_i \mid i \in [n]\} \cup \{c_j \mid j \in [m]\}$ and $E_I \df \left\{ \{v_i, c_j\} \mid A_{i,j} \neq 0, \, i \in [n], j \in [m] \right\}$.
To gain some intuition for how the primal, dual, and incidence graphs are related, consider the following.
For a graph $G$ let $G^2$ denote the \emph{square of $G$} which is obtained from $G$ by adding an edge between all vertices in distance $2$.
For a subset of vertices $W \subseteq V(G)$, we denote by $G[W]$ the subgraph of $G$ induced by $W$.
It is easy to see that $G_P(A) = G_I(A)^2[\{v_i \mid i \in [n]\}]$ and $G_D(A) = G_I(A)^2[\{c_j \mid j \in [m]\}]$.
We define the \emph{incidence treewidth of $A$} to be $\tw_I(A) \df \tw(G_I(A))$.
The following bound will be useful later.
\begin{lemma}[{Kolaitis and Vardi~\cite{KolaitisV:2000}}] \label{lem:inc_prim_tw}
	$\tw_I(A) \leq \tw_P(A) + 1$ and $\tw_I(A) \leq \tw_D(A) + 1$.
\end{lemma}

\begin{proof}
	Construct a tree decomposition $T'$ of $G_I(A)$ from an optimal tree decomposition $T$ of $G_P(A)$ as follows.
	Consider a row $\vea_i$ of $A$: its non-zeros correspond to a clique of columns in $G_P(A)$, so there must exist a bag of $T$ containing all of them; now add the vertex corresponding to $\vea_i$ to this bag.
	Repeating this for all rows and possibly copying bags obtains $T'$ of width at most one larger than $T$.
	The statement for $G_D(A)$ follows by the observation that $G_I(A^\intercal) = G_I(A)$.
\end{proof}

\begin{lemma} \label{lem:primal_treewidth}
	Problem~\eqref{AugIP} can be solved in time $(2g_\infty(A)+1)^{\Oh(\tw_P(A))}n$.
\end{lemma}

\begin{proof}
	The algorithm follows from Freuder's algorithm:
	\begin{proposition}[Freuder~{\cite{Freu,JK}}] \label{prop:primal_treewidth}
		\eqref{IP} can be solved in time $\|\veu - \vel\|_\infty^{O(\tw_P(A))} \cdot n$.
	\end{proposition}
	Replacing $\G(A)\best$ by $B_\infty(g_\infty(A))\best$ in~\eqref{AugIP} yields a subproblem with bounds $\bar{\vel}, \bar{\veu}$ satisfying $\|\bar{\veu} - \bar{\vel}\|_\infty \leq 2 g_\infty(A)+1$ and thus is solvable by Proposition~\ref{prop:primal_treewidth} in the claimed time.
\end{proof}

\begin{lemma} \label{lem:dual_treewidth}
	Problem~\eqref{AugIP} can be solved in time $(2\|A\|_\infty g_1(A))^{\Oh(\tw_D(A))}n$.
\end{lemma}

\begin{proof}
	The algorithm follows from a recent result of Ganian et al.~\cite{GOR}:
	\begin{proposition}[Ganian et al.~{\cite[Theorem 6]{GOR}}] \label{prop:incidence_treewidth}
		\eqref{IP} can be solved in time $\Gamma^{O(\tw_I(A))} \cdot n$,
	\end{proposition}
	In other words, the parameter $\Gamma$ is bounding the largest number in absolute value which appears in any prefix sum of $A \vex$ for any feasible solution $\vex$.
	
	Our goal is to use Proposition~\ref{prop:incidence_treewidth} to solve~\eqref{AugIP}, where, by $\G(A) \subseteq B_1(g_1(A))$, we may replace $\G(A)\best$ with $B_1(g_1(A))\best$.
	Given an instance $(\lambda,\vex)$ of~\eqref{AugIP}, a solution $\veg$ of
	\begin{equation}
	\min f(\vex + \lambda\veg):\, A\veg=\vezero, \, \vel \leq \vex + \lambda\veg \leq \veu, \, \|\veg\|_1 \leq g_1(A), \, \veg \in \Z^n \label{eq:dualtw_subprob}
	\end{equation}
	is certainly also a solution of~\eqref{AugIP}.
	In order to use Proposition~\ref{prop:incidence_treewidth} to solve~\eqref{eq:dualtw_subprob} our only task is to replace the nonlinear constraint $\|\veg\|_1 \leq g_1(A)$ with a linear constraint.
	This is easy by splitting every variable $g_i$ into its positive and negative part $g_i^+$ and $g_i^-$ which we force to be nonnegative by setting $g_i^+, g_i^- \geq 0$.
	Correspondingly, every column $A_i$ of $A$ is now split into $A_i^+ = A_i$ and $A_i^- = -A_i$.
	The bounds $l_i \leq x_i + \lambda g_i \leq u_i$ are rewritten to $l_i \leq x_i + (\lambda g_i^+) - (\lambda g_i^-) \leq u_i$.
	Finally, $\|\veg\|_1 \leq g_1(A)$ in~\eqref{eq:dualtw_subprob} is equivalent to $\sum_{i=1}^n (g_i^+ + g_i^-) \leq g_1(A)$, and we additionally set $0 \leq g_i^-, g_i^+ \leq g_1(A)$.
	It is easy to observe that this auxiliary problem has incidence treewidth at most $2\tw_I(A) + 2$: replace $g_i$ in each bag by $g_i^+, g_i^-$, add the constraint $\sum_{i=1}^n (g_i^+ + g_i^-) \leq g_1(A)$ into each bag, and add one of $l_i \leq x_i + (\lambda g_i^+) - (\lambda g_i^-)$ or $x_i + (\lambda g_i^+) - (\lambda g_i^-) \leq u_i$ for at most one $i$ for each bag, perhaps for multiple copies of the original bag.
	Moreover, by the fact that $\|\veg\|_1 \leq g_1(A)$, we have that $\Gamma \leq \|A\|_\infty g_1(A)$.
	The claim follows.
\end{proof}

\lv{
	\subsection{Memoization of Subproblems}
	\begin{itemize}
		\item because we are looking for a solution in a small box (by bounds on $g_1(A)$ and $g_\infty(A)$, there are only "few" objectives.
		\item show that the structure of DPs is such that we only solve for "few" different right hand sides
		\item in conclusion: only few $\vel, \veu$ (by $g_1(A)$ or $g_\infty(A)$ bounds), only few $\veb$ by structure of recursion/DP, only few $f$ by reducibility $\implies$ only few (indep of $n$) leaf instances, can hash them and get a time complexity of form $f(\td(A)) + n \poly\log n$ instead of multiplicative.
		\item this can be done in a lazy way except for the keeping of a "canonical" $f$ (we'd need to find it and that's extra work).
	\end{itemize}
}

\subsection{Summarizing Lemma} \label{sec:summarizing}
By combining the various ingredients of the framework we have developed in the previous sections we obtain several different algorithms with different bounds corresponding to:
\begin{itemize}
	\item (not) using a relaxation oracle $\RRR$ to reduce bounds and the right hand side,
	\item (not) using the scaling algorithm (Corollary~\ref{cor:scaling_algo}) to obtain ``semi-strongly polynomial'' algorithms, i.e., algorithms whose time complexity depends on $\vel, \veu$ but not $f$.
	\item (not) using the reducibility bounds $\rho$ (Theorem~\ref{thm:lin_red} and Corollary~\ref{cor:sepconvex_red}) to replace the $\log f_{\max}$ term.
\end{itemize}
Even though there are three boolean options, only four choices are sensible for the following reasons.
First, reducibility bounds are only useful when dealing with instances whose lower and upper bounds are polynomial in $n$, i.e., when we use either the scaling algorithm or $\RRR$.
Second, using $\RRR$ essentially only makes sense when the reducibility bounds $\rho$ are used afterwards in order to confine the possible dependence on $\vel, \veu, \veb, f$ to the time complexity of $\RRR$.
Thus the remaining sensible settings are the following:
\begin{description}[align=right,labelwidth=2cm]
	\item[--, --] no relaxation, no scaling, no reducibility bounds.
	\item[scaling, --] scaling algorithm, no reducibility bounds.
	\item[scaling, $\rho$] scaling algorithm, apply reducibility bounds.
	\item[$\RRR$, $\rho$] use $\RRR$, apply reducibility bounds.
\end{description}
The next lemma shows the complexity bounds which can be derived from a ``base'' bound by applying the above described approaches:
\begin{lemma}[Summarizing Lemma] \label{lem:summarizing}
	Assume there is an algorithm solving any~\eqref{IP} in time at most 
	\[g \cdot T_n(n) \cdot T_{\domain}(\domain) \cdot T_{\text{obj}}(\log f_{\max}),\]
	where $g_P, g_D$ are classes of computable functions, $T_n, T_{\domain}$ and $T_{\text{obj}}$ are functions bounded by a polynomial, $\domain \df \log \|\veu-\vel, \veb\|_\infty$, and $g$ is either $g_P \df g_P(\|A\|_\infty, \td_P(A))$ or $g_D \df g_D(\|A\|_\infty, \td_D(A))$.
	Let $g \in \Omega(\min\{g_\infty(A), g_1(A)\})$ and $g^{\Oh(1)} \in g$.
	Let $\RRR$ be an approximate relaxation oracle for $A$.
	
	Then there exist three algorithms, one per row, solving every centered ($\vezero \in [\vel, \veu]$) \eqref{IP} instance $\II$ with finite bounds in time
	\begin{center}
		\begin{tabular}{lcc}  
			\toprule
			& $f$ linear & $f$ separable convex\\
			\midrule
			scaling, -- & $g \cdot T_n(n) \cdot T_{\domain}(\log n) \domain \cdot T_{\text{obj}}(\log n\|\vew\|_\infty)$ & $g \cdot T_n(n) \cdot T_{\domain}(\log n) \domain \cdot T_{\text{obj}}(\log f_{\max})$ \\
			scaling, $\rho$ & $g \cdot T_n(n) \cdot T_{\domain}(\log n) \domain \cdot T_{\text{obj}}(n \log n)$& $g \cdot T_n(n) \cdot T_{\domain}(\log n) \domain \cdot T_{\text{obj}}(n^2 \log n)$\\
			$\RRR$, $\rho$ & $\RR(\II,n) + g \cdot T_n(n) \cdot T_{\domain}(\log n) \cdot T_{\text{obj}}(n \log n)$ & $\RR(\II,n) + g \cdot T_n(n) \cdot T_{\domain}(\log n) \cdot T_{\text{obj}}(n^2 \log n)$ \\
			\bottomrule
		\end{tabular}
	\end{center}
\end{lemma}
\begin{proof}
	The ``scaling, --'' row is derived from the base bound by applying the scaling algorithm (Corollary~\ref{cor:scaling_algo}) and solving the auxiliary instances using the assumed algorithm.
	This results in $\Oh(\domain)$ auxiliary instances, each of which has bounds $\bar{\vel}, \bar{\veu}$ satisfying $\|\veu-\vel\|_\infty \leq g n$, hence $\log \|\bar{\veu}-\bar{\vel}\|_\infty = \log (gn)$.
	In the linear case this means that $\bar{f}_{\max}$ in the auxiliary instance is bounded by $\log (gn\|\vew\|_\infty)$.
	Thus the resulting time complexity contains an additional factor of $\domain$ to account for the $\Oh(\domain)$ auxiliary instances, and $T_{\domain}(\domain)$ becomes $T_{\domain}(\log n)$.
	Since $T_{\domain}$ is a polynomial, the $T_{\domain}(\log g_{\infty}(A))$ part is accounted for in $g'$.
	Next, the ``scaling, $\rho$'' row is obtained by replacing $\log f_{\max}$ by either $n \log n$ or $n^2 \log n$, depending on whether $f$ is linear or separable convex, with the bound following from Theorem~\ref{thm:lin_red} or Corollary~\ref{cor:sepconvex_red}, respectively.
	Finally, the $\RRR, \rho$ row is obtained from bound~\eqref{lem:stronglypoly_master:1} of the Master Lemma and the reducibility bounds.
\end{proof}

\section{Applications} \label{sec:apps}
In this section we will combine the results we have proved so far to obtain the currently fastest algorithm for~\eqref{IP} with small $\td_P(A)$ or $\td_D(A)$.
Furthermore, we will spell out the consequences for other classes of~\eqref{IP} which have small $\td_P(A)$ or $\td_D(A)$, such as 2-stage and multi-stage stochastic IP, and $n$-fold and tree-fold IP, and for problems which have been modeled using these classes of~\eqref{IP}.
We will also briefly mention IPs of small primal and dual treewidth, and the consequences of our results for the solvability of the corresponding fractional relaxations.

\subsection{Corollaries}
To summarize our results we use a table with the following rows and columns.
The rows distinguish bounds for linear and separable convex objective functions.
The columns correspond to the four possible considered settings: the ``base'' setting without any scaling, reduction, or relaxation, denoted ``--, --'', and the settings ``scaling, --'', ``scaling, $\rho$'', and ``$\RRR$, $\rho$'' defined in Section~\ref{sec:summarizing}.
Moreover, if one result in a table dominates another, we only write the stronger result.
Thus, the table only contains incomparable results.
For brevity, we denote by $g$ the ``main'' component of the complexity, by $\domain$ we denote $\log \|\veu-\vel,\veb\|_\infty$, and by $\RR(\II,n)$ we denote the time needed by $\RRR$ to solve the relaxation~\eqref{relax} to accuracy $n$.
The reason why $\domain$ is not just $\log \|\veu-\vel\|_\infty$ but includes a dependence on $\|\veb\|_\infty$ is the necessity of solving feasibility instances containing slack variables with bounds $[\vezero, \veb]$.
If an instance with infinite bounds needs to be handled, an additional $\log \|\vex^* - \vex_0\|_\infty$ factor is incurred (Lemma~\ref{lem:infinite_sepconv}).
Finally, let us redefine (only for the purpose of giving more concise bounds) $f_{\max} \df \max\{f_{\max}, \|\veb\|_1\}$: this is so that we can use the intuition that ``feasibility is as easy as optimization'', recalling that the objective function of the auxiliary instance~\eqref{eq:auxiliary_feasibility} has maximum value $\|\veb\|_1$.
This potentially overestimates the correct complexity in the scaling and relaxation regimes, but only by $\log \|\veb\|_1$ factors, which are dominated by other factors in all known applications.

\begin{corollary}[Primal Algorithm] \label{cor:primal}
	Let an~\eqref{IP} instance $\II$ be given with $A \in \Z^{m \times n}$, let $F$ be a $\td$-decomposition of $G_P(A)$, $\RRR$ be a relaxation oracle for $A$, $\domain \df \log \|\veu-\vel, \veb\|_\infty$, redefine $f_{\max} \df \max\{f_{\max}, \|\veb\|_1\}$, and
	\[
	g \df \stackinset{l}{61pt}{b}{-9pt}{\tiny\rotatebox{33}{$\underbrace{\kern21pt}_{\ttd(F)-1}$}}{%
		$g(\|A\|_\infty,F) \df 2^{2^{\rdots^{2^{(2\|A\|_\infty)^{\Oh\left(2^{\ttd(F)} \cdot \height(F)^2\right)}}}}}$}
	\] 
	If $f(\vex) = \vew \vex$, $\II$ is solvable in time at most
	\begin{center}
		\begin{tabular}{cccc}  
			\toprule	
			--, -- & scaling, -- & scaling, $\rho$ & $\RRR$, $\rho$ \\
			\midrule
			$g\domain n^2(\log f_{\max})$
			& $g\domain n^2\log n \cdot (\log n \|\vew\|_\infty)$ 			
			& $g\domain n^3 \log^2 n$ 
			& $\RR(\II,n)+g n^3 \log^2 n$   \\ 
			\multicolumn{2}{c}{$g\domain n^{1+o(1)}(\log\|\vew\|_\infty)^{\ttd(F)-1}$}			
			& $g\domain n^{\ttd(F)+o(1)}$ 
			& $\RR(\II,n)+gn^{\ttd(F)+o(1)}\log^2 n$   \\
			\bottomrule
		\end{tabular}
	\end{center}
	If $f$ is an arbitrary separable convex function, $\II$ can be solved in time at most
	\begin{center}
		\begin{tabular}{cccc}  
			\toprule	
			--, -- &  scaling, --  &  scaling, $\rho$  &  $\RRR$, $\rho$ \\
			\midrule
			\multicolumn{2}{c}{$g\domain n^2(\log f_{\max})$}
			& $g\domain n^4 \log^2 n$
			& $\RR(\II,n)+gn^4 \log^2 n$ \\
			\multicolumn{2}{c}{$g \domain n^{1+o(1)} (\log f_{\max})^{\ttd(F)-1}$}
			& $g\domain n^{2\ttd(F)-1+o(1)}$
			& $\RR(\II,n)+gn^{2\ttd(F)-1+o(1)}\log^2 n$\\
			\bottomrule
		\end{tabular}
	\end{center}
	The $n^{o(1)}$ term is more precisely $\log^{\ttd(F)+1} n$.
	We present it as $n^{o(1)}$ for brevity and to stress the fact that the algorithm is nearly-linear and \FPT (see inequality~\eqref{eq:lognk}).
\end{corollary}
Let us briefly elaborate on the tables of complexities above and the relationship between individual cells.
The two lines of each table correspond to our two algorithms for primal treedepth: the basic algorithm which essentially follows from the solvability of~\eqref{AugIP} by Lemma~\ref{lem:primal}, and the recursive algorithm of Theorem~\ref{thm:nearlylinear_primal}.
The ``base'' complexity is the one of the first column.
The ``scaling, --'' column is derived from it by adding a factor of $\log n$ and replacing $\log f_{\max}$ with $\log (n \|\vew\|_\infty)$ in the case of linear objectives.
Note that with large $\|\veu-\vel\|_\infty$, $\log (n \|\vew\|_\infty)$ might be much smaller than $\log f_{\max}$.
Next, the ``scaling, $\rho$'' column is obtained from the previous one by replacing $\log f_{\max}$ with $n \log n$ in the linear case and $n^2 \log n$ in the separable convex case, which follows from our reducibility bounds.
Finally, the last column has the form $\RR(\II,n)$ plus the base complexity without the $\domain$ term and with the $\log f_{\max}$ term replaced by the reduced bounds.
\begin{proof}
	We will use the Summarizing Lemma (Lemma~\ref{lem:summarizing}).
	The two lines of each table correspond to our two algorithms for primal treedepth: the basic algorithm based on the solvability of~\eqref{AugIP} by Lemma~\ref{lem:primal}, and the recursive algorithm of Theorem~\ref{thm:nearlylinear_primal}.
	Let us call these algorithms \emph{basic} and \emph{recursive} in the following text.
	Thus, the main task is to give two ``base'' time complexity bounds, for each algorithm.
	By Corollary~\ref{cor:feas_td}, feasibility is as easy as optimization for primal treedepth, so we focus on bounding the time needed for optimization.
		
	Let us first prove the complexities of the basic algorithm, i.e., the first line of both tables.
	By Lemma~\ref{lem:lambda_oracle_initi}, given an initial solution $\vex_0$,~\eqref{IP} can be solved by solving $\Oh(n \domain \log  f_{\max})$ instances of~\eqref{AugIP}.
	By Lemmas~\ref{lem:primal} and~\ref{lem:primal_norm}, one instance of~\eqref{AugIP} can be solved in time $g \cdot n$.
	Hence $\AAap(\|\veu-\vel\|_\infty, f_{\max}) \leq g n^2 \domain \log f_{\max}$ shows the first column.
	Now we apply Lemma~\ref{lem:summarizing} with $g := g$, $T_{\domain}(\domain) = \domain$, $T_{\text{obj}}(\log f_{\max}) = \log f_{\max}$, and $T_n(n) = n^2$ and the remaining columns follow.
	
	Regarding the recursive algorithm, the first column is irrelevant because it is always advantageous to use the scaling algorithm.
	However, our base complexity bound for Lemma~\ref{lem:summarizing} is the one of Lemma~\ref{lem:nearlylinear_primal_aux}, which states that~\eqref{IP} can be solved in time $g(\|A\|_\infty, \td_P(A)) n (\domain)^{\ttd(F)+1}(\log f_{\max})^{\ttd(F)-1}$.
	Now we apply Lemma~\ref{lem:summarizing} with $g \df g$, $\T_{\domain}(\domain) = \domain^{\ttd(F)+1}$, $T_{\text{obj}}(\log f_{\max}) = (\log f_{\max})^{\ttd(F)-1}$, and $T_n(n) = n$ and the remaining columns follow.
\end{proof}

\medskip

\begin{corollary}[Dual Algorithm] \label{cor:dual}
	Let an~\eqref{IP} instance $\II$ be given with $A \in \Z^{m \times n}$, $F$ a $\td$-decomposition of $G_D(A)$, $K=\max_{P \text{: root-leaf path in }F}\prod_{i=1}^{\ttd(F)} (k_i(P)+1)$, $\RRR$ 
	be a relaxation oracle for $A$, $\domain \df \log \|\veu-\vel, \veb\|_\infty$, redefine $f_{\max} \df \max\{f_{\max}, \|\veb\|_1\}$, and let $g \df g(A,F) \df (\|A\|_\infty K)^{\Oh(\height(F) \cdot (K-1))}$.
	$\II$ is solvable in time at most
	\begin{center}
		\begin{tabular}{lcccc}  
			\toprule	
			obj    & --, -- & scaling, -- & scaling, $\rho$ & $\RRR$, $\rho$ \\
			\midrule
			linear 
			& $g n \log n \cdot \domain(\log f_{\max})$
			& $g  n \log^2 n \cdot \domain (\log n \|\vew\|_\infty)$ 			
			& $g  n^2 \log^2 n \cdot \domain$ 
			& $\RR(\II,n) + g n^2 \log^3 n$   \\
			sep. conv.
			& \multicolumn{2}{c}{$g n \log n \cdot \domain(\log f_{\max})$}
			& $g n^3 \log^3 n \cdot \domain$
			& $\RR(\II,n) + gn^3 \log^3 n$ \\
			\bottomrule
		\end{tabular}
	\end{center}
\end{corollary}
\begin{proof}
	By Theorem~\ref{thm:nearlylinear_dual}, the $\AAA$-augmentation procedure can be realized in time $\AAap(\|\veu-\vel\|_\infty, f_{\max}) \leq (\|A\|_\infty g_1(A))^{\Oh(\height(F))} n \log n \cdot \domain \log (f_{\max})$.
	By Corollary~\ref{cor:feas_td}, feasibility is as easy as optimization for dual treedepth.
	This gives the bound of the first column.
	Now we apply Lemma~\ref{lem:summarizing} with $g \df g$, $T_n(n) = n \log n$, $T_{\domain}(\domain) = \domain$, and $T_{\text{obj}}(f) = \log f_{\max}$ and the remaining columns follow.
\end{proof}

\subsubsection{Non-centered Instances, Infinite Bounds}
To obtain correct bounds for the non-typical cases of instances which are not centered and/or which contain infinities in the lower and upper bounds, one can use the following corollaries.

\begin{corollary}[Non-centered Instances]
	If $\vel, \veu$ are finite but $\vezero \not\in [\vel, \veu]$, dependence on $\veb$ is replaced by $\veb-A\vev$ for any $\vev \in [\vel, \veu]$.
\end{corollary}
\begin{proof}
Follows immediately from the translation of the Centering Lemma~\ref{lem:centered}.
\end{proof}

\begin{corollary}[Infinite bounds (linear case)]
	If $\vel, \veu \in \left(\Z \cup \{\pm \infty\}\right)^n$ and $f(\vex) = \vew\vex$, dependence on $\|\veu-\vel\|_\infty$ is replaced by
	$2 n \|\vel_{\fin}, \veu_{\fin}, \veb\|_\infty g_1(A)$.
\end{corollary}
\begin{proof}
	By Lemma~\ref{lem:feas-inf-bounds-bounded}, if there exists a feasible solution, then there is a feasible solution $\vex_0$ satisfying $\|\vex_0\|_1 \leq \|\veb\|_1 g_1(A)$.
	By Lemma~\ref{lem:infinite_linear}, unboundedness can be detected with one call to~\eqref{AugIP} whose cost is negligible, and if it is bounded, there is an optimum $\vex^*$ with $\|\vex^*-\vex_0\|_1 \leq (n \|\vel_{\fin},\veu_{\fin}\|_{\infty} + \|\vex_0\|_1) g_1(A)$.
	By triangle inequality, we obtain
	\[
	\|\vex^*\|_1 \leq \|\vex^* - \vex_0\|_1 + \|\vex_0\|_1 \leq n \|\vel_{\fin},\veu_{\fin},\veb\|_{\infty} g^2_1(A).
	\]
	Thus, we can replace any infinite entry in $\veu$ with this value, and in $\vel$ with its opposite.
\end{proof}
\begin{corollary}[Infinite bounds (separable convex case)]
	If $\vel, \veu \in \left(\Z \cup \{\pm \infty\}\right)^n$ and $f$ is separable convex, dependence on $\|\veu-\vel\|_\infty$ is replaced by $\|\vex^*\|_\infty + \|\veb\|_1 g_1(A)$, and adding an additional multiplicative factor of $\log (\|\vex^*\|_\infty + \|\veb\|_1 g_1(A))$.
\end{corollary}
\begin{proof}
	By Lemma~\ref{lem:infinite_sepconv} together with Lemma~\ref{lem:feas-inf-bounds-bounded}.
\end{proof}

\subsection{IP Classes of Small Treedepth}
\subsubsection{Transportation Problem, Tables, and $n$-fold IP}
The transportation problem, which asks for an optimal routing from several sources to several destinations, has been defined by Hitchcock~\cite{Hitchcock:1941} in 1941 and independently studied by Kantorovich~\cite{Kantorovich:1942} in 1942, and Dantzig~\cite{Dantzig:1951} showed how the simplex method can be applied to it in 1951.
The transportation problem may be seen as a \emph{table problem} where we are given $m$ row-sums and $n$ column-sums and the task is to fill in non-negative integers into the table so as to satisfy these row- and column-sums.
A natural generalization to higher-dimensional tables, called \emph{multiway tables}, has been studied already in 1947 by Motzkin~\cite{Mot}.
It also has applications in privacy in databases and confidential data disclosure of statistical tables, see a survey by Fienberg and Rinaldo~\cite{FienbergR:2007} and the references therein.

Specifically, the three-way table problem is to decide if there exists a non-negative integer $l \times m \times n$ table satisfying given line-sums, and to find the table if there is one.
Deciding the existence of such a table is \NPc already for $l = 3$~\cite{DeLoeraO:2004}.
Moreover, every bounded integer program can be isomorphically represented in polynomial time for some $m$ and $n$ as some $3 \times m \times n$ table problem~\cite{DeLoeraO:2006}.
The complexity with $l,m$ parameters and $n$ variable thus became an interesting problem.
Let the input line-sums be given by vectors $\veu \in \Z^{ml}, \vev \in \Z^{nl}$ and $\vew \in \Z^{nm}$.
Observe that the problem can be formulated as an~\eqref{IP} with variables $x^i_{j,k}$ for $i \in [n]$, $j \in [m]$ and $k \in [l]$, $f \equiv 0$, and the following constraints:
\begin{align*}
\sum_{i=1}^n x^i_{j,k} &= u_{j,k} & \forall j \in [m], k \in [l], \\
\sum_{j=1}^m x^i_{j,k} &= v^i_k & \forall i \in [n], k \in [l], \\
\sum_{k=1}^l x^i_{j,k} &= w^i_j & \forall i \in [n], j \in [m], \\
\vex &\geq \vezero & \enspace .
\end{align*}
Written in matrix form, it becomes $A\vex=\veb, \, \vex \geq \vezero$ with $\veb=(\veu,\vev^1,\vew^1,\vev^2,\vew^2,\dots,\vev^n,\vew^n)$, $I_k$ the $k\times k$, $k \in \N$, identity matrix, $\veone_k$ the all-ones vector of dimension $k \in \N$, and with
\[
A=\left(
\begin{matrix}
I_{ml} & I_{ml} & \cdots & I_{ml} \\
J & & & \\ 
& J & & \\
& & \ddots & \\
& & & J
\end{matrix}\right),
\text{ where } J=\left(
\begin{matrix}I_{l} & \cdots & I_{l} \\
\veone_l &   & \\
& \ddots &  \\
& & \veone_l
\end{matrix}
\right) \in \Z^{(l+m) \times ml} \enspace .
\]
Here, $J$ has $m$ diagonal blocks $\veone_l$ and $A$ has $n$ diagonal blocks $J$.
This formulation gave rise to the study of \emph{$n$-fold integer programs}, 
where the constraint matrix is of the form
\begin{equation}
A^{(n)}=\left(
\begin{matrix}
A_1 & A_1 & \cdots & A_1 \\
A_2 & & & \\ 
& A_2 & & \\
& & \ddots & \\
& & & A_2
\end{matrix}\right),
\label{eq:nfold}
\end{equation}
for $A_1 \in \Z^{r \times t}$ and $A_2 \in \Z^{s \times t}$.
Such IPs have become the main motivation for the study of~\eqref{IP} with bounded $\td_D(A)$ because they essentially correspond to the class of~\eqref{IP} with a $\td$-decomposition of topological height $2$, as we will soon show.

Another example of $n$-fold IP formulation comes from scheduling.
The problem of \emph{uniformly related machines makespan minimization}, denoted $Q||C_{\max}$ in the standard notation, is the following.
We are given $m$ machines, each with speed $0 < s_i \leq 1$, and $n$ jobs, where the $j$-th job has processing time $p_j \in \N$ and processing it on machine $i$ takes time $p_j/s_i$.
The task is to assign jobs to machines such that the time when the last job finishes (the \emph{makespan}) is minimal, i.e., if $M_i$ is the set of jobs assigned to machine $i$, the task is to minimize $\max_{i \in [m]} \sum_{j \in M_i} p_j/s_i$.
The decision version of the problem asks whether there is a schedule of makespan $C_{\max} \in \R$.
We consider the scenario when $p_{\max} = \max_j p_j$ is bounded by a parameter and the input is represented \emph{succinctly} by multiplicities $n_1, \dots, n_{p_{\max}}$ of jobs of each length, i.e., $n_\ell$ is the number of jobs with $p_j = \ell$.
Letting $x_j^i$ be a variable representing the number of jobs of length $j$ assigned to machine $i$, Knop and Koutecký~\cite{KnopK:2018} give the following $n$-fold formulation:
\begin{align}
\sum_{i=1}^m x_j^i &= n_j & \forall j \in [p_{\max}], \label{eq:sched:1}\\
\sum_{j=1}^{p_{\max}} j \cdot x_j^i &\leq \floor{s_i \cdot C_{\max}} & \forall i \in [m] \enspace . \label{eq:sched:2}
\end{align}
Constraints~\eqref{eq:sched:1} ensure that each job is scheduled on some machine, and constraints~\eqref{eq:sched:2} ensure that each machine finishes before time $C_{\max}$.
This corresponds to an $n$-fold formulation with $A_1 = I_{p_{\max}}$ and $A_2 = (1,2,\dots,p_{\max})$ and with $\|A^{(n)}\|_\infty = p_{\max}$.

Another scheduling problem is finding a schedule minimizing the \emph{sum of weighted completion times $\sum w_j C_j$}.
Knop and Koutecký~\cite{KnopK:2018} show an $n$-fold formulation for this problem as well, in particular one which has a separable quadratic objective.
In the context of scheduling, what sets methods based on $n$-fold IP apart from other results is that they allow the handling of many ``types'' of machines (such as above where machines have different speeds) and also ``non-linear'' objectives (such as the quadratic objective in the formulation for $\sum w_j C_j$).

Another field where $n$-fold IP has had an impact is computational social choice.
The problem of \textsc{Bribery} asks for a cheapest manipulation of voters which lets a particular candidate win an election.
An \FPT algorithm was known for \textsc{Bribery} parameterized by the number of candidates which relied on Lenstra's algorithm.
However, this approach has two downsides, namely a time complexity which is doubly-exponential in the parameter, and the fact that voters have to be ``uniform'' and cannot each have an individual cost function.
Knop et al.~\cite{KnopKM:2017} resolved this problem using $n$-fold IP by showing a single-exponential algorithm for many \textsc{Bribery}-type problems, even in the case when each voter has a different cost function.
For other applications see~\cite{KnopK:2018,KnopKM:2017b,JansenKMR:2018}.

\subsubsection*{Algorithmic Improvements for $n$-fold IP}
The main structural property of $n$-fold IPs is the following:
\begin{lemma}[Structure of $n$-fold IP] \label{lem:nfold:tdd}
	Let $A^{(n)}$ be as in~\eqref{eq:nfold}.
	Denote by $P_n$ the path on $n$ vertices, and let $F$ be obtained by identifying one endpoint of $P_r$ with an endpoint of each of $n$ copies of $P_{s+1}$.
	Then $G_D\left(A^{(n)}\right) \subseteq \cl(F)$ and thus $\td_D\left(A^{(n)}\right) \leq r+s$, $\ttd(F) = 2$, $k_1(F) = r$, and $k_2(F) = s$.
\end{lemma}
\begin{corollary}[$n$-fold IP] \label{cor:nfold}
	Let an $n$-fold IP instance $\II$ be given.
	Let $\RRR$ be a relaxation oracle for $A^{(n)}$, $\domain \df \log \|\veu-\vel\|_\infty$, $N \df nt$, and $g \df (\|A\|_\infty rs)^{\Oh(r^2s + rs^2)}$.
	$\II$ is solvable in time\\
	{\small
		\begin{center}
			\begin{tabular}{lcccc}  
				\toprule	
				obj    & --, -- & scaling, -- & scaling, $\rho$ & $\RRR$, $\rho$ \\
				\midrule
				linear 
				& $gN \log N\domain(\log f_{\max})$
				& $gN \log^2 N\domain (\log N \|\vew\|_\infty)$ 			
				& $g N^2 \log^3 N\domain$ 
				& $\RR(\II,N) + g N^2 \log^3 N$   \\
				sep. convex
				& \multicolumn{2}{c}{$g N \log N\domain(\log f_{\max})$}
				& $g N^3 \log^3 N\domain$
				& $\RR(\II,N) + g N^3 \log^3 N$ \\
				\bottomrule
			\end{tabular}
	\end{center}}
\end{corollary}
\begin{proof}
	By Lemma~\ref{lem:nfold:tdd}, $\td_D\left(A^{(n)}\right) \leq r+s$, $\ttd(F)=2$, $k_1(F)=r$, $k_2(F)=s$.
	Corollary~\ref{cor:dual} then implies the claim.
\end{proof}
This improves on the polynomial dependency (i.e., dependency on the dimension $N$ and the input data $\veb, \vel, \veu, f_{\max}$) over all prior algorithms, with the previously best one being an algorithm of Jansen et al.~\cite{JansenLR:2018} which only pertains to the $f(\vex) = \vew\vex$ case and with time complexity $(\|A\|rs)^{\Oh(r^2s + s^2)} N \log^6 N (\log \|\vel,\veu,\veb,\vew\|_\infty)^2$.
Compared to it, our algorithm has a $\log N$ factor instead of a $\log^6 N$ one and applies to general separable convex objectives.
Moreover, the $\domain \cdot (\log f_{\max})$ factor can be bounded as $\log \|\veu-\vel\|_\infty (\log (\|\vew\|_\infty \cdot \|\veu-\vel\|_\infty)) = (\log \|\veu-\vel\|_\infty)(\log \|\vew\|_\infty + \log \|\veu-\vel\|_\infty) \leq (\log \|\vel,\veu,\veb,\vew\|_\infty)^2$, so our algorithm at worst matches theirs in this regard.
Their parameter dependence is better than ours, but we note that in all known applications of $n$-fold IP, $r^2s > rs^2$  when the corresponding problem parameters are plugged in, and thus there is (so far) no benefit in using the algorithm of Jansen et al.~\cite{JansenLR:2018} in regard to the parameter dependence.

Recalling the previously described applications, Corollary~\ref{cor:dual} thus newly implies a strongly-polynomial algorithm for all the mentioned problems (i.e., tables, scheduling, and bribery).
We also note that all of our results transfer to \emph{generalized $n$-fold integer programming} where the constraint matrix has the form 
\begin{equation*}
A^{(n)}=\left(
\begin{matrix}
\bar{A}_1 & \bar{A}_2 & \cdots & \bar{A}_n \\
A_1 & & & \\ 
& A_2 & & \\
& & \ddots & \\
& & & A_n
\end{matrix}\right),
\end{equation*}
and $r,s$ are the maximum number of rows of $\bar{A}_i$ or $A_i$ over all $i \in [n]$, respectively, and $N$ is the total number of columns of $A^{(n)}$, i.e., the dimension of the~\eqref{IP}.

\subsubsection{2-stage Stochastic IP} \label{sec:2stage}
Another important model arises in decision making under uncertainty.
Here, one is asked to make a partial decision in a ``first stage'', and after realization of some random data, one has to complete their decision in a ``second stage''.
The goal is minimizing the ``direct'' cost of the first-stage decision plus the expected cost of the second-stage decision.
Random data are often modeled by a finite set of $n$ \emph{scenarios}, each with a given probability.
Assume that the scenarios are represented by integer vectors $\veb^1, \dots, \veb^n \in \Z^{t}$, their probabilities by $p_1, \dots, p_n \in (0,1]$, the first-stage decision is encoded by a variable vector $\vex^0 \in \Z^{r}$, and the second-stage decision for scenario $j \in [n]$ is encoded by a variable vector $\vex^j \in \Z^s$. 
Setting $\vex \df (\vex^0, \vex^1, \dots, \vex^n)$ and $\veb \df (\veb^1, \dots, \veb^n)$ then makes it possible to write this problem as
\begin{equation}
\min \vew^0\vex^0 + \sum_{j=1}^{n} p_j \vew' \vex^j \colon B^{(n)}\vex = \veb, \, \vel \leq \vex \leq \veu,\, \vex \in \Z^{r+ns},\, \text{where }
B^{(n)}=
\left(\begin{matrix}
A_1 & A_2     &  &\\
\vdots &      &  \ddots & \\
A_1&      &   & A_2
\end{matrix}\right),
\label{eq:2stage}
\end{equation}
with $A_1 \in \Z^{t \times r}$, $A_2 \in \Z^{t \times s}$, and $\vel, \veu \in \Z^{r+ns}$ some lower and upper bounds.
Problem~\eqref{eq:2stage} is called \emph{$2$-stage stochastic IP} and finds many applications in various areas~\cite{BirgeL:1997,HigleS:1996,KallW:1994,Prekopa:1995,Ruszczynski:1999} and references therein.
Note that~\eqref{eq:2stage} is not \emph{exactly} problem~\eqref{IP} because of the fractional values $p_1, \dots, p_n$ in the objective function: recall that we require $f$ to satisfy $\forall \vex \in \Z^n:\, f(\vex) \in \Z$.
However, this is easily overcome by scaling all $p_j$, $j \in [n]$, by a common large enough integer.

\subsubsection*{Algorithmic Improvements for 2-stage stochastic IP}
It is easy to see that $B^{(n)} = \left(A^{(n)}\right)^{\intercal}$ and thus Lemma~\ref{lem:nfold:tdd} immediately implies as a corollary:
\begin{corollary}[Structure of $2$-stage stochastic IP] \label{cor:2stage:tdp}
	Let $B^{(n)}$ be as in~\eqref{eq:2stage}.
	Denote by $P_n$ the path on $n$ vertices, and let $F$ be obtained by identifying one endpoint of $P_r$ with an endpoint of each of $n$ copies of $P_{s+1}$.
	Then $G_P\left(B^{(n)}\right) \subseteq \cl(F)$ and thus $\td_P(B^{(n)}) \leq r+s$, $\ttd(F) = 2$, $k_1(F) = r$, and $k_2(F) = s$.
\end{corollary}
As a corollary of the above and the Primal Algorithm Corollary (Corollary~\ref{cor:primal}), we get that:
\begin{corollary}[$2$-stage stochastic] \label{cor:2stage}
	Let a $2$-stage stochastic IP instance $\II$ be given.
	Let $\RRR$ be a relaxation oracle for $B^{(n)}$, $\domain \df \log \|\veu-\vel\|_\infty$, and $g \df 2^{(2\|A\|_\infty)^{\Oh(r^2s+rs^2)}}$.
	$\II$ is solvable in time at most
	\begin{center}
		\begin{tabular}{lccc}  
			\toprule	
			obj     & scaling, -- & scaling, $\rho$ & $\RRR$, $\rho$ \\
			\midrule
			linear 
			& $g n \log^3 n\domain(\log\|\vew\|_\infty)$
			& $g n^{2}\log^5 n\domain$ 
			& $\RR(\II,n)+gn^{2}\log^5 n$   \\
			sep. convex
			& $g n\log^3 n \domain(\log f_{\max})$
			& $g n^{3}\log^5 n \domain$
			& $\RR(\II,n)+gn^{3}\log^5 n$\\
			\bottomrule
		\end{tabular}
	\end{center}
\end{corollary}
\begin{proof}
	By Corollary~\ref{cor:2stage:tdp}, $\td_P\left(B^{(n)}\right) \leq r+s$, $\ttd(F)=2$, $k_1(F)=r$, $k_2(F)=s$.
	Corollary~\ref{cor:primal} then implies the claim.
\end{proof}
In the weakly polynomial case, this matches the parameter dependence of Klein~\cite{Klein} and improves the dependency on $n$ from $n^2$ to $n \log^3 n$.
The only strongly-polynomial algorithm previously known is for the linear case $f(\vex) = \vew \vex$ due to the preliminary conference version of this paper~\cite{KouteckyLO:2018} which gave a bound of $g' n^6$ for some computable function $g'$ without any concrete bounds.
Thus, Corollary~\ref{cor:2stage} improves the parameter dependence to doubly-exponential and the polynomial dependence from $n^6$ to $n^2 \log^5 n$.

\subsubsection{Multi-stage Stochastic and Tree-fold Matrices} \label{subsec:multistage_treefold}
In the following we let $T$ be a rooted tree of height $\tau \in \N$.
For a vertex $v \in T$, let the \emph{depth of $v$} be the distance of $v$ from the root.
Let all leaves of $T$ be at depth $\tau-1$.
For a vertex $v \in T$, let $T_v$ be the subtree of $T$ rooted in $v$ and let $\ell(v)$ denote the number of leaves of $T$ contained in $T_v$.
Let $A_1,A_2,\dots,A_{\tau}$ be a sequence of integer matrices
with each $A_s$ having $l \in \N$ rows and $n_s$ columns, where $n_s\in\N$, $n_s\geq 1$.
We shall define a \emph{multi-stage stochastic} matrix $T^P(A_1, \dots, A_{\tau})$ inductively; the superscript $P$ refers to the fact that $T^P(A_1, \dots, A_{\tau})$ has bounded $\td_P$, as we will later see.

For a leaf $v \in T$, $T^P_v(A_{\tau}) \df A_{\tau}$.
Let $d \in \N$, $d \in [0, \tau-2]$, and assume that for all vertices $v \in T$ at depth $d+1$, matrices $T^P_v(A_{d+2}, \dots, A_{\tau})$ have been defined.
For $s \in \N$, $s \in [\tau]$, we set $T^P_v(A_{[s:	\tau]}) \df T^P_v(A_s, \dots, A_{\tau})$.
Let $v \in T$ be a vertex at depth $d$ with $\delta$ children $v_1, \dots, v_\delta$.
We set
\[
T^P_v (A_{[d+1 : \tau]}) := \left(
\begin{array}{cccc}
A_{d+1, \ell(v_1)}      & T^P_{v_1}(A_{[d+2 : \tau]})        & \cdots & 0\\
\vdots   & \vdots    & \ddots & \vdots\\
A_{d+1, \ell(v_\delta)}            & 0       & \cdots & T^P_{v_\delta}(A_{[d + 2 : \tau]})\\
\end{array}
\right)
\]
where, for $N \in \N$, $A_{s,N}=\left(\begin{smallmatrix} A_s \\ \vdots \\ A_s\end{smallmatrix}\right)$
consists of $N$ copies of the matrix $A_s$.

The structure of a multi-stage stochastic matrix makes it natural to partition any solution of a multi-stage stochastic IP into \emph{bricks}.
Bricks are defined inductively: for $T_v^P(A_{\tau})$ there is only one brick consisting of all coordinates; for $T_v^P(A_{[s:\tau]})$ the set of bricks is composed of all bricks for all descendants of $v$, plus the first $n_s$ coordinates form an additional brick.

\begin{example}
	For $\tau=3$ and $T$ with root $r$ of degree $2$ and its children $u$ and $v$ of degree $2$ and $3$, we have $T^P_u(A_2, A_3)=
	\left(\begin{smallmatrix}
	A_2 & A_3  &     \\
	A_2 &      & A_3
	\end{smallmatrix}\right)$,
	$T^P_v(A_2, A_3)=
	\left(\begin{smallmatrix}
	A_2 & A_3  &    & \\
	A_2 &      & A_3 & \\
	A_2 &      &  & A_3
	\end{smallmatrix}\right)$, and
	$T^P(A_1, A_2, A_3) = T^P_r(A_1, A_2, A_3)=
	\left(\begin{smallmatrix}
	A_1 & A_2 & A_3 &     &     &     &    & \\
	A_1 & A_2 &     & A_3 &     &     &    & \\
	A_1 &     &     &     & A_2 & A_3 &    & \\
	A_1 &     &     &     & A_2 &     & A_3& \\
	A_1 &     &     &     & A_2 &     &    & A_3 \\
	\end{smallmatrix}\right)$, with a total of $8$ bricks.
\end{example}


\emph{Tree-fold} matrices are transposes of multi-stage stochastic ILP matrices.
Let $T$ be as before and $A_1, \dots, A_{\tau}$ be a sequence of integer matrices with each $A_s \in \Z^{r_s \times t}$, where $t \in \N$, $r_s \in \N$, $r_s \geq 1$.
We shall define $T^D(A_1, \dots, A_{\tau})$ inductively; the superscript $D$ refers to the fact that $T^D(A_1, \dots, A_{\tau})$ has bounded $\td_D$.
The inductive definition is the same as before except that, for a vertex $v \in T$ at depth $d$ with $\delta$ children $v_1, \dots, v_\delta$,
we set
\[
T^D_v (A_{[d + 1 :\tau]})\df \left(
\begin{array}{ccccc}
A_{d + 1, \ell(v_1)}      & A_{d + 1, \ell(v_2)}  & \cdots & A_{d + 1, \ell(v_\delta)}   \\
T^D_{v_1}(A_{[d + 2 :\tau]})      & 0       & \cdots & 0\\
0 & T^D_{v_2}(A_{[d + 2 :\tau]})           & \cdots & 0\\
\vdots   & \vdots   & \ddots & \vdots\\
0      & 0 & \cdots & T^D_{v_\delta}(A_{[d + 2 :\tau]})      \\
\end{array}
\right)
\]
where, for $N \in \N$, $A_{s,N}=\left( A_s \\ \cdots \\ A_s\right)$ consists of $N$ copies of the matrix $A_s$.
A solution $\vex$ of a tree-fold IP is partitioned into bricks $(\vex^1, \dots, \vex^n)$ where $n$ is the number of leaves of $T$, and each $\vex^i$ is a $t$-dimensional vector.

\medskip

Multi-stage stochastic IPs arise, as the name suggests, in models generalizing the two-stage process of decision making under uncertainty previously described in Section~\ref{sec:2stage}, see~\cite{BirgeL:1997,KallW:1994,Ruszczynski:1999}.
Here, instead of making a decision in two stages and with the random data being realized completely after the first stage, the random data is being realized gradually, with partial decisions being made in each stage.
Note that, as before, once a decision is made by an actor, it cannot be reversed.

Tree-fold IPs have been introduced by Chen and Marx~\cite{MC} in order to show an \FPT algorithm for the \textsc{Subtree Cover} problem.
In this problem, we are given a rooted tree $T$ and integers $m,k$, and the task is to cover $T$ with $m$ rooted subtrees which have the same root as $T$, and each contains at most $k$ edges.
This problem is equivalent to multi-agent TSP on a rooted tree.
Chen and Marx show that this problem admits an \FPT algorithm when parameterized by $k$ by giving a tree-fold IP formulation with $\tau=k$ and showing that tree-fold ILP is \FPT parameterized by $\tau, t, r_1, \dots, r_\tau$.

\subsubsection*{Algorithmic Improvements for tree-fold and multi-stage stochastic IP}
Let us first give a structural lemma as in the previous cases.
\begin{lemma}[Structure of multi-stage stochastic and tree-fold IP] \label{lem:tfmss_structure}
	Let $T$ be a rooted tree of height $\tau$ and $A_1, \dots, A_\tau$ be integer matrices, each having $l$ ($r_1, \dots, r_\tau$) rows and $n_1, \dots, n_\tau$ ($t$) columns, respectively.
	Let $F$ be obtained from $T$ as follows.
	For each $i \in [0,\tau-1]$, replace every vertex $v$ at depth $i$ with a copy of the path $P = P_{n_{i+1}}$ on $n_{i+1}$ ($P = P_{r_{i+1}}$ on $r_{i+1}$) vertices, respectively, in such a fashion that one endpoint of $P$ is connected to all former children of $v$, and the other endpoint is adjacent to the parent of $v$.
	Then, 
	\[
	G_P\left(T^P(A_1, \dots, A_{\tau})\right) \subseteq \cl(F) \quad \left(G_D\left(T^D(A_1, \dots, A_{\tau})\right) \subseteq \cl(F) \right),
	\]
	and $\height(F) = \sum_{i=1}^\tau n_i$ ($\height(F) = \sum_{i=1}^\tau r_i$), $\ttd(F) = \tau$, and for each $i \in [n]$, $k_i(F) = n_i$ ($k_i(F) = r_i$), respectively.
\end{lemma}
\begin{proof}
	First observe that the transpose of a multi-stage stochastic matrix is a tree-fold matrix, whose blocks are the transposes of the blocks of the original matrix, i.e., \[\left(T^P(A_1, \dots, A_{\tau})\right)^\intercal = T^D(A^\intercal_1, \dots, A^\intercal_{\tau}) \enspace .\]
	Thus we focus on proving the statement for $T^P(A_1, \dots, A_{\tau})$ and the dual case then immediately follows.
	The proof may be carried out by induction on $\tau$.
	As the base case, when $\tau=1$, there is only one block, $A_1$, which has $n_1$ columns, and clearly the path $P_{n_1}$ on $n_1$ is a $\td$-decomposition of $G_P(A_1)$.
	Assume that the claim holds for each $\tau' < \tau$ and consider the matrix $T^P(A_1, \dots, A_{\tau})$.
	Let $A'$ be $T^P(A_1, \dots, A_{\tau})$ without the first $n_1$ columns.
	By the definition of $T^P(A_1, \dots, A_{\tau})$, $A'$ has block diagonal structure with blocks $T^P_{v_j}(A_{[2:\tau]})$ for $v_1, \dots, v_\delta$ the children of the root of $T$.
	By the induction hypothesis, the lemma holds for each $T^P_{v_j}(A_{[2:\tau]})$.
	Denote by $F_1, \dots, F_\delta$ the rooted trees implied by the lemma such that, for each $i \in [\delta]$, $F_i$ is a $\td$-decomposition of $G_P\left(T^P_{v_i}(A_{[2:\tau]})\right)$, the height of $F_i$ is $\sum_{j=2}^\tau n_j$, $\ttd(F_i) = \tau-1$, and for each $j \in [\tau-1]$, $k_j(F_i) = n_{j+1}$.
	
	Let $F$ be obtained by joining the roots of $F_1, \dots, F_\delta$ by an edge with one endpoint of $P_{n_1}$ and letting the other endpoint be a root.
	Clearly $F$ is a $\td$-decomposition of $G_P\left(T^P(A_1, \dots, A_{\tau})\right)$, its height is $n_1 + \max_{i \in [\tau]} \height(F_i) = \sum_{i=1}^\tau n_i$, $\ttd(F) = \tau$, and $k_i(F) = n_i$ for each $i \in [\tau]$, as claimed.
\end{proof}
From this lemma we are able to obtain algorithmic corollaries.
\begin{corollary}[Multi-stage stochastic IP] \label{cor:mss}
	Let a multi-stage stochastic IP instance $\II$ be given, $\RRR$ be a relaxation oracle for $T^P(A_1, \dots, A_{\tau})$, $\domain \df \log \|\veu-\vel\|_\infty$, $k \df \sum_{i=1}^\tau n_i$, $N \df |V(T)|$ and
	\[
	g \df \stackinset{l}{4pt}{b}{-9pt}{\tiny\rotatebox{33}{$\underbrace{\kern21pt}_{\tau-1}$}}{%
		$2^{2^{\rdots^{2^{(2\|A\|_\infty)^{\Oh\left(2^{\tau} \cdot k^2\right)}}}}}$}
	\] 
	When $f(\vex) = \vew \vex$, $\II$ is solvable in time at most
	\begin{center}
		\begin{tabular}{cccc}  
			\toprule	
			--, -- & scaling, -- & scaling, $\rho$ & $\RRR$, $\rho$ \\
			\midrule
			$g\domain N^2(\log f_{\max})$
			& $g\domain N^2\log N \cdot (\log N \|\vew\|_\infty)$ 			
			& $g\domain N^3 \log^2 N$ 
			& $\RR(\II,N)+g N^3 \log^2 N$   \\ 
			\multicolumn{2}{c}{$g\domain N^{1+o(1)}(\log\|\vew\|_\infty)^{\tau-1}$}			
			& $g\domain N^{\tau+o(1)}$ 
			& $\RR(\II,N)+gN^{\tau+o(1)}\log^2 N$   \\
			\bottomrule
		\end{tabular}
	\end{center}
	When $f$ is an arbitrary separable convex function, $\II$ can be solved in time at most
	\begin{center}
		\begin{tabular}{cccc}  
			\toprule	
			--, -- &  scaling, --  &  scaling, $\rho$  &  $\RRR$, $\rho$ \\
			\midrule
			\multicolumn{2}{c}{$g\domain N^2(\log f_{\max})$}
			& $g\domain N^4 \log^2 N$
			& $\RR(\II,N)+g N^4 \log^2 N$ \\
			\multicolumn{2}{c}{$g \domain N^{1+o(1)} (\log f_{\max})^{\tau-1}$}
			& $g\domain N^{2\tau-1+o(1)}$
			& $\RR(\II,N)+gN^{2\tau-1+o(1)}\log^2 N$\\
			\bottomrule
		\end{tabular}
	\end{center}
	The $N^{o(1)}$ term above is more precisely $\log^{\tau+1} N$.
\end{corollary}
\begin{proof}
	By Lemma~\ref{lem:tfmss_structure}, $\td_P\left(T^P(A_1, \dots, A_{\tau})\right) \leq k$, $\ttd(F)=\tau$, and $k_i(F) = n_i$ for each $i \in [\tau]$.
	Let $n_{\max} \df \max_{i \in [\tau]} n_i$ and note that $n_{\max} \leq k$.
	The number of columns $n$ of $T^P(A_1, \dots, A_{\tau})$ is upper bounded by $N n_{\max} \leq Nk$, and $k \leq g$ and thus the factor $k$ gets consumed by the big-$\Oh$ (Landau) notation in the definition of $g$.
	Hence $n$ in Corollary~\ref{cor:dual} is correctly replaced here by $N$ although $N$ is less than the dimension.
	Corollary~\ref{cor:dual} implies the rest of the claim.
\end{proof}
An \FPT algorithm for multi-stage stochastic IP follows from the work of Aschenbrenner and Hemmecke~\cite{AH} although it is not clearly stated there, similarly to the treatment of De Loera et al.~\cite{DHK}.
The currently fastest algorithm is due to Klein~\cite{Klein} and attains runtime $g n^2 \log f_{\max}$.
(We note that even though Klein's algorithm is only stated for linear objectives, it is easily extended to the separable convex case.)
His complexity matches ours in the general case, but we additionally provide a better bound with a dependence on $n$ of $n^{1+o(1)}$ in the more restricted regime when $\log(f_{\max})^{\tau-1}$ is less than $n \cdot \log(f_{\max})$, or, in the linear case $f(\vex) = \vew \vex$, when $(\log \|\vew\|_\infty)^{\tau-1}$ is less than $n \cdot \log(f_{\max})$.
Note that even this restrictive regime is quite useful: it captures for instance the case of deciding feasibility of a multi-stage stochastic IP (i.e., deciding whether there is a decision path satisfying any sequence of scenarios), as well as the case when the costs and probabilities (scaled to integers) are bounded by a polynomial in $n$.
Moreover, as in the case of 2-stage stochastic IP, the only strongly-polynomial algorithm previously known is for the linear case $f(\vex) = \vew \vex$ due to the preliminary conference version of this paper~\cite{KouteckyLO:2018} which gave a bound of $g' n^6$ for some computable function $g'$ without any concrete bounds.
Thus, Corollary~\ref{cor:2stage} improves the parameter dependence to an exponential tower and the polynomial dependence from $n^6$ to $n^3 \log^2 n$.

Turning to tree-fold IP, we obtain the following algorithmic result:
\begin{corollary}[Tree-fold IP] \label{cor:tf}
	Let a tree-fold IP be given.
	Let $\RRR$ be a relaxation oracle for $T^D(A_1, \dots, A_{\tau})$, $\domain \df \log \|\veu-\vel\|_\infty$, $K \df \prod_{i=1}^{\tau} (r_i+1)$, $k \df \sum_{i=1}^\tau r_i$, $N := nt$, and $g \df (\|A\|_\infty K)^{\Oh(k \cdot (K-1))}$.
	$\II$ is solvable in time at most
	\begin{center}
		\begin{tabular}{lcccc}  
			\toprule	
			obj    & --, -- & scaling, -- & scaling, $\rho$ & $\RRR$, $\rho$ \\
			\midrule
			linear 
			& $g N \log N\domain(\log f_{\max})$
			& $g  N \log^2 N \domain (\log N \|\vew\|_\infty)$ 			
			& $g  N^2 \log^2 N \domain$ 
			& $\RR(\II,N) + g N^2 \log^3 N$   \\
			s. conv.
			& \multicolumn{2}{c}{$g N \log N\domain(\log f_{\max})$}
			& $g N^3 \log^3 N\domain$
			& $\RR(\II,N) + gN^3 \log^3 N$ \\
			\bottomrule
		\end{tabular}
	\end{center}
\end{corollary}
\begin{proof}
	By Lemma~\ref{lem:tfmss_structure}, $\td_D\left(T^D(A_1, \dots, A_{\tau})\right) \leq k$, $\ttd(F)=\tau$, and $k_i(F) = r_i$ for each $i \in [\tau]$.
	Corollary~\ref{cor:dual} then implies the claim.
\end{proof}
The previously best algorithm for tree-fold IP is due to the preliminary conference version of this paper~\cite{EisenbrandHK:2018} and gave a bound of $g N^2 \domain (\log f_{\max})$.
Thus, we improve its polynomial dependence from $N^2$ to $N \log N$.
The same applies for the strongly-polynomial algorithm when $f$ is linear, where a conference version of this paper~\cite{KouteckyLO:2018} gave a $g' N^6$ algorithm for an unspecified computable function $g'$ only depending on $r_1, \dots, r_\tau$, which we improve here to $gN^2 \log^3 N$.

\subsection{Small Treewidth} \label{sec:apps:tw}
We have seen that when $A$ is a matrix for which $g_\infty(A)$ or $g_1(A)$ and $\tw_P(A)$ or $\tw_D(A)$ can be bounded by parameters, then the algorithms of Lemmas~\ref{lem:primal_treewidth} and~\ref{lem:dual_treewidth}, respectively, solve~\eqref{AugIP} in \FPT time.
In combination with Corollary~\ref{cor:lambda_oracle} this immediately implies solvability of~\eqref{IP} with a constraint matrix $A$ in \FPT time:
\begin{theorem} \label{thm:tw}
	\eqref{IP} is solvable in time \[\min\left\{(2g_\infty(A))^{\Oh(\tw_P(A))},(2\|A\|_\infty g_1(A))^{\Oh(\tw_D(A))}\right\} n^2 \log \|\veu-\vel\|_\infty \log f_{\max} + \Oh(n^{\omega})\enspace .\]
\end{theorem}
We remark that purification can be realized quickly by Proposition~\ref{prop:tw_pure} and it is easy to show that the matrix $A_I = (A~I)$ of the auxiliary feasibility instance~\eqref{eq:auxiliary_feasibility} and the scaling algorithm~\eqref{eq:s-ip-scaled} satisfies $\tw_P(A_I) \leq \tw_P(A)+1$ and $\tw_D(A_I) = \tw_D(A)$ by an analogue of Lemma~\ref{lem:feas_td}, thus circumventing the $\Oh(n^\omega)$ additive factor.

In general, it is not possible to bound $g_\infty(A)$ or $g_1(A)$ even if $\tw_P(A), \tw_D(A) =1$ and $\|A\|_\infty=2$ (Lemma~\ref{lem:largeg1}), which is the reason why matrices with small primal or dual treewidth have not been the center of our attention.
However, in specific cases good bounds can be proven and advantageously used, see~Gavenčiak et al.~\cite{GavenciakKK:2019}.

\subsection{Fast Relaxation Algorithms}
By Corollary~\ref{cor:scaling_relaxation} fast algorithms for~\eqref{IP} imply essentially as fast (up to a $\log n \frac{1}{\epsilon}$ factor) algorithms for the fractional relaxation~\eqref{relax}.
Specifically, an $\epsilon$-accurate solution to~\eqref{relax} can be found in time $T \cdot \log(n \frac{1}{\epsilon})$, where $T$ is the time complexity appearing in Corollaries~\ref{cor:nfold},~\ref{cor:2stage},~\ref{cor:mss},~\ref{cor:tf} and Theorem~\ref{thm:tw}.
In particular, we have the following result:
\begin{theorem}[$n$-fold and $2$-stage stochastic relaxation, informal]
There is a nearly-linear \FPT algorithm for the relaxation of $n$-fold and $2$-stage stochastic IP.
\end{theorem}

\section{Hardness and Lower Bounds} \label{sec:hardness}
In most of our hardness reductions we will use the famous \NPh \textsc{Subset Sum} problem:
\prob{\textsc{Subset Sum}}
{Positive integers $a_1, \dots, a_n, b$.}
{Is there $I \subseteq [n]$ such that $\sum_{i \in I} a_i = b$?}

\begin{proposition}[\eqref{ILP} hardness] \label{prop:ilphardness}
	\eqref{ILP} is \NPh already when $\|A\|_\infty=1$ or when $m=1$.
\end{proposition}
\begin{proof}
	First, consider the \NPh \textsc{Vertex Cover} problem in which we are given a graph $G$ and the task is to find $C \subseteq V(G)$ such that $\forall e \in E(G):\, e \cap C \neq \emptyset$ and $|C|$ is minimized.
	Without loss of generality assume $V(G) = [n]$ and let $x_i$, $i \in [n]$, be a $0/1$ variable encoding whether vertex $i$ belongs to $C$.
	Furthermore, we need a slack variable $s_{i,j}$ for every edge $\{i,j\} \in E(G)$.
	The following~\eqref{ILP} instance encodes the given \textsc{Vertex Cover} instance:
	\begin{align*}
	\min & \sum_{i=1}^n x_i & \\
	x_i + x_j - s_{i,j}&= 1 & \forall \{i,j\} \in E(G) \\
	x_i & \geq 0 & \forall i \in [n] \\
	s_{i,j} & \geq 0 & \forall \{i,j\} \in E(G)
	\end{align*}
	Clearly the largest coefficient is $1$, hence~\eqref{ILP} with $\|A\|_\infty \geq 1$ is \NPh.
	
	As for the second part, let $\vea = (a_1, \dots, a_n), b$ be a given \textsc{Subset Sum} instance.
	As before, let $x_i$, $i \in [n]$, be a $0/1$ variable encoding whether $i \in I$.
	Deciding whether the following~\eqref{ILP} instance is feasible is equivalent to deciding the given \textsc{Subset Sum} instance:
	\begin{align*}
	\vea \vex &= b & \\
	\vex &\geq \vezero &
	\end{align*}
	Since the constraint matrix is $A=\vea$, we have $m = 1$ and~\eqref{ILP} is \NPh already when $m=1$.
\end{proof}

\medskip

Eiben et al.~\cite{EibenGKOPW:2019} have recently shown that~\eqref{ILP} is \NPh already when the more permissive incidence treedepth $\td_I(A)$ is $5$ and $\|A\|_\infty=2$.
Hence, $\td_P$ and $\td_D$ in our results cannot be replaced with $\td_I$.

\subsection{\NP-hardness of Non-separable Convex and Separable Concave Integer Optimization}
\begin{proposition}[{Part~\ref{it:nearlylinear_primal:claim1} in~\cite[Proposition 1]{LeeORW:2012}}]~ \label{prop:nphconvexconcave}
	\begin{enumerate}
		\item \label{it:nphprop:1} It is \NPh to minimize a non-separable quadratic convex function over $\Z^n$.
		\item Problem~\eqref{IP} is \NPh already when \label{it:nphprop:2} $f$ is separable concave and $A=(1\cdots1)$.
	\end{enumerate}
\end{proposition}
\begin{proof}
	Let a \textsc{Subset Sum} instance be given and denote $\vea \df (a_1, \dots, a_n)$.
	\paragraph*{Part~\ref{it:nphprop:1}.}
	We encode the \textsc{Subset Sum} instance into $n$ binary variables $x_1, \dots, x_n$, so the goal is to enforce that an optimal solution $\vex$ satisfies $\vea \vex = b$ if and only if the instance is a \textsc{Yes}-instance.
	The idea here is that the objective function allows us to encode a ``barrier function'' which attains its minimum if and only if $(\vea \vex - b)^2 = 0$ and $x_i \in \{0,1\}$ for each $i \in [n]$.
	Already setting $f'(\vex) \df (\vea \vex -b)^2$ shows that~\eqref{IP} is \NPh with a non-separable convex objective and with bounds $\vezero \leq \vex \leq \veone$.
	Then, the bounds can be encoded in exactly the same way, setting $f \df f'(\vex) + \sum_{i=1}^n (2x_i-1)^2$. 
	Because $\min_{x_i \in \Z}(2x_i-1)^2=1$ is attained when $x_i \in \{0,1\}$, we have that $\min_{\vex \in \Z^n} f(\vex) = n$ if and only if $\vex \in \{0,1\}^n$ and $\vea \vex = b$, i.e., when the instance is a \textsc{Yes}-instance.
	
	\paragraph*{Part~\ref{it:nphprop:2}.}
	The idea of the proof is to use the objective function to encode a disjunction, i.e., for each variable $x_i$, enforcing $x_i \in \{0, a_i\}$.
	This is done by setting $\vel \df \vezero$, $\veu \df \vea$, and for each $i \in [n]$, $f_i(x_i) \df -(x_i - \frac{a_i}{2})^2$.
	Because $\min_{x_i \in [0,a_i]} f_i(x_i) = -\frac{a_i^2}{4}$ is attained when $x_i \in \{0,a_i\}$, it holds that $\min_{\vex \in [\vel, \veu]} f(\vex) = \sum_{i=1}^n -\frac{a_i^2}{4}$ if and only if $x_i \in \{0, a_i\}$ for each $i \in [n]$. Then,
	the single linear constraint $\sum_{i=1}^n x_i=b$ with the bounds $\vel, \veu$ is equivalent to the input instance being a \textsc{Yes}-instance.
\end{proof}

\subsection{\NP-hardness for Treewidth and Double-exponential Lower Bounds for Treedepth}
Our goal now is to exhibit two encodings of the \textsc{Subset Sum} problem which show that~\eqref{IP} is \NPh already when $\tw_P(A), \tw_D(A) \leq 2$ and $\|A\|_\infty = 2$, and then derive double-exponential time complexity lower bounds for~\eqref{IP} parameterized by $\td_P(A), \td_D(A)$.
For the $\td_D(A)$ parameter our lower bound is off by roughly a $\td_D(A)$ factor in the exponent when compared to our upper bounds, which asymptotically means that the dependency on the level heights $k_1(F) ,\dots, k_{\ttd(F)}(F)$ is inherent, assuming ETH.
Regarding the $\td_P(A)$ parameter, no non-trivial lower bounds were previously known.
Our encoding of \textsc{Subset Sum} is inspired by~\cite[Theorem 12]{GOR}.

We begin with the natural encoding of \textsc{Subset Sum} with $n$ boolean variables $x_1, \dots, x_n$:
\begin{equation}
\sum_{i=1}^n a_i x_i = b \enspace . \label{eq:subsetsum}
\end{equation}
Unfortunately, constraint~\eqref{eq:subsetsum} contains large coefficients and has primal treewidth $n$.
In the following we will use two tricks to overcome these two problems.

Let us now describe these two tricks informally.
The first trick is to rewrite a constraint such as~\eqref{eq:subsetsum} into $n$ new constraints, each only involving $3$ variables, while introducing $n$ new variables.
To see how this could be done, consider constraints $z_1 = a_1 x_1$ and $z_i = a_i x_i + z_{i-1}$ for all $i \geq 2$ -- clearly then $z_n = \sum_{i=1}^n a_i x_i$.
The second trick is the idea of introducing, for each item $i \in [n]$, variables $y_i^j$ such that, if $x_i = 0$ then $y_j^i = 0$ and if $x_i = 1$ then $y_i^j = 2^j$.
Then, it is possible to obtain the term $a_i x_i$ by summing up those $y_i^j$ which correspond to the digits of $a_i$ in its binary encoding which are equal to one.
In the following we will assume not only a base-$2$ encoding, but encodings in some general base $\Delta$.

Formally, let $\Delta \in \N_{\geq 2}$, and assuming $a_i \leq b$ for all $i \in [n]$, let $L_\Delta \df \ceil{\log_\Delta (b + 1)}$.
Denote by $[a_i]_\Delta = (\alpha_i^0, \dots, \alpha_i^{L_\Delta - 1})$ the base-$\Delta$ encoding of $a_i$, i.e., $a_i = \sum_{j=0}^{L_\Delta - 1} \alpha_i^j \Delta^j$.
Thus, $a_i x_i = \sum_{j=0}^{L_\Delta - 1} \alpha_i^j y_i^j$.
(Note that the superscript $\bullet^j$ only means the $j$-th power when written over $\Delta$.)
Now, let
\begin{align}
y_i^0 &= x_i & \forall i \in [n] \tag{$X_i$} \label{eq:tw:iy0}\\
y_i^j &= \Delta \cdot y_i^{j-1} & \forall i \in [n], \, \forall j \in [L_\Delta - 1] \tag{$Y_i^j$} \label{eq:tw:yij} \\
\sum_{i=1}^n \sum_{j=0}^{L_\Delta - 1} \alpha_i^j y_i^j &= b &  \enspace .\tag{$S$} \label{eq:tw:dualagg}
\end{align}
\begin{lemma} \label{lem:dualdtwdecomp}
	Let $A$ be the matrix of constraints~\eqref{eq:tw:yij}--\eqref{eq:tw:dualagg}.
	$G_D(A)$ has a path decomposition of width $2$ and length $n(L_\Delta-1)-1$.
\end{lemma}
\begin{proof}
	We will disregard the variables $x_i$ and thus also the constraint~\eqref{eq:tw:iy0} for the sake of slightly improved bounds as the variables $y_i^0$ play an identical role.
	Let $A$ be the matrix of constraints~\eqref{eq:tw:yij}--\eqref{eq:tw:dualagg}.
	The graph $G_D(A)$ contains the following edges:
	\begin{itemize}
		\item Between $S$ and each~$Y_i^j$.
		\item Between $Y_i^{j-1}$ and $Y_i^{j}$ for each $i \in [n]$ and $j \in [L_\Delta - 1]$.
	\end{itemize}
	We construct a tree decomposition (in fact, a path decomposition), by consecutively taking the following segment of bags for each $i \in [n]$:
	\[
	\{S, Y_i^0, Y_i^1\}, \{S, Y_i^1, Y_i^2\}, \dots, \{S, Y_i^{L_\Delta-2}, Y_i^{L_\Delta-1}\} \enspace .
	\]
	Since each bag is of size $3$, the treewidth is $2$.
	Moreover, since each segment comprises $L_\Delta-1$ bags, and there are $n$ segments, the length of the path decomposition is $n(L_\Delta-1)-1$.
	Note that there are $n L_\Delta$ variables and $n (L_\Delta -1) + 1$ constraints.
\end{proof}
\begin{corollary}[$\tw_D$ hardness] \label{cor:twdhardness}
	\eqref{IP} is \NPh already when $\tw_D(A) = 2$ and $\|A\|_\infty = 2$.
\end{corollary}
\begin{proof}
Let $\Delta=2$ and apply Lemma~\ref{lem:dualdtwdecomp}.
\end{proof}

When it comes to $\tw_P$, the constraint~\eqref{eq:tw:dualagg} corresponds to a large clique.
Consider instead the following set of constraints:
\begin{align}
z_i^j &=\begin{cases}
y_1^0 & \text{if } i=1,\, j=0 \\
z_{i-1}^{L_\Delta-1} + \alpha_i^0 y_i^0 & \text{if } i>1, j=0\\
z_{i}^{j-1} + \alpha_i^j y_i^j & \text{if } j>0\end{cases} \label{eq:tw:zij} \tag{$Z_i^j$}\\
z_n^{L_\Delta - 1} &= b & \enspace .\tag{$S'$} \label{eq:tw:primalagg} 
\end{align}
The intuitive meaning of $z_i^j$ is that it is a prefix sum of the constraint~\eqref{eq:tw:dualagg}, i.e., $z_i^j = \left(\sum_{k=1}^{i-1} \sum_{\ell=0}^{L_\Delta-1} \alpha_k^\ell y_k^\ell\right) + \left(\sum_{\ell=0}^{j} \alpha_i^\ell y_i^\ell \right)$.
\begin{lemma} \label{lem:primaldtwdecomp}
	Let $A$ be the matrix of constraints~\eqref{eq:tw:yij},~\eqref{eq:tw:zij}, and~\eqref{eq:tw:primalagg}.
		$G_P(A)$ has a path decomposition of width $2$ and length at most $2 n L_\Delta $.
\end{lemma}
\begin{proof}
	Let us analyze the primal treewidth of constraints~\eqref{eq:tw:yij},~\eqref{eq:tw:zij}, and~\eqref{eq:tw:primalagg}.
	We shall again disregard the variables $x_i$ and simply identify them with $y_i^0$.
	Denoting the constraint matrix as $A$, the graph $G_P(A)$ has the following edges:
	\begin{itemize}
		\item $\{y_i^{j-1}, y_i^{j}\}$, $\{z_i^j, z_i^{j-1}\}$, and $\{y_i^{j}, z_i^{j-1}\}$ for each $i \in [n]$ and $j \in [L_\Delta-1]$,
		\item $\{z_{i-1}^{L_\Delta-1} , y_i^0\}$ and $\{z_i^0, z_{i-1}^{L_\Delta -1 }\}$ for each $i \in [n]$,
		\item $\{z_i^j, y_i^j\}$ for each $i \in [n]$ and $j \in [0,L_\Delta-1]$.
	\end{itemize}
	The following sequence of bags constitutes a path decomposition of $G_P(A)$ of width $2$ and length $2 n L_\Delta - n - 1$:
	\begin{align*}
	\{y_1^0,z_1^0,y_1^1\},\{z_1^0,y_1^1,z_1^1\},\{y_1^1,z_1^1,y_1^2\},& \dots, \{y_1^{L_\Delta-2},z_1^{L_\Delta-2},y_1^{L_\Delta-1}\}, \{z_1^{L_\Delta-2},y_1^{L_\Delta-1},z_1^{L_\Delta-1}\}, \\
	\{z_1^{L_\Delta-1},y_2^0,z_2^0\},\{y_2^0,z_2^1,y_2^1\}, &\dots,\{y_2^{L_\Delta-2},z_2^{L_\Delta-2},y_2^{L_\Delta-1}\}, \{z_2^{L_\Delta-2},y_2^{L_\Delta-1},z_2^{L_\Delta-1}\}, \\
	&\enspace\vdots\\
	\{z_{n-1}^{L_\Delta-1},y_n^0,z_n^0\},\{y_n^0,z_n^0,y_n^1\}, &\dots,\{y_n^{L_\Delta-2},z_n^{L_\Delta-2},y_n^{L_\Delta-1}\}, \{z_n^{L_\Delta-2},y_n^{L_\Delta-1},z_n^{L_\Delta-1}\}\enspace .
	\end{align*}
	Moreover, $\|A\|_\infty = \Delta$, the number of variables is $2n L_\Delta $, and the number of constraints is $2nL_\Delta - n + 1$.
\end{proof}
\begin{corollary}[$\tw_P$ hardness] \label{cor:twphardnness}
	\eqref{IP} is \NPh already when $\tw_P(A) = 2$ and $\|A\|_\infty = 2$.
\end{corollary}
\begin{proof}
Again, let $\Delta=2$ and apply Lemma~\ref{lem:primaldtwdecomp}.
\end{proof}

\begin{remark}
The system given by constraints~\eqref{eq:tw:yij},~\eqref{eq:tw:zij}, and~\eqref{eq:tw:primalagg} also has constant degree and dual and incidence treewidth, but we will not use this fact.
\end{remark}

\medskip

Let us turn our attention to treedepth.
Say that an instance $(a_1, \dots, a_n, b)$ of \textsc{Subset Sum} is \emph{balanced} if the encoding length of $b$ is roughly $n$, i.e., if $n \in \Theta(\log_2 b)$.
We will use the following ETH-based lower bound for \textsc{Subset Sum}:
\begin{proposition}[\cite{KnopPW:2018}]
	\label{prop:sssum}
	Unless ETH fails, there is no algorithm for \textsc{Subset Sum} which would solve every balanced instance in time $2^{o(n+\log b)}$.
\end{proposition}
We remark that this proposition is obtained via the standard \NP-hardness reduction from \textsc{3-Sat} to \textsc{Subset Sum}, which starts from a \textsc{3-Sat} formula with $n$ variables and $m$ clauses and produces a \textsc{Subset Sum} instance with $\tilde{n} = 2(n+m)$ and $3(n+m) \leq \log_2 b \leq 4(n+m)$~\cite[Theorem 34.15]{CLRS}, hence $\frac{3}{2}\tilde{n} \leq \log_2 b \leq 2\tilde{n}$.
This is the reason why the lower bound holds for balanced instances.

In the next definition, we want to define a tree which is in some sense maximal among all trees with the same level heights and an additional constraint on the degrees of non-degenerate vertices.
\begin{definition}[$\vek$-maximal tree]
	Let $\vek = (k_1, \dots, k_\ell) \in \N^\ell$ and denote by $F_{\vek}$ the maximal (w.r.t. the number of vertices) rooted tree such that each root-leaf path $P$ of $F$ satisfies the following:
	\begin{enumerate}
		\item it contains $\ell$ non-degenerate vertices, i.e., $\ttd(F_\vek) = \ell$,
		\item $k_i(P) = k_i$ for each $i \in [\ell]$, thus $P$ has length $\|\vek\|_1$,
		\item the $i$-th non-degenerate vertex on $P$ has (in $F$) out-degree $k_i+1$, for each $i \in [\ell-1]$.
	\end{enumerate}
\end{definition}

\begin{lemma}\label{lem:fvekbound}
	Let $\vek \in \N^\ell$.
	Then $F_\vek$ has $K_\vek \df \left(\prod_{i=1}^\ell (k_i+1)\right)-1$ vertices.
\end{lemma}
\begin{proof}
	The proof goes by induction on $\ell$.
	In the base case when $\ell=1$, $F_{(k_1)}$ is a path on $k_1$ vertices and clearly $(k_1+1)-1 = k_1$.
	In the induction step, let $\vek' \df (k_2, \dots, k_\ell)$, so by the induction hypothesis $K_{\vek'} = \left(\prod_{i=2}^\ell (k_i+1)\right)-1$.
	Observe the structure of $F_{\vek}$: the segment between its root and its first non-degenerate vertex $v$ is a path on $k_1$ vertices, and the subtree of each child of $v$ is isomorphic to $F_{\vek'}$ and hence has $K_{\vek'}$ vertices.
	Thus, the number of vertices of $F_{\vek'}$ is
	\begin{multline*}
	K_{\vek} =\underbrace{k_1}_{\text{path}} + \underbrace{(k_1+1)K_{\vek'}}_{\text{subtrees}} = k_1 + (k_1+1)\left(\prod_{i=2}^\ell (k_i+1) - 1\right) = \\ =k_1 + \left(\prod_{i=1}^{\ell}(k_i+1)\right) - (k_1+1)	=\left(\prod_{i=1}^{\ell}(k_i+1)\right) - 1 \qedhere
	\end{multline*}
\end{proof}

\begin{lemma}\label{lem:vekpath}
	Let $\vek \in \N^\ell$, $K_{\vek}$ as in Lemma~\ref{lem:fvekbound}, $P_{K_{\vek}}$ be a path on $K_{\vek}$ vertices. Then
	$P_{K_{\vek}} \subseteq \cl(F_\vek)$.
\end{lemma}
\begin{proof}
	The proof goes by induction over $\ell$.
	In the base case $\ell=1$, $F_{(k_1)}$ is a path on $k_1$ vertices, so clearly $P_{k_1} \subseteq \cl\left(F_{(k_1)}\right)$.
	Assume that the claim holds for all $\ell' < \ell$ and let $\vek' \df (k_2, \dots, k_\ell)$.
	Note that in $P_{K_{\vek}}$ there exist $k_1$ vertices whose deletion partitions $P_{K_{\vek}}$ into $k_1+1$ paths on $K_{\vek'}$ vertices.
	Denote these vertices by $v_1, \dots, v_{k_1}$ and let $P'$ be the path $(v_1, v_2, \dots, v_{k_1})$.
	By the inductive hypothesis we have $P_{K_{\vek'}} \subseteq \cl(F_{\vek'})$.
	Take $k_1+1$ copies of $F_{\vek'}$ and connect each of its roots to $v_{k_1}$.
	Then $P_{K_{\vek}}$ is contained in the closure of this tree, and it is easy to see that this tree is isomorphic to $F_{\vek}$.
\end{proof}

\subsubsection{Lower Bound for Primal Treedepth}
\begin{theorem}[$\td_P$ lower bound]\label{thm:tdp_lowerbound}
	Let $\mathfrak{C}_P(\ell, \Delta)$ be the class of all~\eqref{IP} instances with $\|A\|_\infty \leq \Delta$ and such that the primal graph $G_P(A)$ admits a $\td$-decomposition with 
$\ttd(F) \leq \ell$,
and let $d \df \height(F)$.
	Unless ETH fails, there is no algorithm that solves every instance in $\mathfrak{C}_P(\ell, \Delta)$ in time
	\[
2^{o \left( 2 \sqrt{\log_2 \Delta}  \left( \tfrac{d}{2 \ell}\right)^{\ell/2}     \right)}
\]
\end{theorem}
When $d > 2\ell$, this is a double-exponential lower bound in terms of the topological height.
\begin{proof}[Proof of Theorem~\ref{thm:tdp_lowerbound}] 
	Take a balanced \textsc{Subset Sum} instance with $n$ items, i.e., $n \leq \log_2 b \leq 2n$, obtained by the reduction from a \textsc{3-Sat} instance.
	Taking the logarithm to the basis of $\Delta$, we obtain
	\[
		\frac{n}{\log_2 \Delta} \leq \log_\Delta b \leq \frac{2n}{\log_2 \Delta}.
	\]
	Encoding the instance with constraints~\eqref{eq:tw:primalagg},~\eqref{eq:tw:yij},~\eqref{eq:tw:zij}, we obtain $n (L_\Delta + 1) \leq \frac{2}{\log_2 \Delta} n^2 - 1$ variables $z_i^j$ and the same number of variables $y_i^j$.
	To obtain a $\td$-decomposition of $G_P(A)$, we will proceed as follows.
	We will first obtain a $\td$-decomposition for the $z_i^j$ variables only.
	Then, we will insert the $y_i^j$ variables in such a fashion that the topological height remains the same, and the treedepth increases by at most a factor of $2$.	
	
	To this end, let $k = \left\lfloor \sqrt[\ell]{\frac{2 n^2 }{\log_2 \Delta}} \right\rfloor$, and $r \in [1,\ell]$ such that
	\begin{align*}
	(k+1)^{r-1} k^{\ell - r + 1} \leq \frac{2 n^2 }{\log_2 \Delta} &< (k+1)^{r} k^{\ell - r} \\
		&= (k+1)^{r-1} k^{\ell-r + 1} + (k+1)^{r-1} k^{\ell-r} \\
		&\leq \frac{2n^2 }{\log_2 \Delta} + \frac{2n^2 }{\log_2 \Delta} \enspace .
	\end{align*}
	Set $\vek \df \{k\}^{r} \times \{k-1\}^{\ell -r}$, and observe that the constraints~\eqref{eq:tw:zij} form a path on the $z_i^j$ variables and the other constraints do not affect this path.
	By adding dummy variables (and thus at most doubling the number of variables),
	we may assume we have precisely $K_\vek$ variables $z_{i}^j$.
	Thus, by Lemma~\ref{lem:vekpath}, $F_\vek$ is a $\td$-decomposition of the subgraph of $G_P(A)$ induced by the $z_i^j$ variables.
	Now, for every $i \in [n]$, $j \in [0,L_\Delta]$, replace the vertex $z_{i}^j$ in $F_{\vek}$ with an edge $\{z_i^j , y_i^j\}$ in such a fashion that $z_i^j$ remains connected to its (possible) children, and the (possible) parent of $z_i^j$ gets connected to $y_i^j$.
	After doing this for every $z_{i}^j$ one by one, we obtain a tree $F$ on all vertices of $G_P(A)$.
	It remains to show that $F$ is a $\td$-decomposition of $G_P(A)$.
	Let us check the requirements of the definition of a $\td$-decomposition.
\begin{enumerate}
	\item For all $i \in [n]$, $j \in [L_\Delta]$, the edge $\{y_i^j, z_i^{j}\}$ is in $\cl( F )$ by construction.
	\item For $i \in [n]$, $j \in [L_\Delta]$, observe that the edge $e = \{z_i^j, z_i^{j-1}\}$ is in $\cl( F )$.
	Let w.l.o.g.\ $z_{i}^j$ be above (i.e., closer to the root than) $z_{i}^{j-1}$.
	As $y_i^{j-1}$ is a child of $z_{i}^{j-1}$, this implies that there is a path from the root to $y_i^{j-1}$ containing the vertices $z_{i}^{j}$, $y_{i}^{j}$, $z_{i}^{j-1}$.
	Thus all edges between any two of these vertices are contained in $\cl(F)$, in particular the edges $\{y_{i}^{j}, y_{i}^{j-1}\}$ and $\{y_{i}^{j}, z_{i}^{j-1}\}$.
	\item Hence, the only edges left are $\{z_{i-1}^{L_\Delta}, y_i^0\}$.
	But as the edge $\{z_{i-1}^{L_\Delta}, z_i^0\}$ is contained in $F$, the same arguments as before hold.
	\end{enumerate}
	The resulting tree $F$ has the same topological height as $F_\vek$, as we only introduced degenerate vertices.
	Moreover, the level heights do not decrease, and increase at most by a factor of $2$, as each vertex $z_i^j$ on a root-leaf path gets replaced by an edge (but not the dummy vertices).
	Hence, we modeled the instance of \textsc{Subset Sum} as an instance of~\eqref{IP} with a $\td$-decomposition $F$ of $G_P(A)$ with $\ttd(F)=\ell$ and
	\begin{align*}
		\height(F) = d &\leq 2 \ell k \leq 2 \ell \sqrt[\ell]{\tfrac{2 n^2 }{\log_2 \Delta}} \\
		\Rightarrow \qquad n^2 &\geq \left( \tfrac{d}{2\ell} \right)^\ell \log_2(\Delta).
	\end{align*}
	As we cannot solve the initial instance faster than $2^{o \left( n + \log_2 b \right)}$, the lower bound follows:
	\[
		2^{o \left( n + \log_2 b \right)} \geq 2^{o \left( 2n \right)}
			\geq 2^{o \left( 2 \sqrt{\log_2 \Delta}  \left( \tfrac{d}{2 \ell}\right)^{\ell/2}     \right)}		 \enspace . \qedhere
	\]
\end{proof}

\subsubsection{Lower Bound for Dual Treedepth}
In order to obtain the doubly-exponential lower bound for the dual case, we consider a multidimensional generalization of \textsc{Subset sum}.

\prob{\textsc{Multidimensional subset sum}}
{Integral vectors $\vea_1, \dots, \vea_n, \veb \in \Z^k$, multiplicities $\veu \in (\Z_{> 0}\cup \{\infty\})^n $.}
{Is there $\vex \in [\vezero,\veu]$ such that $\sum_{i =1}^n \vea_i x_i = \veb$?}
For short, we say that an instance of \textsc{Multidimensional subset sum} of dimension $k$ is a \textsc{$k$-dimensional subset sum} instance.
This problem generalizes \textsc{Subset sum} in two ways.
First, the vectors $\vea_i$ may also contain non-positive entries.
Moreover, every item $\vea_i$ has a multiplicity $u_i$, hence we can choose it up to $u_i$ times.
The following lemma easily follows by choosing $k=1$.

\begin{lemma}
The \textsc{Multidimensional subset sum} problem is \NPh even in fixed dimension.
\end{lemma}

\begin{lemma}
\label{lem:sssum-md-ssum}
Let $a_1,\dots,a_n,b$ be an instance of the \textsc{Subset sum} problem, $M \df \max (\{a_i \mid i=1,\dots,n\} \cup \{b\})$, and $\Delta \in \Z_{\geq 2}$.
There is an equivalent instance\footnote{i.e., there is a bijection between the sets of solutions} of \textsc{Multidimensional subset sum} with dimension $\lceil \log_\Delta (M + 1) \rceil$, $n + \lceil \log_\Delta (M + 1) \rceil -1$ items,
and all numbers bounded by $\Delta$ in absolute value.
\end{lemma}
\begin{proof}
Let $r \df \lceil \log_\Delta (M + 1) \rceil$.
Recall that $[a_i]_{\Delta} = (\alpha_i^0, \dots, \alpha_i^{r-1})$ is the base-$\Delta$ encoding of $a_i$, and similarly for $[b]_\Delta = (\beta^0, \dots, \beta^{r-1})$, and note that the upper indices do \emph{not} denote exponentiation.
Consider the following system of equations.
\begin{align*}
\begin{pmatrix}
\alpha_{1}^0 & \alpha_2^0 & \dots & \alpha_n^0 & - \Delta \\
\alpha_1^1 & \alpha_2^1 & \dots & \alpha_n^1 & 1 & \ddots \\
\vdots & & & \vdots & & \ddots & - \Delta \\
\alpha_1^{r-1} & \alpha_2^{r-1} & \dots & a_n^{r-1} & & & 1
\end{pmatrix}
\begin{pmatrix}
\vex \\ \ves
\end{pmatrix}
&=
\begin{pmatrix}
\beta^0 \\
\beta^1 \\
\vdots \\
\beta^{k-1}
\end{pmatrix}, \qquad 
\begin{matrix}
\vezero \leq \vex \leq \veone \\
\vezero \leq \ves
\end{matrix}.
\end{align*}
If we multiply the vector $(1,\Delta, \Delta^{2}, \dots, \Delta^{r-1})$ from the left on both sides,
we retrieve the original \textsc{Subset sum} instance where only the $\vex$ variables occur.
Hence, the projection $\pi: \, (\vex,\ves) \mapsto \vex$ maps solutions to solutions.
To show that every solution $\vex$ for the \textsc{Subset sum} instance has a pre-image $(\vex,\ves)$ that is a solution to the \textsc{Multidimensional subset sum} instance, reformulate the initial equation as
\[
\begin{array}{llllllll}
& 
\left( \sum_{j=0}^{r-1} \alpha_1^j \Delta^j \right) x_1 
&+&
\left( \sum_{j=0}^{r-1} \alpha_2^j \Delta^j \right) x_2 
&+& \dots &+&
\left( \sum_{j=0}^{r-1} \alpha_n^j \Delta^j \right) x_n
\\
= & \left( \sum_{i=1}^n \alpha_i^0 x_i \right)
&+&
\left( \sum_{i=1}^n \alpha_i^1 x_i \right) \Delta
&+& \dots &+&
\left( \sum_{i=1}^n \alpha_i^{r-1} x_i \right) \Delta^{r-1} \\
= & \left( \sum_{i=1}^n \alpha_i^0 x_i - \Delta s_1 \right)
&+&
\left( s_1 + \sum_{i=1}^n \alpha_i^1 x_i - \Delta s_2 \right) \Delta
&+& \dots &+&
\left( s_{r-1} + \sum_{i=1}^n \alpha_i^{r-1} x_i \right) \Delta^{r-1} \\
= & \hspace{10pt} \beta^0 \hspace{10pt} + \hspace{10pt} \beta^1 \Delta 
&+&
\hspace{10pt} \beta^2 \Delta^2
&+& \dots &+&
\hspace{10pt} \beta^{r-1} \Delta^{r-1}.
\end{array}
\]
Due to integrality, the difference $\sum_{i=1}^n (\alpha_i^0 x_i) -  \beta^0$ has to be a multiple of $\Delta$, and we can choose $s_1 \in \Z$ such that $\sum_{i=1}^n (\alpha_i^0 x_i) - \Delta s_1 =  \beta^0$.
Applying the argument iteratively, every difference $(s_{j-1} + \sum_{i=1}^n \alpha_i^j x_i) -  \beta^j$ has to be a multiple of $\Delta^{j+1}$. Defining $s_j$ iteratively by $s_{j} \Delta = (s_{j-1} + \sum_{i=1}^n \alpha_i^{j-1} x_i ) - \beta^{j-1}$ thus yields an integral vector $(\vex,\ves)$ fulfilling the system of equations.
To see that $\ves \geq \vezero$, first note that all entries $\alpha_i^j \geq 0$.
Furthermore, we have $\beta^j < \Delta^j$.
This immediately implies that $s_1 \geq 0$, and applying this argument iteratively again, shows $s_i \geq 0$ for all $i \in [n]$.

Hence, the  projection $\pi$ restricted to solutions is surjective.
As the last $r-1$ columns are linearly independent, we also know that our choice of $\ves$ is unique and $\pi$ is indeed a bijection.
\end{proof}

\begin{theorem}
\label{thm:k-dim-sssum-lower}
Assuming the ETH, there is no algorithm solving every \textsc{$k$-dimensional subset sum} instance with absolute values of the entries bounded by $\Delta$ in time $\Delta^{o(k)}$.
\end{theorem}
\begin{proof}
Choose a balanced \textsc{Subset sum} instance, and for $\Delta \in \Z_{\geq 2}$, let $k \in \Z$ be the unique integer such that $\Delta^{k-1} \leq b + 1 < \Delta^{k}$.
By Lemma~\ref{lem:sssum-md-ssum}, there exists an equivalent \textsc{$k$-dimensional subset sum} instance with entries bounded by $\Delta$.
As we cannot solve the initial instance faster than $2^{o (n + \log_2 (b))}$, where $3/2 n \leq \log_2 (b) \leq 2n$, we cannot solve the equivalent instance faster than
\begin{align*}
2^{o(n + \log_2 b)} &\geq 2^{o\left( n + (k-1)\log_2(\Delta) \right)} \\
	&\geq 2^{o\left( 1/2 (k-1)\log_2(\Delta)  + (k-1)\log_2(\Delta)\right)} \\
	&\geq \Delta^{o\left( 3/2 (k-1) \right) } \enspace . \qedhere
\end{align*}
\end{proof}

To eventually obtain the double-exponential lower bound, we will encode a \textsc{Subset sum} instance first as a \textsc{Multidimensional subset sum} instance, and then as an~\eqref{IP}.
To this end, we will first discuss how the decomposition of Lemma~\ref{lem:dualdtwdecomp} adapts for the multi-dimensional case.
Take a \textsc{$k$-dimensional subset sum} instance, given by a matrix $\bar{A} \in \Z^{k \times n}$ and $\bar{\veb} \in \Z^k$,
i.e., we are to solve the system $\bar{A}\vex = \bar{\veb}$ with $\vex \in [\vezero,\veu]^n$.
Choosing some $\Delta$, we encode this instance analogously to the encoding of \textsc{Subset sum} by constraints~\eqref{eq:tw:iy0},\eqref{eq:tw:yij}, and~\eqref{eq:tw:dualagg}, but instead of the single constraint~\eqref{eq:tw:dualagg}, we use $k$ constraints, one for each dimension.
Let $M = \max \{\Vert \bar{A} \Vert_\infty , \Vert \bar{\veb} \Vert_\infty \}$, and $L_\Delta = \log_\Delta (M + 1)$.
For the $\ell$-th entry of the $i$-th item,
denote the base-$\Delta$ encoding by $[a_{\ell,i}]_\Delta = (\alpha_{\ell,i}^{0},\dots,\alpha_{\ell,i}^{L_\Delta-1})$, and similarly for $[b_\ell]_{\Delta} = (\beta_\ell^0, \dots, \beta_\ell^{L_\Delta-1})$.

For the variables $x_i$, we introduce again the variables $y_i^j$ and constraints~\eqref{eq:tw:yij} for all $i \in [n]$, $j \in [L_\Delta]$.
Additionally, we add constraints
\begin{align}
\sum_{i=1}^n \sum_{j=0}^{L_\Delta-1} \alpha_{\ell,i}^j y_i^j &= b_\ell &  \forall \ell \in [k].
\tag{$S_\ell$} \label{eq:tw:dualagg-indexed}
\end{align}
Denote by $A$ the matrix of constraints \eqref{eq:tw:iy0},\eqref{eq:tw:yij}, and~\eqref{eq:tw:dualagg-indexed}.
For each $i \in [n]$, the constraints~\eqref{eq:tw:yij} induce a path $P_i$ in $G_D(A)$. 
For different $i$'s, these paths are not connected by any edge, hence we can use Lemma~\ref{lem:vekpath} to obtain a $\td$-decomposition of each path $P_i$ independently.
If all entries $\alpha_{\ell,i}^j$ are non-zero, i.e., the maximum number of edges in the $G_D(A)$ is attained, the constraints~\eqref{eq:tw:dualagg-indexed} induce a clique in $G_D(A)$, and there are edges between~\eqref{eq:tw:dualagg-indexed} and~\eqref{eq:tw:yij} for all $\ell \in [k]$, $i \in [n]$, and $j \in [0,L_{\Delta}-1]$.
Thus, we will take a path on the variables~\eqref{eq:tw:dualagg-indexed}, and attach one $\td$-decomposition of each path $P_i$ to it.

The whole proof for the lower bound will therefore be divided in two steps.
Starting with a balanced instance of \textsc{Subset Sum}, we first encode by an instance of \textsc{Multidimensional subset sum}.
In a second step, we encode this instance by an instance of~\eqref{IP}.
Though it seems convenient to start already with a \textsc{Multidimensional subset sum} instance as in Theorem~\ref{thm:k-dim-sssum-lower}, observe that the dimension $k$ depends on the size of the \textsc{Subset sum} instance it is derived from, as well as with the choice of an upper bound for $\Vert A \Vert_\infty$ of the~\eqref{IP} instance we want to obtain.
Hence, we find it cleaner to not use this intermediate step.

\begin{theorem}[$\td_D$ lower bound]\label{thm:tdd_lowerbound}
	Let $\Delta \geq 2$, $\ell \geq 1$, and $\mathfrak{C}_D(\ell, \Delta)$ be the class of all~\eqref{IP} instances with $\|A\|_\infty \leq \Delta$,
	such that the dual graph $G_D(A)$ admits a $\td$-decomposition with $\ttd(F)=\ell$.
	Let $d \df \height(F)$.
	Assuming ETH there is no algorithm that solves every instance in $\mathfrak{C}_D(\ell, \Delta)$ in time
	\[
	\Delta^{o \left( \left(\frac{d}{\ell} + \tfrac{1}{12} \right)^\ell \right)} \enspace .
	\]
	Specifically, assuming ETH, no algorithm solves every generalized $n$-fold IP in time
	\[
		\|A\|_\infty^{o \left( (r+s)^2 \right)} \enspace .
	\]
\end{theorem}
\begin{proof}
	Fix $\Delta \in \Z_{\geq 2}$ and $\ell \geq 2$.
	Let $n_0$ be large enough (implying that the choice of $k$ will be large enough), and take a balanced subset sum instance with $n \geq n_0$ items.
	Define 
	\[
	k \df \left\lfloor \sqrt[\ell]{ \log_\Delta (b) } \right\rfloor,
	\]
	and let $1 \leq r \leq \ell$ be the unique integer such that
	\[
	(k+1)^{r-1} k^{\ell-r + 1} \leq \log_\Delta (b) < (k+1)^r k^{\ell - r}.
	\]
	We will assume $r \leq \ell-1$, and discuss the case $\ell = r$ in the end.
	Setting 
	\[
	\widetilde{\Delta} \df \Delta^{(k+1)^{r-1} k^{\ell - r}},
	\]
	this implies that $b < \widetilde{\Delta}^{k+1}$, and due to integrality, all occurring numbers are at most $\widetilde{\Delta}^{k+1} - 1$. 
	By Lemma~\ref{lem:sssum-md-ssum}, we can encode this instance as a $(k+1)$-dimensional subset sum instance with $n^\prime = n + k$ items and with the largest absolute value of any number bounded by $\widetilde{\Delta}$.
	
	We further encode this $(k+1)$-dimensional subset sum instance as an~\eqref{IP} with constraints~\eqref{eq:tw:yij} and~\eqref{eq:tw:dualagg-indexed}.
	Let $A$ denote the constraint matrix we obtain this way, and
\[
L_\Delta \df \ceil{ \log_\Delta \left( \widetilde{\Delta} + 1 \right)} \leq \Delta^{(k+1)^{r-1} k^{\ell - r}} + 1.
\]
	For each $i$, the constraints~\eqref{eq:tw:yij} induce a path $\Gamma_i$ on $L_\Delta$ vertices as a subgraph in $G_D(A)$.
	We obtain $n^\prime$ paths in total, and there are no edges between distinct paths in $G_D(A)$.
	It remains to construct a $\td$-decomposition containing these paths and the Constraints~\eqref{eq:tw:dualagg-indexed}.
	
	Set $\vek \df \{k\}^{r} \times \{k - 1\}^{\ell - r - 1}$, and note that $F_\vek$ has
	$(k+1)^{r} k^{\ell -r -1} - 1 \geq L_\Delta$ vertices, for $k$ large enough.
	By Lemma~\ref{lem:vekpath} $F_{\vek}$ is a $\td$-decomposition for each subgraph $\Gamma_i$ (possibly after adding dummy constraints).
	Take one copy $F_i$ of $F_\vek$ for each $i$.

	For the subgraph $\Sigma$ induced by the Constraints~\eqref{eq:tw:dualagg-indexed}, take a path on $k+1$ vertices to be a $\td$-decomposition.
	Declare one of its endpoints to be the root and append the copies $F_i$ to the other end.
	As all edges of $G_D(A)$ are either within the subgraph $\Sigma$, a subgraph $\Gamma_i$, or between $\Sigma$ and one of the graphs $G_i$, the tree $F$ we obtain this way is a $\td$-decomposition for $G_D(A)$.
	We have $\ttd(F) = \ell$, its first level height is $k+1$, and the other level heights either $k$ or $k-1$.
	In particular, its height is 
\begin{align}
\label{eq:height-in-eth-lb-dual}
	d &\df (k+1) + rk + (\ell-r-1)(k-1) = \ell k - (\ell - r - 2) \leq \ell k + 1,
\end{align}
	where we used the fact that $r \leq \ell - 1$.
	
	By Proposition~\ref{prop:sssum}, the chosen subset sum instance cannot be solved faster than $2^{ o (n + \log_2 (b))}$, and due to the choice of parameters, we can estimate the exponent to
	\begin{align}
		n + \log_2 (b) \geq \frac32 \log_2 (b) 
			&\geq \frac32 \log_2 (\Delta) \log_\Delta (b) \nonumber \\
			&\geq \frac32 \log_2 (\Delta) (k+1)^{r-1} k^{\ell-r + 1}. \label{eq:aux9}
	\end{align}

	We now distinguish two cases.
	First, assume $r \geq \tfrac{2}{3} \ell - 2$, and let $\tfrac{1}{4} > \varepsilon > 0$.
	Choosing $n_0$ large enough so that $k$ is large enough, we obtain the two estimates
	\begin{align}
	( \tfrac{k}{k+1} )^3 &\geq \tfrac{4}{5}, \label{eq:aux12} \allowdisplaybreaks \\
	(k+1)^2 k &= k^3 + 2k^2 + k \geq k^3 + (2- \varepsilon) k^2 + \tfrac{4}{3} k + \tfrac{8}{27} \nonumber \\
	&\geq (k + \tfrac{2-\varepsilon}{3})^3, \label{eq:aux11} 
	\end{align}
	and can compute
	\begin{align*}
		(k+1)^{r-1} k^{\ell-r + 1} &= \left(\frac{k}{k+1}\right)^3 (k+1)^{r+2} k^{\ell - r - 2} \allowdisplaybreaks \\
			&\geq \frac45 \left( (k+1)^2k \right)^{\ell/3} & (with~\eqref{eq:aux12}), \allowdisplaybreaks \\
			&\geq \frac45 \left( k + \tfrac{2-\varepsilon}{3} \right)^\ell & (with~\eqref{eq:aux11}), \allowdisplaybreaks \\
			&\geq \frac45 \left( \tfrac{d}{\ell} - \tfrac{1}{\ell} + \tfrac{2-\varepsilon}{3} \right)^\ell & (with~\eqref{eq:height-in-eth-lb-dual}), \allowdisplaybreaks \\
			&> \frac45 \left( \frac{d}{\ell} + \tfrac{1}{12} \right)^\ell. & (with \ \ell \geq 2, \varepsilon < 1/4).
	\end{align*}

	The second case, $r < \ell/3 - 2$, can only occur if $\ell > 9$, as $r \geq 1$ by definition.
	Then
	\begin{align*}
	\allowdisplaybreaks
		(k+1)^{r-1} k^{\ell-r + 1} &\geq \left( \frac{\ell k }{\ell}\right)^\ell \\
		&= \left( \frac{ \ell k - (\ell - r - 2)}{\ell} + \frac{(\ell-r-2)}{\ell} \right)^\ell \\
		&> \left( \frac{ d }{\ell} + \frac{1}{3}\right)^\ell,
	\end{align*}
	where we used the definition of $d$ for the first term, and the estimate on $r$ and $\ell$ for the second.
	In the end, we obtain that there is no algorithm solving every IP instance~\eqref{IP} that has a $\td_D$-decomposition with height $d$ and topological height at most $\ell$ in time
	\begin{align*}
	2^{o (n + \log_2 b)} &= 2^{o \left( \log_2 (\Delta) (k+1)^{r-1} k^{\ell-r + 1} \right)} \\
	&= \Delta^{o \left( \left(\frac{d}{\ell} + \tfrac{1}{12} \right)^\ell \right)}.
	\end{align*}
	A generalized $n$-fold IP with coefficients $r,s$ has a $\td_D$-decomposition with topological height $2$ and level heights $r,s$.
	Thus, the class of generalized $n$-fold IPs cannot be solved faster than $\Vert A \Vert_\infty^{o ((r+s)^2)}$.

It is left to discuss the case
\[
	(k+1)^{\ell - 1} k \leq \log_{\Delta} (b) < (k+1)^\ell.
\]
The idea is to multiply the initial subset sum instance by some integer, changing our initial choice from $k$ to $k+1$.
We have $\Delta^{(k+1)^{\ell - 1} k} \leq b$, and since
\begin{align*}
	(k+1)^{\ell - 1} k + k^{\ell -1} &= (k+1)^ \ell, \\
	(k+1)^\ell + k^{\ell-1} &\leq (k+1)((k+1) + 1)^{\ell-1} = (k+1)(k+2)^{\ell-1},
\end{align*}
we can scale $b^\prime \df b \cdot \Delta^{k^{\ell-1}}$, and the new instance implies $1 \leq r \leq \ell -1$. 
We multiplied each number in the instance by at most $b$, which is assumed to be the largest integer.
The rest of the proof holds for the modified instance, up to the Estimate~\eqref{eq:aux9}, where we use that the initial instance was balanced.
However, this changes to
\[
	n + \log_2 (b) \geq \tfrac{3}{2} \log_2 (b) \geq \tfrac{3}{4} \log_2 (b^\prime) \geq \tfrac{3}{4} \log_2(\Delta)(k+1)^{r-1}k^{\ell-r+1},
\]
and as the constant $3/4$ vanishes in the Landau notation, the claim still holds.
\end{proof}
\begin{remark} 
	The parameter dependence of the lower bound of Theorem~\ref{thm:tdd_lowerbound} almost coincides with the norm bound $g_1(A)$ of Lemma~\ref{lem:dual_norm} by providing a $\td$-decomposition whose level heights are evenly split (i.e., each $k_i(F)$ is either $k$ or $k+1$).
	As compared with the recent lower bound of Knop, Pilipczuk, and Wrochna~\cite{KnopPW:2018}, our lower bound is slightly stronger (see below), and, more importantly, gives a tighter lower bound for each class of~\eqref{IP} with a fixed topological height $\ell \in \N$.
	Let us elaborate.
	
	Our algorithms have parameter dependence of $(2g_1(A))^{\Oh(\td_D(A))}$ which is $(\|A\|_\infty K)^{\Oh(\td_D(A) \cdot (K-1))}$, with $K \leq \prod_{i=0}^{\ttd(F)} k_i(F)$.
	Thus, in an asymptotic sense, the exponential dependence of our algorithm differs from our lower bound by a factor of $\td_D(A) \log K$.
	We find this remarkable, because closing this gap turns out to be a generalization of a major open problem: is there an $(\|A\|_\infty m)^{\Oh(m)} n$ algorithm for~\eqref{ILP}? The fastest known algorithm has complexity $(\|A\|_\infty m)^{\Oh(m^2)} n$~\cite{EisenbrandW:2018}, and there is no $(\|A\|_\infty m)^{o(m)} n$ algorithm assuming ETH~\cite{KnopPW:2018}.
	Taking $n$-fold IP as an analogue, there is an algorithm with parameter dependence $(\|A\|_\infty rs)^{\Oh(r^2s + s^2)}$~\cite{JansenLR:2018} and Theorem~\ref{thm:tdd_lowerbound} implies that under ETH no algorithm achieves $(\|A\|_\infty)^{o(r^2 + rs + s^2)}$ (improving this to $(\|A\|_{\infty} rs)^{o(rs)}$ would also be interesting).
	
	(Here and in the following we identify $n$-fold IP with~\eqref{IP} instances with $\ttd(F)=2$, and tree-fold IP with $\tau$ levels with~\eqref{IP} instances with $\ttd(F)=\tau$. This follows from the fact that any~\eqref{IP} where $G_D(A)$ has a $\td$-decomposition with $\ttd(F)=2$ can be embedded into an $n$-fold IP with similar parameters, and analogously for tree-fold, see~\cite{KouteckyLO:2018}.)
	
	The question similarly generalizes to tree-fold IP. A particularly nice form of this question is this: tree-fold IP with $\tau$ levels and with each block having $k \geq 2$ rows has an algorithm with parameter dependence $(\|A\|_\infty k)^{\Oh(k \cdot \tau^2 \cdot k^\tau)}$, and Theorem~\ref{thm:tdd_lowerbound} says that no algorithm can reduce this to $(\|A\|_\infty)^{o(k^\tau)}$ unless ETH fails.
	What is the ``true'' complexity of such tree-fold IP instances?
	
	Regarding the exact relationship of our lower bound to the recent result of Knop, Pilipczuk and Wrochna~\cite{KnopPW:2018}, we note that for instances with topological height $\ell = d$, our lower bound is asymptotically the same, but slightly stronger, in the following way.
	We obtain, in the exponent, a dependence of $o((3/2)^d) = o(2^{cd})$ for some constant $c < 1$, whereas they have $2^{o(d)}$.
	An example of a running time that is ruled out by our bound, but not theirs, is $(1/d) 2^{cd} = 2^{cd - \log_2 d}$:
	\begin{align*}
	\lim_{d \rightarrow \infty} \left| \frac{1/d 2^{cd}}{2^{cd}} \right| &= \lim_{d \rightarrow \infty} 1/d = 0, \\
	\lim_{d \rightarrow \infty} \left| \frac{cd - \log_2 d}{d} \right| = c \neq 0 \enspace .
	\end{align*}
\end{remark}

\medskip

\subsubsection{Lower Bound on Graver Elements}
The two deciding factors of the efficiency of our algorithms are small norms of Graver basis elements (Lemmas~\ref{lem:primal_norm} and~\ref{lem:dual_norm}), and small treewidth (see Lemmas~\ref{lem:primal_treewidth} and~\ref{lem:dual_treewidth}) of the constraint matrix.
The hardness of Corollaries~\ref{cor:twdhardness} and~\ref{cor:twphardnness} suggest that there exist matrices with constant treewidth yet large $g_\infty(A)$.
The next lemma exhibits such a matrix, and one may observe that it is precisely the matrix of constraints~\eqref{eq:tw:yij}.
\begin{lemma}[Constant $\tw$ but exponential $g_\infty(A)$] \label{lem:largeg1}
	For each $n \geq 2$, $n \in \N$, there exists a matrix $A \in \Z^{(n-1) \times n}$ with $\tw_P(A) = \tw_D(A) = 1$, $\|A\|_\infty = 2$, and $2^{n-1} \leq g_\infty(A) \leq g_1(A)$.
\end{lemma}
\begin{proof}
	Let
	\[
	A:=\left(
	\begin{array}{ccccccc}
	2      & -1  & 0 & \cdots & 0 & 0   \\
	0      & 2       & -1 & \cdots & 0 & 0\\
	0   & 0  & 2    &  \cdots & \vdots & \vdots\\
	\vdots   &  \vdots & \vdots     &  \ddots & -1 & 0\\
	0      & 0 & 0 &\cdots  &    2 & -1  \\
	\end{array}
	\right),
	\]
	be an $(n-1) \times n$ matrix.
	The sequence $\{1,2\}, \{2,3\}, \dots, \{n-1,n\}$ forms a tree decomposition of $G_P(A)$ of width $1$; analogously $\{1,2\}, \{2,3\}, \dots, \{n-2,n-1\}$ is a tree decomposition of $G_D(A)$ of width $1$.
	Observe that every $\vex \in \Z^n$ with $A \vex = \mathbf{0}$ must satisfy that $x_{i+1} = 2x_i$ for each $i \in [n-1]$.
	Clearly  $\veh = (1, 2, 4, \dots, 2^{n-1}) \in \N^{n}$ satisfies $A \veh = \mathbf{0}$, and from the previous observation it immediatelly follows that there is no $\veh' \in \N^n$ with $\veh' \sqsubset \veh$, and thus $\veh \in \G(A)$.
\end{proof}
\begin{acks}
Eisenbrand, Hunkenschröder, and Klein were supported by the Swiss National Science Foundation (SNSF) within the project \emph{Convexity, geometry of numbers, and the complexity of integer programming (Nr. 163071)}.
Levin and Onn are partially supported by Israel Science Foundation grant 308/18.
Onn was also partially supported by the Dresner Chair at the Technion.
Koutecký is partially supported by Charles University project UNCE/SCI/004, and by the project 17-09142S of GA ČR.
\end{acks}

\bibliographystyle{ACM-Reference-Format}
\bibliography{psp}

\end{document}